\documentclass[onefignum,onetabnum]{siamart220329}

\usepackage{lipsum}
\usepackage{amsfonts}
\usepackage{amssymb}
\usepackage{graphicx}
\usepackage{epstopdf}
\usepackage{algorithmic}
\usepackage{enumitem}
\usepackage{array}
\usepackage[caption=false]{subfig}
\usepackage{graphicx}

\usepackage{algorithmic}
\ifpdf
  \DeclareGraphicsExtensions{.eps,.pdf,.png,.jpg}
\else
  \DeclareGraphicsExtensions{.eps}
\fi

\newsiamremark{remark}{Remark}
\newsiamremark{hypothesis}{Hypothesis}
\crefname{hypothesis}{Hypothesis}{Hypotheses}
\newsiamthm{claim}{Claim}

\usepackage{tikz}
\usepackage{calc}

\tikzset{ 
    dim_node/.style={
        rectangle, 
        draw=blue!80, 
        fill=blue!5, 
        thick, 
        text centered, 
        rounded corners
    },
    tree_edge/.style={
        thick, ->
    },
    edge from parent/.style={
        draw, thick, ->
    } 
}

\tikzstyle{level 1}=[level distance=12.5mm,sibling distance=3cm] 
\tikzstyle{level 2}=[level distance=12.5mm,sibling distance=1.5cm]
\tikzstyle{level 3}=[level distance=12.5mm,sibling distance=0.75cm]

\newcommand{\etal}{\textit{et al.}~}

\headers{Adaptive Rank Integrators for Kinetic Transport Equations}{W.A. Sands, W. Guo, J.-M. Qiu, and T. Xiong}

\title{High-order Adaptive Rank Integrators for Multi-scale Linear Kinetic Transport Equations in the Hierarchical Tucker Format \thanks{Submitted to the editors DATE.
\funding{W.A. Sands and J.-M. Qiu wish to acknowledge support provided by Department of Energy DE-SC0023164, NSF NSF-DMS-2111253, and the Air Force Office of Scientific Research FA9550-22-1-0390. W. Guo was supported by the NSF NSF-DMS-2111383 and the Air Force Office of Scientific Research FA9550-18-1-0257. 
}}}

\author{William A. Sands\thanks{Department of Mathematical Sciences, University of Delaware, Newark, DE, 19716, United States 
  (\email{wsands@udel.edu}).}
\and Wei Guo\thanks{Department of Mathematics and Statistics, Texas Tech University, Lubbock, TX, 70409, United States (\email{weimath.guo@ttu.edu}).}
\and Jing-Mei Qiu\thanks{Department of Mathematical Sciences, University of Delaware, Newark, DE, 19716, United States 
  (\email{jingqiu@udel.edu}).}
\and Tao Xiong\thanks{School of Mathematical Sciences, University of Science and Technology of China, Hefei, Anhui 230026, People’s Republic of China 
  (\email{taoxiong@ustc.edu.cn}).}
}

\begin{document}

\maketitle

\begin{abstract}
    In this paper, we present a new adaptive rank approximation technique for computing solutions to the high-dimensional linear kinetic transport equation. The approach we propose is based on a macro-micro decomposition of the kinetic model in which the angular domain is discretized with a tensor product quadrature rule under the discrete ordinates method. To address the challenges associated with the curse of dimensionality, the proposed low-rank method is cast in the framework of the hierarchical Tucker decomposition. The adaptive rank integrators we propose are built upon high-order discretizations for both time and space. In particular, this work considers implicit-explicit discretizations for time and finite-difference weighted-essentially non-oscillatory discretizations for space. The high-order singular value decomposition is used to perform low-rank truncation of the high-dimensional time-dependent distribution function. The methods are applied to several benchmark problems, where we compare the solution quality and measure compression achieved by the adaptive rank methods against their corresponding full-grid methods. We also demonstrate the benefits of high-order discretizations in the proposed low-rank framework.
\end{abstract}



%
%
\section{Introduction}

Simulating linear kinetic transport equations is critical to understanding physical processes found in diverse disciplines such as nuclear engineering, astrophysics, and computational medicine. Kinetic models for transport capture the evolution of a high-dimensional probability density function $f \left(\mathbf{x}, \mathbf{\Omega}, t \right)$ which represents the probability of finding a particle, such as a neutron or photon, at position $\mathbf{x} \in \mathbb{R}^{d}$, moving in the direction $\mathbf{\Omega} \in \mathbb{S}^{d-1}$, at time $t$. The storage of the distribution function in the fully discrete setting is known to suffer from the curse of dimensionality (CoD). This issue is further compounded by the need to resolve multi-scale effects such as collisions with other particles and interactions with the background material. Depending on the relative importance of the terms, the equations can also change type in different physical regimes. In these circumstances, naive discretization techniques can become inefficient because of highly restrictive stability conditions or inaccurate due to physical inconsistencies associated with approximate models. These characteristics have motivated the development of asymptotic-preserving (AP) methods which capture the macroscopic behavior at the discrete level.

The technical report by Brunner \cite{Brunner2002} provides an excellent overview of several classes of numerical methods for transport problems, including Monte Carlo methods, discrete ordinates methods, and moment methods based on expansions in the spherical harmonic basis. Monte Carlo methods \cite{vassiliev2017mc,dimarco2018APMC,krotz2023hybrid}, which track the evolution of particles along characteristics, are conceptually simple and highly parallel. However, these methods require many simulation particles to combat statistical noise, and the solution quality degrades in more rarefied regions containing fewer particles. The discrete ordinates ($S_{N}$) \cite{MillerLewis93,azmy2010advances} method restricts the angular component of the solution to a set of directions on the unit sphere. In spite of its efficiency, its primary disadvantages include a loss of rotational invariance and numerical artifacts known as ray effects when the angular domain is under-resolved. On the other hand, spherical harmonics discretizations \cite{McClarrenTRTEspherical2008} preserve the rotational invariance, but encounter challenges in highly kinetic regions, as the quality of the spectral approximation degrades in the presence of angular discontinuities \cite{McClarrenTRTEspherical2008}. Filtered spherical harmonics methods \cite{mcclarren2010robust,frank2016convergence} can improve the quality of the solutions, but many challenges remain such as the enforcement of more general boundary conditions. Adaptive sparse grid representations based on tensor products of wavelets have also been considered as a way to address the high dimensionality in kinetic simulations \cite{GrellaSchwab,GuoChengSparseGridDG2016,GuoChengSparseGridDG2017}. Another interesting approach was proposed by Zhang \etal \cite{zhang2023asymptotic} in which the semi-Lagrangian method was used to remove the stiffness of convection terms in the intermediate regimes, resulting in a method with uniform unconditional stability. Peng \etal \cite{peng2024RBM} recently proposed a reduced basis method for the linear transport equation that constructs subspaces for the components of the macro-micro decomposition using an efficient iterative algorithm, which also respects the diffusion limit.

Adaptive rank integrators, such as the dynamical low-rank (DLR) and step-and-truncate (SAT) methods, have emerged as a popular tool for accelerating the solution of high-dimensional kinetic equations and reducing their memory footprint \cite{kormann2015TTvlasov,einkemmer2021mass,GuoVlasovDLRVlasovDynamics,GuoVlasovLoMacDG2023,Einkemmer2024alfven}. Application of these methods to problems of interest in the transport community is fairly recent. In \cite{pengDLR2020spherical-harmonics}, Peng \etal developed a low-rank scheme using a spherical harmonics angular discretization. These ideas were later used to develop high-order-low-order schemes that evolve the high fidelity model using a low-rank scheme that is used to define a closure to a fluid system \cite{PengDLR2021Holo}. Ding \etal \cite{DingDLRdiffusion2021} analyzed the behavior of DLR methods in the diffusive regime of the linear transport equation, and they showed that low-rank methods are capable of capturing such limits. Einkemmer \etal \cite{EinkemmerDLR-AP} proposed a DLR method for the linear transport equation which is AP and can be second-order accurate in both time and space. Their low-rank scheme is built on the macro-micro decomposition and staggered grid discretization originally proposed by Lemou and Mieussens \cite{Lemou-Mieussens2008}. A similar macro-micro decomposition strategy was later applied to the BGK equation in the collisional setting to enforce the limiting Navier-Stokes equations \cite{EinkemmerDLR-AP-BGK} and, more recently, the Lenard-Bernstein equation \cite{coughlin2024robust}. Peng and McClarren proposed a low-rank scheme which is first-order accurate in both time and space \cite{PengDLR2023discrete-ordinates} using the $S_{N}$ method. Their approach applies a low-rank discretization to each octant of the sphere which is evolved using a transport sweep with source iteration. In \cite{hu2022adaptiveBoltzmann}, Hu and Wang presented an adaptive low-rank scheme for the nonlinear Boltzmann equation and combined this with a Fourier spectral method to approximate the collision operator. An energy stable DLR scheme for transport problems was recently proposed by Einkemmer \etal \cite{Einkemmer_AP_DLR_energy2024} based on the spherical harmonics discretization of the angular domain. Using a certain time step restriction, they showed that the resulting scheme obeyed an energy dissipation principle.

Many of the aforementioned approaches use a projector splitting technique \cite{lubich2014projector,einkemmer2018VPsplitting} or the ``unconventional" basis update Galerkin (BUG) integrator \cite{ceruti2022unconventional} to update the low-rank basis independently. A limitation of these earlier approaches is that they require a fixed rank. Recent extensions of the BUG integrator have been proposed to achieve rank adaptivity \cite{ceruti2022rank-adaptive,HauckPCdynamic23}. In contrast, the adaptive rank SAT framework \cite{GuoVlasovFlowMap2022} avoids the projector splitting used by the DLR approach and can be combined with traditional high-order discretizations, such as the one proposed in \cite{JangAP-IMEX-DG-2015}. Additionally, the SAT framework provides considerable flexibility in the arrangement of the dimension trees used to represent tensors. In this paper, we represent high-dimensional functions using the hierarchical Tucker tensor (HTT) format \cite{HackbuschKuhnHT2009,GrasedyckHTT2010}, which has been proven to be effective in taming the challenges associated with the CoD in kinetic simulations. We achieve notable improvements in efficiency through the use of a high-order singular value decomposition (HOSVD) compression, which is used to truncate nodes of the dimension tree with small singular values. Such benefits have also been observed in similar work by Truong \etal \cite{truong2024tensor} who considered tensor network approaches in applications of the high-dimensional, time-independent neutron transport equations.

This work contributes several novel low-rank approximation techniques for solving the high-dimensional linear kinetic transport equation. Our approach utilizes a macro-micro decomposition of the distribution function in the kinetic model, which ensuring the scheme recovers the limiting linear diffusion equation under the appropriate scalings. The proposed methods integrate traditional high-order AP methods with novel tensor formats that support rank adaptivity. The discretization of space is performed using a high-order weighted essentially non-oscillatory (WENO) finite-difference method, while the temporal discretization is performed using high-order implicit-explicit (IMEX) Runge-Kutta (RK) methods with the globally stiffly accurate (GSA) property to handle the stiffness associated with the diffusion limit. The discretization of the angular domain is performed using the $S_{N}$ method. In this work, we consider the Chebyshev-Legendre (CL) quadrature nodes, which form a tensor product quadrature rule on the surface of the unit sphere. We also introduce a projection technique to enforce a condition on the zeroth angular moment of $g$ that is fundamental to the macro-micro decomposition. Additionally, we discuss the evaluation of Hadamard products, which are needed in transport applications to account for interactions with the background material, and highlight the computational challenges associated with these terms in the low-rank setting. Our numerical results demonstrate the capabilities of the proposed methods on two example problems from the transport literature. The first example considers a strongly varying scattering-cross section, which combines characteristics of both free-streaming and collisional regimes, and the second is a modification of the lattice problem, which contains discontinuities in material cross-sections. The low-rank structures present in these benchmark problems are carefully examined by studying the growth in the hierarchical rank and the decay of the singular values at the nodes of the tensor trees. We find that the proposed methods can significantly reduce the storage requirements.

The organization of the remainder of this paper begins with an overview of the problem formulation in \cref{sec:formulation}. In \cref{sec:discretization main}, we provide the relevant details of the discretization adopted by the proposed methods. We begin with the time discretization for the system and its AP property, before moving to the discretization of the angular and spatial domains. Then, we connect these elements to the HTT format and address some complications relevant to transport applications. We also establish an AP property, which accounts for the truncation error in the low-rank approximation. \Cref{sec:results} presents the numerical results which considers several benchmark problems from the literature. We conclude the paper with a brief summary in \cref{sec:conclusion main}.

\section{The Linear Transport Equation and Macro-micro Decomposition}
\label{sec:formulation}

We consider the time-dependent linear transport equation under a diffusive scaling:
\begin{equation}
    \partial_{t} f + \frac{1}{\epsilon} \mathbf{\Omega} \cdot \nabla_{x} f = \frac{\sigma_{s}}{\epsilon^{2}} \Big( \langle f \rangle_{\mathbf{\Omega}} - f \Big) - \sigma_{a} f + Q, \quad (\mathbf{x}, \mathbf{\Omega}) \in D \times \mathbb{S}^{2}, \label{eq:linear kinetic equation}
\end{equation}
where $D \subset \mathbb{R}^2$ in this work, and $\mathbb{S}^{2}$ is the surface of the unit sphere. Additionally, we let $\sigma_{a} \left(\mathbf{x}\right) \geq 0$ and $\sigma_{s} \left(\mathbf{x}\right) > 0$ denote the absorption and scattering cross-sections for the background material and $Q(\mathbf{x}, t)$ is the radiant source. Additionally, the Knudsen number $\epsilon$ characterizes the strength of the collisions in the problem. In this work, we only consider isotropic collisions, so the relevant collision operator consists of integration over the unit sphere, specifically
\begin{equation*}
    \langle f \rangle_{\mathbf{\Omega}} := \frac{1}{4\pi} \int_{\mathbb{S}^{2}} f \left( \mathbf{x}, \mathbf{\Omega}', t \right)  \, d\mathbf{\Omega}'.
\end{equation*}
Although it is more conventional to use inflow-outflow type boundary conditions, we shall restrict ourselves to periodic domains and leave the complication of boundary conditions to future work.

One of the challenges associated with the kinetic model \eqref{eq:linear kinetic equation} is the capturing the range of behavior observed as $\epsilon$ changes. When $\epsilon = \mathcal{O}(1)$, the kinetic equation behaves as a hyperbolic equation, but transitions to a parabolic or diffusive regime when $\epsilon \ll 1$. Lemou and Mieussens \cite{Lemou-Mieussens2008} proposed the macro-micro decomposition, which separates $f$ into an equilibrium and non-equilibrium components, namely
\begin{equation}
    f \left(\mathbf{x}, \mathbf{\Omega},t \right) = \rho \left(\mathbf{x},t \right) + \epsilon g \left(\mathbf{x}, \mathbf{\Omega},t\right), \label{eq:decomposition}
\end{equation}
where $\rho = \langle f \rangle_{\mathbf{\Omega}}$, and we have the condition $\langle g \rangle_{\mathbf{\Omega}} = 0$. Using the decomposition \eqref{eq:decomposition}, it can be shown that the linear kinetic equation \eqref{eq:linear kinetic equation} is equivalent to the system
\begin{align}
    &\partial_{t} \rho + \nabla_{x} \cdot \left( \left\langle \mathbf{\Omega} g \right\rangle_{\mathbf{\Omega}} \right) = - \sigma_{a} \rho + Q, \label{eq:macro} \\
    &\partial_{t} g + \frac{1}{\epsilon} \Big( I - \Pi \Big) \Big( \nabla_{x} \cdot \left(  \mathbf{\Omega} g \right) \Big) + \frac{1}{\epsilon^{2}}\nabla_{x} \cdot \left( \mathbf{\Omega} \rho \right) = -\frac{\sigma_{s}}{\epsilon^{2}} g - \sigma_{a} g, \label{eq:micro}
\end{align}
where $I$ is the identity operator, and we use the notation
\begin{equation*}
    \Big( I - \Pi \Big) \Big( \nabla_{x} \cdot \left(  \mathbf{\Omega} g \right) \Big) := \nabla_{x} \cdot \left(  \mathbf{\Omega} g \right) - \nabla_{x} \cdot \left( \left\langle  \mathbf{\Omega} g \right\rangle_{\mathbf{\Omega}} \right).
\end{equation*}

The macro-micro system \eqref{eq:macro}-\eqref{eq:micro} preserves the asymptotic limit of the kinetic equation \eqref{eq:linear kinetic equation} obtained from a Chapman-Enskog expansion \cite{Lemou-Mieussens2008} when the initial data is \textit{well-prepared}, which can be defined as follows.
\begin{definition}
Let $\epsilon > 0$ denote a small parameter. The initial data $f_{0}^{\epsilon}$ are said to be well-prepared if they admit an asymptotic expansion of the form
\begin{equation*}
    f_{0}^{\epsilon} = f_{0} + \epsilon f_{1} + \epsilon^2 f_{2} + \cdots
\end{equation*}
such that each $f_{j}$ is consistent with the corresponding term in the formal expansion of the solution. This ensures compatibility with the leading-order dynamics observed in the Chapman-Enskog analysis and prevents the formation of spurious initial layers or fast transients in the limit $\epsilon \rightarrow 0$ that would otherwise need to be resolved by the discretization.
\end{definition}

With this in mind, we now formally state the AP property for the macro-micro decomposition, as summarized in the following theorem. A brief proof is included for completeness.
\begin{theorem}
    \label{thm:MM AP property}
    If the initial data is well-prepared, then the macro-micro decomposition \eqref{eq:macro}--\eqref{eq:micro} preserves the diffusion limit of the kinetic equation. In particular, as $\epsilon \to 0$, the system reduces to the linear diffusion equation
    \begin{equation}
        \label{eq:diffusion limit}
        \partial_t \rho - \nabla_x \cdot \left( \frac{1}{3 \sigma_s} \nabla_x \rho \right) = - \sigma_a \rho + Q.
    \end{equation}
\end{theorem}
\begin{proof}
    Since the initial data is assumed to be well-prepared, then $g$ contains no initial layers. In the limit $\epsilon \rightarrow 0$, equation \eqref{eq:micro} reduces to
    \begin{equation*}
    g = - \frac{1}{\sigma_{s}} \mathbf{\Omega} \cdot \nabla_{x} \rho.
    \end{equation*}
    Substituting this result into the macroscopic equation \eqref{eq:macro}, and using the identity $\left\langle  \mathbf{\Omega} \otimes \mathbf{\Omega} \right\rangle_{\mathbf{\Omega}} = I/3$, we obtain the diffusion equation \eqref{eq:diffusion limit}. This completes the proof.
\end{proof}

\section{Discretization of the Macro-micro System}
\label{sec:discretization main}

This section provides details for the discretization of the macro-micro system used to build the proposed low-rank methods. We first discuss the IMEX-RK discretization for problems with relaxations, before introducing the angular discretization. Two different dimension trees are presented for representing tensors, and their interplay with the discretizations is discussed. We also analyze asymptotic properties of the method and consider the impact of tensor truncation on the preservation of the diffusion limit. Projection techniques are introduced to enforce physical constraints in the proposed low-rank methods. Finally, we highlight the complications created by element-wise products between the material cross-sections and solution tensors.

\subsection{IMEX Time Discretizations}
\label{subsec:IMEX}

Due to the stiffness associated with the diffusion limit of the system system \eqref{eq:macro}-\eqref{eq:micro}, we shall restrict the discussion to IMEX discretizations with the GSA property \cite{BoscarinoAP-IMEX2013}. It is customary to treat terms of $\mathcal{O}(1/\epsilon^{2})$ implicitly while those of size $o(1/\epsilon^{2})$ are treated explicitly, as the former will become extremely stiff in the limit $\epsilon \rightarrow 0$. The basic building block is the first-order accurate semi-discrete system
\begin{align}
    &\rho^{n+1} = \rho^{n}  - \Delta t \nabla_{x} \cdot \left( \left\langle \mathbf{\Omega} g^{n} \right\rangle_{\mathbf{\Omega}} \right) - \Delta t \sigma_{a} \rho^{n} + \Delta t Q^{n}, \label{eq: rho IMEX1} \\
    &g^{n+1} = g^{n} -\frac{\Delta t}{\epsilon} \Big( I - \Pi \Big) \Big( \nabla_{x} \cdot \left(  \mathbf{\Omega} g^{n} \right) \Big) - \frac{\Delta t}{\epsilon^{2}} \nabla_{x} \cdot \left( \mathbf{\Omega} \rho^{n+1} \right) - \frac{\Delta t \sigma_{s}}{\epsilon^{2}} g^{n+1} - \sigma_{a} \Delta t g^{n}. \label{eq: g IMEX1}
\end{align}
Using algebra, the solution $g^{n+1}$ in the second equation \eqref{eq: g IMEX1} is found to be
\begin{equation*}
    g^{n+1} = \frac{1}{\epsilon^{2} + \Delta t \sigma_{s}} \Bigg[ \epsilon^{2}g^{n} -\epsilon \Delta t \Big( I - \Pi \Big) \Big( \nabla_{x} \cdot \left(  \mathbf{\Omega} g^{n} \right) \Big) - \Delta t \nabla_{x} \cdot \left( \mathbf{\Omega} \rho^{n+1} \right) - \epsilon^{2} \Delta t \sigma_{a} g^{n} \Bigg].
\end{equation*}
Note that the coefficient multiplying the terms in brackets is spatially dependent, in the general case, and its action is defined point-wise. This form of the macroscopic equation is used in practice, although the form \eqref{eq: g IMEX1} is more convenient for the application of higher-order IMEX schemes, which we now describe.
 
IMEX-RK schemes of higher order can be represented by two Butcher tables \cite{ascher1997implicit}
\begin{equation}
    \label{eq:IMEX double tableau}
    \def\arraystretch{1.5}
    \begin{tabular}{c|c}
        $\widetilde{\mathbf{c}}$ & $\widetilde{A}$ \\
        \hline
                               & $\widetilde{\mathbf{b}}^{T}$
    \end{tabular}
    \quad\quad
    \def\arraystretch{1.5}
    \begin{tabular}{c|c}
        $\mathbf{c}$ & $A$ \\
        \hline
                               & $\mathbf{b}^{T}$
    \end{tabular}
\end{equation}
where the matrix $\widetilde{A}$ represents the explicit part of the discretization, so $\widetilde{a}_{ij} = 0$, for $j \geq i$. Its implicit counterpart $A$ is assumed to be lower triangular so that $a_{ij} = 0$, when $j > i$. Additionally, the entries of the coefficient vectors are defined as
\begin{equation*}
    \widetilde{c}_{i} = \sum_{j=1}^{i-1} \widetilde{a}_{ij}, \quad c_{i} = \sum_{j=1}^{i} a_{ij}, \quad i = 1, 2, \cdots, s.
\end{equation*}
The vectors $\widetilde{\mathbf{b}}$ and $\mathbf{b}$ represent the quadrature weights for the stages of the RK scheme.

Applying the IMEX-RK discretization to the system \eqref{eq:macro}-\eqref{eq:micro}, we obtain the following semi-discrete formulation:
\begin{align}
    \rho^{n+1} &= \rho^{n}  + \Delta t \sum_{\ell=1}^{s} \widetilde{b}_{\ell} \Bigg( -\nabla_{x} \cdot \left( \left\langle \mathbf{\Omega} g^{(\ell)} \right\rangle_{\mathbf{\Omega}} \right) - \sigma_{a} \rho^{(\ell)} + Q^{(\ell)} \Bigg), \label{eq:IMEX rho final} \\
    g^{n+1} &= g^{n} - \frac{\Delta t}{\epsilon} \sum_{\ell=1}^{s} \widetilde{b}_{\ell} \Big( I - \Pi \Big) \Big( \nabla_{x} \cdot \left(  \mathbf{\Omega} g^{(\ell)} \right) \Big) \label{eq:IMEX g final} \\ 
    &\qquad - \Delta t \sum_{\ell=1}^{s} \widetilde{b}_{\ell} \sigma_{a} g^{(\ell)} - \frac{\Delta t}{\epsilon^{2}} \sum_{\ell=1}^{s} b_{\ell} \Big( \nabla_{x} \cdot \left( \mathbf{\Omega} \rho^{(\ell)} \right) + \sigma_{s} g^{(\ell)} \Big). \nonumber
\end{align}
The values of $\rho^{(\ell)}$ and $g^{(\ell)}$ in each stage $\ell = 1, 2, \cdots, s$ are calculated as
\begin{align}
    \rho^{(\ell)} &= \rho^{n}  + \Delta t \sum_{j=1}^{\ell-1} \widetilde{a}_{\ell j} \Big( -\nabla_{x} \cdot \left( \left\langle \mathbf{\Omega} g^{(j)} \right\rangle_{\mathbf{\Omega}} \right) - \sigma_{a} \rho^{(j)} + Q^{(j)} \Big), \label{eq: IMEX rho stage} \\
    g^{(\ell)} &= g^{n} - \frac{\Delta t}{\epsilon} \sum_{j=1}^{\ell-1} \widetilde{a}_{\ell j} \Big( I - \Pi \Big) \Big( \nabla_{x} \cdot \left(  \mathbf{\Omega} g^{(j)} \right) \Big) \label{eq: IMEX g stage} \\ 
    &\qquad - \Delta t \sum_{j=1}^{\ell-1} \widetilde{a}_{\ell j} \sigma_{a} g^{(j)} - \frac{\Delta t}{\epsilon^2} \sum_{j=1}^{\ell} a_{\ell j} \Bigg( \nabla_{x} \cdot \left( \mathbf{\Omega} \rho^{(j)} \right) + \sigma_{s} g^{(j)} \Bigg). \nonumber
\end{align}
Schemes that are GSA have the property that the numerical solution at $t^{n+1}$ is equivalent to the last internal stage of the RK method \cite{dimarco2013asymptotic}, so it follows that $$\rho^{n+1} = \rho^{(s)}, \quad g^{n+1} = g^{(s)}.$$ Note that the solution of equation \eqref{eq: IMEX g stage} at each stage $\ell = 1, 2, \cdots, s$ is
\begin{multline}
    \label{eq:IMEX g stage explicit}
    g^{(\ell)} = \frac{1}{\epsilon^{2} + a_{\ell \ell} \Delta t \sigma_{s}} \Bigg[ \epsilon^{2} g^{n} - \epsilon \Delta t \sum_{j=1}^{\ell-1} \widetilde{a}_{\ell j} \Big( I - \Pi \Big) \Big( \nabla_{x} \cdot \left(  \mathbf{\Omega} g^{(j)} \right) \Big) - \epsilon^{2} \Delta t \sum_{j=1}^{\ell-1} \widetilde{a}_{\ell j} \sigma_{a} g^{(j)} \\  - \Delta t \sum_{j=1}^{\ell} a_{\ell j} \nabla_{x} \cdot \left( \mathbf{\Omega} \rho^{(j)} \right) -  \Delta t \sum_{j=1}^{\ell-1} a_{\ell j} \sigma_{s} g^{(j)} \Bigg].
\end{multline}

The IMEX discretization defined by equations \eqref{eq:IMEX rho final}-\eqref{eq: IMEX g stage} possesses an AP property, which can be established in a relatively straightforward manner if discretizations for both space and angle are treated as exact. We provide a short proof of this statement which is similar to that given in Section 2.2 of \cite{BoscarinoAP-IMEX2013}. This leads to the following theorem. 
\begin{theorem}
    \label{thm:MM AP property IMEX}
    If the discretizations for the spatial and angular variables are exact and the initial data is well-prepared, then the GSA IMEX discretization of the macro-micro decomposition, given by equations \eqref{eq:IMEX rho final}-\eqref{eq: IMEX g stage}, reduces to a consistent semi-discrete approximation of the linear diffusion equation \eqref{eq:diffusion limit} as $\epsilon \rightarrow 0$.
\end{theorem}

\begin{proof}
    First, consider the stage update \eqref{eq: IMEX g stage}, which can be written in the form
    \begin{multline*}
        \Delta t \sum_{j=1}^{\ell} a_{\ell j} \Bigg( \nabla_{x} \cdot \left( \mathbf{\Omega} \rho^{(j)} \right) + \sigma_{s} g^{(j)} \Bigg) = \epsilon^{2}\left( g^{n} - g^{(\ell)} \right) \\ - \epsilon \Delta t \sum_{j=1}^{\ell-1} \widetilde{a}_{\ell j} \Big( I - \Pi \Big) \Big( \nabla_{x} \cdot \left(  \mathbf{\Omega} g^{(j)} \right) \Big) - \epsilon^{2}\Delta t \sum_{j=1}^{\ell-1} \widetilde{a}_{\ell j} \sigma_{a} g^{(j)}.
    \end{multline*}
    Since the matrix $A$ in equation \eqref{eq:IMEX double tableau} is invertible, we can write each stage as
    \begin{multline*}
        \Delta t \Bigg( \nabla_{x} \cdot \left( \mathbf{\Omega} \rho^{(\ell)} \right) + \sigma_{s} g^{(\ell)} \Bigg) = \epsilon^{2}\sum_{j=1}^{\ell}\omega_{\ell j}\left( g^{n} - g^{(j)} \right) \\ - \epsilon \Delta t \sum_{j=1}^{\ell} \sum_{k=1}^{\ell-1} \omega_{\ell j}\widetilde{a}_{\ell k} \Big( I - \Pi \Big) \Big( \nabla_{x} \cdot \left(  \mathbf{\Omega} g^{(k)} \right) \Big) - \epsilon^{2}\Delta t \sum_{j=1}^{\ell} \sum_{k=1}^{\ell-1} \omega_{\ell j} \widetilde{a}_{\ell k} \sigma_{a} g^{(k)},
    \end{multline*}
    where the coefficients $\omega_{\ell j}$ denote the elements of $A^{-1}$. Using the assumption that $g$ is well-prepared, then in the limit $\epsilon \rightarrow 0$, the above equation simplifies to
    \begin{equation*}
        g^{(\ell)} = -\frac{1}{\sigma_{s}}\nabla_{x} \cdot \left( \mathbf{\Omega} \rho^{(\ell)} \right), \quad \ell = 1,2, \cdots, s. 
    \end{equation*}
    Multiplying this result by $\mathbf{\Omega}$ and applying $\left\langle \cdot \right\rangle_{\mathbf{\Omega}}$, we find that
    \begin{equation}
        \left\langle \mathbf{\Omega} g^{(\ell)} \right\rangle_{\mathbf{\Omega}} = -\frac{1}{\sigma_{s}} \nabla_{x} \cdot \left( \left\langle \mathbf{\Omega} \otimes \mathbf{\Omega} \right\rangle \rho^{(\ell)} \right) = -\frac{1}{3\sigma_{s}} \nabla_{x}  \rho^{(\ell)}, \quad \ell = 1,2,\cdots,s. \label{eq:limiting moment of g at stages}
    \end{equation}
    Inserting this result into the stage update \eqref{eq: IMEX rho stage}, we obtain the scheme
    \begin{equation*}
        \rho^{(\ell)} = \rho^{n}  + \Delta t \sum_{j=1}^{\ell-1} \widetilde{a}_{\ell j} \nabla_{x} \cdot \left( \frac{1}{3 \sigma_{s}} \nabla_{x}  \rho^{(j)} \right) + \Delta t \sum_{j=1}^{\ell-1} \widetilde{a}_{\ell j} \Big(- \sigma_{a} \rho^{(j)} + Q^{(j)} \Big), \quad \ell = 1,2,\cdots,s,
    \end{equation*}
    and by the GSA property of the scheme, we have that $$\rho^{n+1} = \rho^{(s)}.$$ This is an explicit temporal discretization that is consistent with limiting diffusion equation \eqref{eq:diffusion limit}. This completes the proof.
\end{proof}

\subsection{Angular Discretization}
\label{subsec:SN discretization}

In this work, we consider tensor product quadrature rules on the unit sphere in the framework of the $S_{N}$ method, which approximates the angular dependence of the solution at a fixed collection of points on the unit sphere $\mathbb{S}^{2}$. Recall that the unit vector $\mathbf{\Omega} \in \mathbb{S}^{2}$ can be defined in spherical coordinates as
\begin{equation*}
    \mathbf{\Omega} = \Big( \cos (\theta) \sin (\phi), \sin (\theta) \sin (\phi), \cos (\phi) \Big),
\end{equation*}
where $\theta \in [0, 2\pi)$ is the azimuthal angle, and $\phi \in [0, \pi]$ is the polar angle. Making the substitution $\mu = \cos (\phi)$, we can alternatively write the unit vector $\mathbf{\Omega}$ as
\begin{equation}
    \label{eq:S2 parameterization}
    \mathbf{\Omega} = \left(\xi, \eta, \mu \right) = \left( \sqrt{1 - \mu^2} \cos (\theta), \sqrt{1 - \mu^2} \sin (\theta), \mu \right),
\end{equation}
with $\mu \in [-1,1]$. Applying this change of variables to the surface integral gives
\begin{equation}
    \label{eq:angular integral change of variable}
    \frac{1}{4\pi} \int_{\mathbb{S}^{2}}  g (\mathbf{x}, \mathbf{\Omega}',t) \,d\mathbf{\Omega}' = \frac{1}{4\pi} \int_{0}^{2\pi} \int_{-1}^{1}  g (\mathbf{x}, \theta, \mu, t) \,d\mu \,d\theta.
\end{equation}

The last integral in equation \eqref{eq:angular integral change of variable} is approximated with a CL quadrature rule, which is constructed using a tensor product between a composite midpoint rule for the integral in $\theta$ and a Gauss-Legendre quadrature rule for the integral over $\mu$. Let $\{w_{\ell}, \mu_{\ell} \}$, with $\ell = 1, \cdots, N$ denote the Gauss-Legendre weights and nodes over the interval $[-1,1]$. Then the analogous weights and nodes for the CL rule are given by
\begin{equation*}
    w_{k,\ell} = \frac{\pi}{N} w_{\ell}, \quad \mathbf{\Omega}_{k,\ell} = \left( \sqrt{1 - \mu_{\ell}^{2}} \cos (\theta_{k}), \sqrt{1 - \mu_{\ell}^{2}} \sin (\theta_{k}), \mu_{\ell} \right), \quad \theta_{k} = \frac{(2k - 1)\pi}{2N},
\end{equation*}
where $k = 1, \cdots, 2N$ and $\ell = 1, \cdots, N.$
The surface integral \eqref{eq:angular integral change of variable} can then be calculated according to
\begin{equation*}
    \int_{0}^{2\pi} \int_{-1}^{1}  g (\mathbf{x}, \theta, \mu, t) \,d\mu \,d\theta \approx \frac{1}{4\pi} \sum_{k = 1}^{2N} \sum_{\ell = 1}^{N} w_{k,\ell} g_{k,\ell} (\mathbf{x}, t).
\end{equation*}
The order $N$ quadrature rule contains a total of $2N^{2}$ points and is known to integrate spherical polynomials up to degree $2N$ exactly. Alternative quadratures, such as the Lebedev rule \cite{Lebedev1976}, can provide a more uniform partitioning of the sphere, but are not a tensor product quadrature. When referring to the angular discretization, we use $S_{N}$ to indicate the discrete ordinates method obtained with a CL rule of order $N$.

\subsection{Spatial Discretization}
\label{subsec:spatial discretization}

The discretization of the spatial derivatives is performed using a fifth-order finite-difference classical WENO scheme \cite{shu2009high}. Other options for this component could be used instead, but the specifics are not the focus of the present work. We label the discrete operators on the spatial mesh as $\mathbf{D}_{x}^{\pm}$, $\mathbf{D}_{y}^{\pm}$, etc. with the sign of the superscript indicating the bias in the approximation to the derivative. For example, $\mathbf{D}_{x}^{-}$ and $\mathbf{D}_{x}^{+}$ represent, respectively, the left- and right-biased approximations to $\partial_{x}$. The result is a high-order, locally conservative discretization of the fluxes, in the sense that
\begin{equation*}
    \mathbf{D}_{x}^{\pm} f_{i} := \frac{f_{i+1/2}^{\pm} - f_{i-1/2}^{\pm}}{\Delta x} = \partial_{x} f \Big\rvert_{x = x_{i}} + \mathcal{O}\left( \Delta x^{5} \right),
\end{equation*}
in smooth regions. In the discussion that follows, we present the discretization in the setting of a two-dimensional spatial mesh, but note that the extension to three spatial dimensions is straightforward.

Applying an upwind discretization to the macroscopic equation \eqref{eq: IMEX rho stage} yields
\begin{align}
    \rho^{(\ell)} &= \rho^{n}  - \Delta t \sum_{j=1}^{\ell-1} \widetilde{a}_{\ell j} \left\langle \xi^{+} \mathbf{D}_{x}^{-} g^{(j)} + \xi^{-} \mathbf{D}_{x}^{+} g^{(j)} + \eta^{+} \mathbf{D}_{y}^{-} g^{(j)} + \eta^{-} \mathbf{D}_{y}^{+} g^{(j)} \right\rangle_{\mathbf{\Omega}} \label{eq:rho stage spatial discretization} \\
    &\qquad - \Delta t \sum_{j=1}^{\ell-1} \widetilde{a}_{\ell j} \sigma_{a} \rho^{(j)} + \Delta t \sum_{j=1}^{\ell-1} \widetilde{a}_{\ell j} Q^{(j)}, \nonumber
\end{align}
where we have defined, following \eqref{eq:S2 parameterization}, the wind directions
\begin{equation}
    \label{eq:wind directions}
    \xi^{+} = \max\left(\xi, 0\right), \quad \xi^{-} = \min\left(\xi, 0\right), \quad \eta^{+} = \max\left(\eta, 0\right), \quad \eta^{-} = \min\left(\eta, 0\right).
\end{equation}

The discretization of the microscopic equation \eqref{eq: IMEX g stage} requires a bit more care in the treatment of the fluxes to ensure stability. We apply an upwind discretization to the convection terms involving the microscopic function $g$ and an alternating discretization to the convection terms of the macroscopic function $\rho$, which gives
\begin{multline}
    g^{(\ell)} = g^{n} - \Delta t \sum_{j=1}^{\ell-1} \widetilde{a}_{\ell j} \sigma_{a} g^{(j)} - \Delta t \sum_{j=1}^{\ell} \frac{a_{\ell j}  \sigma_{s}}{\epsilon^{2}} g^{(j)} \\  - \Delta t \sum_{j=1}^{\ell-1} \frac{\widetilde{a}_{\ell j}}{\epsilon} \Big( I - \Pi \Big) \left( \xi^{+} \mathbf{D}_{x}^{-} g^{(j)} + \xi^{-} \mathbf{D}_{x}^{+} g^{(j)} + \eta^{+} \mathbf{D}_{y}^{-} g^{(j)} + \eta^{-} \mathbf{D}_{y}^{+} g^{(j)} \right) \\  
    -\Delta t \sum_{j=1}^{\ell} \frac{a_{\ell j}}{\epsilon^{2}} \Big( \xi^{-} \mathbf{D}_{x}^{-} \rho^{(j)} + \xi^{+} \mathbf{D}_{x}^{+} \rho^{(j)} + \eta^{-} \mathbf{D}_{y}^{-} \rho^{(j)} + \eta^{+} \mathbf{D}_{y}^{+} \rho^{(j)} \Big), \label{eq:g stage spatial discretization}
\end{multline}
where $\xi^{\pm}$ and $\eta^{\pm}$ are defined in equation \eqref{eq:wind directions}. The motivation behind the choice of alternating fluxes for the approximation of $\nabla_{x} \cdot \left( \mathbf{\Omega} \rho \right)$ in equation \eqref{eq:g stage spatial discretization} is that the limiting diffusion equation should use information from both wind directions when combined with equation \eqref{eq:rho stage spatial discretization}. Otherwise, the discretization for the diffusion equation will be unstable. 

\subsection{Representation of Tensors in the HTT Format}
\label{subsec:HTT}

One of the primary challenges associated with kinetic simulations is the immense storage cost associated with the underlying high-dimensional phase space. For example, in multigroup transport calculations, which account for particle interactions at different energy levels $\lambda$, the number of dimensions can be as high as seven ($\mathbf{x} \in \mathbb{R}^{3}$, $\mathbf{\Omega} \in \mathbb{S}^{2}$, $\lambda \in \mathbb{R}$) plus time. The use of standard storage containers, such as arrays, quickly become impractical for such problems. In this work, we instead, seek a representation for the components of the system \eqref{eq:macro}-\eqref{eq:micro} in the HTT format \cite{HackbuschKuhnHT2009,GrasedyckHTT2010}, in which functions are represented hierarchically as a linear combination of tensor products. The dimensions of the problem, e.g., $(x,y,\theta,\mu)$, are organized in a binary tree that can be partitioned into disjoint subsets until the leaf nodes of the tree contain only a singleton set. The HOSVD \cite{DeLathauwerHOSVD2000} is applied recursively to nodes of this tree and generates the basis at the leaf nodes as well as the transfer tensors at the non-leaf nodes, which describe the interactions between the different dimensions. 

An advantage of the HTT format is its flexibility in the arrangement of the dimension tree which provides opportunities to exploit low-rank structures in the evolution of the distribution function. A key difference with previous work \cite{EinkemmerDLR-AP,hu2022adaptiveBoltzmann} is that we allow for the possibility of low-rank structures between the individual dimensions of the spatial and angular domains. The methods we propose are naturally rank-adaptive and retain their efficiency through HOSVD truncation. To demonstrate, consider the representation of the function $g(x, y, \theta, \mu, t)$ in the HTT format. For convenience, we shall encode the dimensions of the function with the following integer labels $\left\{x, y, \theta, \mu \right\} \mapsto \left\{1, 2, 3, 4 \right\}$. Then, we define a binary tree $\mathcal{T}$ whose nodes represent subsets of the dimension labels. The root node of $\mathcal{T}$ includes all of the dimensions, and each non-leaf node contains two children. At the continuous level, a fully tensorized approximation of $g$ in the HTT format is given by
\begin{equation}
    g(x, y, \theta, \mu,t) \approx \sum_{\ell_{12} = 1}^{r_{12}} \sum_{\ell_{34} = 1}^{r_{34}} B_{\ell_{12}, \ell_{34}}^{(1,2,3,4)} (t) U_{\ell_{12}}^{(1,2)} (x, y, t) \otimes U_{\ell_{34}}^{(3,4)} (\theta, \mu, t), \label{eq:g HTT}
\end{equation}
where the bases $U$ are defined hierarchically as
\begin{align}
    U_{\ell_{12}}^{(1,2)}(x,y,t) &= \sum_{\ell_{1}=1}^{r_{1}} \sum_{\ell_{2}=1}^{r_{2}} B_{\ell_{1},\ell_{2},\ell_{12}}^{(1,2)}(t)  U_{\ell_{1}}^{(1)} (x, t) \otimes U_{\ell_{2}}^{(2)} (y, t), \label{eq:space HTT} \\ 
    U_{\ell_{34}}^{(3,4)}(\theta, \mu, t) &= \sum_{\ell_{3}=1}^{r_{3}} \sum_{\ell_{4}=1}^{r_{4}} B_{\ell_{3},\ell_{4},\ell_{34}}^{(3,4)} (t)  U_{\ell_{3}}^{(3)} (\theta, t)\otimes U_{\ell_{4}}^{(4)} (\mu, t). \label{eq:angular HTT}
\end{align}
Here $\mathbf{r} = \left\{ r_{1234}, r_{12}, r_{34}, r_{1}, r_{2}, r_{3}, r_{4} \right\}$ are the hierarchical ranks at each of the nodes in the dimension tree. Within a given step, the basis for each dimension ($U$'s at the leaf nodes) and the corresponding interaction coefficients ($B$'s at the non-leaf nodes) can be evolved following a SAT approach \cite{GuoVlasovFlowMap2022}. To address problems posed on general spatial domains, we adjust the organization of the dimension tree by forgoing the tensorization \eqref{eq:space HTT}. We refer to this as an ``unsplit" representation. \Cref{fig:HTT architectures} provides a simple illustration of the unsplit and split dimension trees considered in this work. We remark that the HTT format can be considered a generalization of other tensor formats, including the popular tensor train decomposition \cite{TToseledets2011}, which uses a maximally unbalanced dimension tree.

\begin{figure}[t]
    \centering
    \begin{subfloat}[Unsplit]{
    \begin{tikzpicture}[scale=1.1]
        \node(1234)[dim_node] {$\left\{ x, y, \theta, \mu \right\}$}
            child{node(12)[dim_node] {$\left\{ x, y \right\}$}}
            child{node(34)[dim_node] {$\left\{ \theta, \mu \right\}$} 
                child{node(3)[dim_node] {$\left\{ \theta \right\}$}}
                child{node(4)[dim_node] {$\left\{ \mu \right\}$}}
            };
    \end{tikzpicture}
    }
    \end{subfloat}
    \begin{subfloat}[Split]{
    \begin{tikzpicture}[scale=1.1]
        \node(1234)[dim_node] {$\left\{ x, y, \theta, \mu \right\}$}
            child{node(12)[dim_node] {$\left\{ x, y \right\}$} 
                child{node(1)[dim_node] {$\left\{ x \right\}$}}
                child{node(2)[dim_node] {$\left\{ y \right\}$}}   
            }
            child{node(34)[dim_node] {$\left\{ \theta, \mu \right\}$} 
                child{node(3)[dim_node] {$\left\{ \theta \right\}$}}
                child{node(4)[dim_node] {$\left\{ \mu \right\}$}}
            };
    \end{tikzpicture}
    }
    \end{subfloat}
    \caption{Two candidate dimension trees for a four-dimensional tensor in the HTT format. The unsplit approach (left) treats the spatial variables as a single leaf node, while the split approach (right) assumes a fully tensorized phase space in which the leaf nodes are singleton sets.}
    \label{fig:HTT architectures}
\end{figure}

In the semi-discrete setting, the function $g$ is sampled on a mesh resulting in a tensor with mode sizes $N_{\mu}$, where $\mu = 1, \cdots, 4$. The approximation for $g$ is given by
\begin{equation}
    \label{eq:g HTT discrete}
    \mathbf{g}(t) \approx \sum_{\ell_{12} = 1}^{r_{12}} \sum_{\ell_{34} = 1}^{r_{34}} \mathbf{B}_{\ell_{12}, \ell_{34}}^{(1,2,3,4)} (t) \mathbf{U}_{\ell_{12}}^{(1,2)} (t) \otimes \mathbf{U}_{\ell_{34}}^{(3,4)} (t).
\end{equation}
Using the split dimension tree, the matrices for the bases are further expressed as
\begin{align}
    \mathbf{U}_{\ell_{12}}^{(1,2)}(t) &= \sum_{\ell_{1}=1}^{r_{1}} \sum_{\ell_{2}=1}^{r_{2}} \mathbf{B}_{\ell_{1},\ell_{2},\ell_{12}}^{(1,2)}(t)  \mathbf{U}_{\ell_{1}}^{(1)} (t) \otimes \mathbf{U}_{\ell_{2}}^{(2)} (t), \label{eq:space HTT discrete} \\ 
    \mathbf{U}_{\ell_{34}}^{(3,4)}(t) &= \sum_{\ell_{3}=1}^{r_{3}} \sum_{\ell_{4}=1}^{r_{4}} \mathbf{B}_{\ell_{3},\ell_{4},\ell_{34}}^{(3,4)} (t)  \mathbf{U}_{\ell_{3}}^{(3)} (t)\otimes \mathbf{U}_{\ell_{4}}^{(4)} (t), \label{eq:angular HTT discrete}
\end{align}

Discrete derivatives, such as $\mathbf{D}_{x}^{\pm}$, can be applied to the discrete tensor $\mathbf{g}(t)$ by applying matrix-vector products to the column vectors of the corresponding leaf node basis in the tree. For example, to compute a left-biased derivative $\mathbf{D}_{x}^{-} \mathbf{g}(t) := \mathbf{D}_{1}^{-} \mathbf{g}(t)$, we can use the definition \eqref{eq:space HTT discrete} used in the representation \eqref{eq:g HTT discrete} to find that
\begin{equation*}
    \mathbf{D}_{1}^{-}  \mathbf{U}_{\ell_{12}}^{(1,2)}(t) = \sum_{\ell_{1}=1}^{r_{1}} \sum_{\ell_{2}=1}^{r_{2}} \mathbf{B}_{\ell_{1},\ell_{2},\ell_{12}}^{(1,2)}(t) \left( \mathbf{D}_{1}^{-} \mathbf{U}_{\ell_{1}}^{(1)} (t) \right) \otimes \mathbf{U}_{\ell_{2}}^{(2)} (t).
\end{equation*}
In the split dimension representation, WENO differentiation is applied in a matrix-free manner to each column of the basis associated with a given mode, using periodic extensions to enforce boundary conditions.

Other operations, e.g., angular integration, are can be efficiently performed by contracting a quadrature tensor with the discrete representation defined by \eqref{eq:g HTT} through its basis \eqref{eq:angular HTT discrete} along the appropriate set of dimensions. At the end of each RK stage, the HOSVD is applied to remove redundancies in the basis. Tensor truncation is a essential to the efficiency of the SAT approach \cite{GuoVlasovFlowMap2022}, so it is relevant to discuss its cost. If we let $r$ denote the maximum hierarchical rank and $N$ mode size in each dimension, then the cost of truncating tensors in the HTT format is $\mathcal{O} \left(dNr^{2} + dr^{4}\right)$ \cite{GrasedyckHTT2010,kressner2012htucker}. Furthermore, we remark that the storage cost of such tensors is $\mathcal{O} \left(dNr + (d-2)r^{3} + r^{2}\right)$, which avoids the $\mathcal{O}(N^{d})$ scaling from the CoD.

We now establish an AP property of the low-rank methods under the assumption of exact discretizations for the spatial and angular domains. First, we recall a known error bound \cite{GrasedyckHTT2010} (See Theorem 3.11) for the truncation of tensors in the HTT format:
\begin{theorem}
    \label{thm:HTT truncation}
    Let $\mathcal{A} \in \mathbb{R}^{n_{1} \times n_{2} \times \cdots \times n_{d}}$ be a $d$-dimensional tensor, and let $\widetilde{\mathcal{A}}$ denote the truncation of $\mathcal{A}$ to a HTT with dimension tree $\mathcal{T}$ with hierarchical ranks $\left\{ r_{t} : t\in \mathcal{T} \right\}$. Then the resulting approximation satisfies the following bound
    \begin{equation*}
        \left\lvert\left\lvert \mathcal{A} - \widetilde{\mathcal{A}} \right\rvert\right\rvert \leq \sqrt{\sum_{t\in \mathcal{T}}\sum_{i > r_{t}} \sigma_{t,i}^{2} },
    \end{equation*}
    where $\sigma_{t,i} = \sigma_{i}\left( \mathcal{A}^{(t)} \right)$ denotes the singular values of the mode $t$ matricization of $\mathcal{A}$ and $ \left\lvert\left\lvert \cdot \right\rvert\right\rvert $ denotes the Frobenius norm which is defined as
    $$ \left\lvert\left\lvert \mathcal{A} \right\rvert\right\rvert = \left( \sum_{i_{1} = 1}^{n_{1}}\cdots \sum_{i_{d} = 1}^{n_{d}} \mathcal{A}_{i_{1,}\cdots, i_{d}}^{2} \right)^{1/2}.$$
\end{theorem}
In what follows it is helpful to write the error estimate in \cref{thm:HTT truncation} in the form
\begin{equation*}
    \left\lvert\left\lvert \mathcal{A} - \widetilde{\mathcal{A}} \right\rvert\right\rvert \lesssim \max_{t \in \mathcal{T}} \sigma_{t,r_{t}+1}.
\end{equation*}
$\lesssim$ denotes less than or equal with a constant, and here the constant depends on the tree $\mathcal{T}$ and tensor mode sizes $\{n_{1}, n_{2}, \cdots, n_{d}\}$. While this constant can be large, it will nonetheless remain finite.

Next, we establish a bound on the difference between the macroscopic flux and its diffusion limit. Here we assume that tensor truncation is applied only on the microscopic variable $g$.
\begin{lemma}
\label{lem:unsplit asymptotic error bound}
    Let $\boldsymbol{\rho}^{(\ell)}$ and $\mathbf{g}^{(\ell)}$ denote the solution tensors at a given stage $\ell = 1,2, \cdots s$ of the system \eqref{eq:macro}-\eqref{eq:micro} that is discretized using a GSA IMEX method. Also, let $\widetilde{\mathbf{g}}^{(\ell)}$ denote the low-rank decomposition of $\mathbf{g}^{(\ell)}$ in the HTT format with an unsplit dimension tree $\mathcal{T}$. Suppose, in addition, that no truncation is applied to the tensor $\boldsymbol{\rho}^{(\ell)}$ and that the initial data are well-prepared with $\nabla_{x}g$ remaining bounded. If the spatial and angular discretizations are exact, then the solution at each stage satisfies the inequality
    $$
        \left\lVert -\frac{1}{3\sigma_s} \nabla_x \boldsymbol{\rho}^{(\ell)} - \left\langle \mathbf{\Omega} \widetilde{\mathbf{g}}^{(\ell)} \right\rangle_{\mathbf{\Omega}} \right\rVert \lesssim \epsilon + (2 + \epsilon)\max_{\substack{1 \leq j < \ell \\ t \in \mathcal{T}}} \sigma_{t,r_{t}+1}^{(j)},
    $$
    where $\sigma_{t,r_{t}+1}^{(j)}$ denotes the leading singular value of the mode $t$ matricization of $\mathbf{g}$ at stage $j$ which is removed by truncation.
\end{lemma}
\begin{proof}
    Using the triangle inequality, we can write
    \begin{align*}
        \Bigg\lVert -\frac{1}{3\sigma_s} \nabla_x \boldsymbol{\rho}^{(\ell)} - \left\langle \mathbf{\Omega} \widetilde{\mathbf{g}}^{(\ell)} \right\rangle_{\mathbf{\Omega}} \Bigg\rVert &\leq \Bigg\lVert -\frac{1}{3\sigma_s} \nabla_x \boldsymbol{\rho}^{(\ell)} - \left\langle \mathbf{\Omega} \mathbf{g}^{(\ell)} \right\rangle_{\mathbf{\Omega}} \Bigg\rVert \\ &+ \Bigg\lVert \left\langle \mathbf{\Omega} \mathbf{g}^{(\ell)} \right\rangle_{\mathbf{\Omega}} - \left\langle \mathbf{\Omega}\widetilde{\mathbf{g}}^{(\ell)} \right\rangle_{\mathbf{\Omega}} \Bigg\rVert.
    \end{align*}
    The first term clearly satisfies
    \begin{equation*}
        \left\lVert -\frac{1}{3\sigma_s} \nabla_x \boldsymbol{\rho}^{(\ell)} - \left\langle \mathbf{\Omega} \mathbf{g}^{(\ell)} \right\rangle_{\mathbf{\Omega}} \right\rVert \lesssim \epsilon, \quad \ell = 1,2, \cdots, s,
    \end{equation*}
    since the full-grid scheme is AP in the sense of \cref{thm:MM AP property IMEX}.

    The treatment of the second term requires a bit more care. First, consider the stage update \eqref{eq: IMEX rho stage}. Let $\widetilde{\boldsymbol{\rho}}^{(\ell)}$ denote the density obtained from the update \eqref{eq: IMEX rho stage} with $\widetilde{\mathbf{g}}^{(\ell)}$ in place of $\mathbf{g}^{(\ell)}$. Then, for each $\ell = 1, 2, \cdots, s$, the inequality
    \begin{align}
        \Big\lVert \boldsymbol{\rho}^{(\ell)} - \widetilde{\boldsymbol{\rho}}^{(\ell)} \Big\rVert &\lesssim \Big\lVert \boldsymbol{\rho}^{n} - \widetilde{\boldsymbol{\rho}}^{n} \Big\rVert  + \sum_{j=1}^{\ell-1} \Big\lVert \mathbf{g}^{(j)} - \widetilde{\mathbf{g}}^{(j)} \Big\rVert + \sum_{j=1}^{\ell-1} \Big\lVert \boldsymbol{\rho}^{(j)} - \widetilde{\boldsymbol{\rho}}^{(j)} \Big\rVert, \label{eq:rho stage bound} \\
        &\lesssim \max_{\substack{1 \leq j < \ell \\ t \in \mathcal{T}}} \sigma_{t,r_{t}+1}^{(j)}, \nonumber
    \end{align}
    holds, where $\sigma_{t,r_{t}+1}^{(j)}$ denotes the leading singular value of the mode $t$ matricization of $\mathbf{g}$ at stage $j$ which is removed by truncation. This argument relies on the boundedness of angular projections and $\nabla_{x}g$.

    Next, we apply a similar argument to the stage update \eqref{eq:IMEX g stage explicit}, and with the aid of the inequality \eqref{eq:rho stage bound}, we obtain
    \begin{align*}
        \Big\lVert \left\langle \mathbf{\Omega} \mathbf{g}^{(\ell)} \right\rangle_{\mathbf{\Omega}} - \left\langle \mathbf{\Omega}\widetilde{\mathbf{g}}^{(\ell)} \right\rangle_{\mathbf{\Omega}} \Big\rVert &\lesssim \Big\lVert \mathbf{g}^{(\ell)} - \widetilde{\mathbf{g}}^{(\ell)} \Big\rVert, \\
        & \lesssim \epsilon^{2}  \Big\lVert \mathbf{g}^{n} - \widetilde{\mathbf{g}}^{n} \Big\rVert + (1 + \epsilon) \sum_{j=1}^{\ell-1} \Big\lVert \mathbf{g}^{(j)} - \widetilde{\mathbf{g}}^{(j)} \Big\rVert \\ &+ \epsilon^{2} \sum_{j=1}^{\ell-1}  \Big\lVert \mathbf{g}^{(j)} - \widetilde{\mathbf{g}}^{(j)} \Big\rVert + \sum_{j=1}^{\ell}\Big\lVert \boldsymbol{\rho}^{(j)} - \widetilde{\boldsymbol{\rho}}^{(j)} \Big\rVert, \\
        & \lesssim (2 + \epsilon) \max_{\substack{1 \leq j < \ell \\ t \in \mathcal{T}}} \sigma_{t,r_{t}+1}^{(j)}.
    \end{align*}
    Combining inequalities together yields the conclusion, which completes the proof.
\end{proof}

We are now prepared to state our main result.
\begin{theorem}
\label{thm:AP property with HTT}
Suppose that the initial data are well-prepared and that $\nabla_{x}g$ remains bounded. Assume that exact discretizations are used for the spatial and angular variables, and that time integration is performed using a GSA IMEX discretization. Further suppose that $g$ is represented in the HTT format using an unsplit dimension tree $\mathcal{T}$ and that no truncation is applied to the macroscopic density. Then, in the limit $\epsilon \rightarrow 0$, the scheme reduces to a consistent discretization of the linear diffusion equation \eqref{eq:diffusion limit}.
\end{theorem}

\begin{proof}
    To achieve the AP property, we must show that as $\epsilon \rightarrow 0$, then
    $$
        \left\lVert -\frac{1}{3\sigma_s} \nabla_x \boldsymbol{\rho}^{(\ell)} - \left\langle \mathbf{\Omega} \widetilde{\mathbf{g}}^{(\ell)} \right\rangle_{\mathbf{\Omega}} \right\rVert \rightarrow 0, \quad \ell = 1, 2, \cdots, s,
    $$
    as well. A direct application of \cref{lem:unsplit asymptotic error bound} yields the bound
    $$
        \left\lVert -\frac{1}{3\sigma_s} \nabla_x \boldsymbol{\rho}^{(\ell)} - \left\langle \mathbf{\Omega} \widetilde{\mathbf{g}}^{(\ell)} \right\rangle_{\mathbf{\Omega}} \right\rVert \lesssim \epsilon + (2 + \epsilon)\max_{\substack{1 \leq j < \ell \\ t \in \mathcal{T}}} \sigma_{t,r_{t}+1}^{(j)},
    $$
    where $\widetilde{\mathbf{g}}^{(\ell)}$ is the HTT approximation of $g^{(\ell)}$ and $\sigma_{t,r_{t}+1}^{(j)}$ denotes the leading singular value of the mode $t$ matricization of $\mathbf{g}$ at stage $j$ which is discarded during truncation.

    From the estimate above, this convergence holds in the limit $\epsilon \rightarrow 0$ provided that the singular values satisfy
    \begin{equation*}
        \max_{\substack{1 \leq j < \ell \\ t \in \mathcal{T}}} \sigma_{t,r_{t}+1}^{(j)} \rightarrow 0.
    \end{equation*}

    From equation \eqref{eq:limiting moment of g at stages}, we see that the leading order solution at each stage reduces to
    $$
        \left\langle \boldsymbol{\Omega} \, \mathbf{g}^{(\ell)} \right\rangle_{\boldsymbol{\Omega}} = -\frac{1}{3\sigma_s} \nabla_x \boldsymbol{\rho}^{(\ell)} \otimes \mathbf{1}_{\theta} \otimes \mathbf{1}_{\mu}, \quad \ell = 1, 2, \dots, s,
    $$
    which corresponds to a tensor in the HTT format with $r_{t} = 1$ for $t \in \mathcal{T}$. Therefore, the scheme will be AP as long as $r_{t} \geq 1$ for $t \in \mathcal{T}$, which is always true.
\end{proof}

\begin{remark}
    In order to preserve the AP property, we note that it is not necessary to require that the truncation tolerance be selected so that it is majorized by $\epsilon$, i.e.,
    \begin{equation*}
        \max_{\substack{1 \leq j < \ell \\ t \in \mathcal{T}}} \sigma_{t,r_{t}+1}^{(j)} \lesssim \epsilon.
    \end{equation*}
    This condition is considerably stronger than required and may be overly conservative in practice. In fact, \cref{thm:AP property with HTT} demonstrates that the AP property is achieved even with a significantly larger truncation tolerance.
\end{remark}
\begin{remark}
    \cref{thm:AP property with HTT} assumes that the unsplit dimension tree is used in the representation of the microscopic component. A similar AP property can be established for the split dimension tree by simply using a smaller truncation tolerance, so that nodes in the dimension tree corresponding to space essentially remain full-rank. However, it may be possible to make a more refined argument based on the decay properties of the solution's spectrum when $\epsilon \rightarrow 0$.
\end{remark}

\subsection{Projection Techniques to Enforce Physical Constraints}

This section outlines two projection techniques that can be used to enforce physical constraints for the macro-micro system. The first projection can be used with the split low-rank approach to preserve the total mass of the system in problems that do not contain sources or absorption effects. The second projection technique is compatible with any of the dimension trees and ensures that the function $g$ has zero density.

\subsubsection{Conservation of Total Mass}
\label{subsubsec:mass projection}

Given the density $\rho$ at two discrete time levels, e.g., $t^{n}$ and $t^{n+1}$, it is desirable to have
\begin{equation}
    \label{eq:mass conservation property}
    \int_{D} \rho^{n+1}(\mathbf{x}) \, d\mathbf{x} = \int_{D} \rho^{0}(\mathbf{x}) \, d\mathbf{x},
\end{equation}
where $D$ represents the spatial domain in the problem. In the case of the split dimension tree, it is convenient to represent the density in the HTT format as well. To avoid growth in the storage requirements, however, it is necessary to apply HOSVD truncation to this tensor, which destroys the conservation property of the scheme. 

In problems that neglect sources or absorption effects, the total mass is conserved, and we can use a simple technique to satisfy the condition \eqref{eq:mass conservation property}. The total mass of the system at time $t^{n}$, on domain $D$, is
\begin{equation}
    \label{eq:total mass at t^n}
    M^{n} = \int_{D} \rho^{n}(\mathbf{x}) \, d\mathbf{x}.
\end{equation}
Then, we define the adjusted density
\begin{equation}
    \label{eq:adjusted density}
    \widetilde{\rho}^{n+1}(\mathbf{x}) = \rho^{n+1}(\mathbf{x}) + \left(M^{0} - M^{n+1} \right) \delta \rho^{n+1} (\mathbf{x}),
\end{equation}
where $\rho^{n+1}$ is a truncated density with total mass $M^{n+1}$, $M^{0}$ is the target mass, and the correction $\delta \rho^{n+1} (\mathbf{x})$ satisfies
\begin{equation}
    \label{eq:total mass correction constraint}
    \int_{D} \delta \rho^{n+1} (\mathbf{x}) \, d\mathbf{x} = 1.
\end{equation}
In this work, we select the constant function $\delta \rho^{n+1} (\mathbf{x}) = 1$ and normalize the result by the volume of the computational domain so that it satisfies property \eqref{eq:total mass correction constraint}. This has the effect of shifting the mass, identically, at each of the grid points by a small amount. At the end of the time step, we redefine the density as $\rho^{n+1} = \widetilde{\rho}^{n+1}$, which satisfies the property \eqref{eq:mass conservation property}. This fact can be confirmed through a simple integration of equation \eqref{eq:adjusted density} on the domain $D$, making use of property \eqref{eq:total mass correction constraint}.


\subsubsection{The Zero Density Constraint for the Microscopic Component}
\label{subsubsec:projection for g}

Under the macro-micro decomposition \eqref{eq:decomposition}, it is assumed that the microscopic component $g$ satisfies the zero density condition $\langle g \rangle_{\mathbf{\Omega}} = 0$. However, if we apply HOSVD truncation to the microscopic component $g$, this property is no longer guaranteed. To address this, we apply an orthogonal decomposition of the tensor $g$ to project away numerical artifacts introduced by the time advance of the solution and subsequent rank truncation. The approach we take is similar in spirit to the LoMaC truncation technique of Guo and Qiu \cite{GuoVlasovDLRVlasovDynamics,GuoVlasovLoMacDG2023}. 

To proceed, we first apply non-conservative HOSVD truncation with tolerance $\epsilon$ to the microscopic component $g$. We shall denote this quantity as $\mathcal{T}_{\epsilon}(g)$. Then we compute the macroscopic density associated with the truncated $\mathcal{T}_{\epsilon}(g)$ as
\begin{equation}
    \label{eq:rho_g projection step}
    \rho_{g} = \left\langle \mathcal{T}_{\epsilon}(g) \right\rangle_{\mathbf{\Omega}}.
\end{equation}
Next, we define an orthogonal projection relative to the subspace
\begin{equation*}
    \mathcal{N} = \text{span} \left\{ \mathbf{1}_{\theta} \otimes \mathbf{1}_{\mu} \right\},
\end{equation*}
then define a projection of $\mathcal{T}_{\epsilon}(g)$ onto this subspace as
\begin{equation}
    \label{eq:g_N}
    g_{\mathcal{N}} := \frac{\rho_{g} \otimes \left( \mathbf{1}_{\theta} \otimes \mathbf{1}_{\mu} \right)}{ \left\langle \mathbf{1}_{\theta} \otimes \mathbf{1}_{\mu} \right\rangle_{\mathbf{\Omega}} } = \rho_{g} \otimes \left( \mathbf{1}_{\theta} \otimes \mathbf{1}_{\mu} \right).
\end{equation}
A projection onto the orthogonal complement $\mathcal{N}^{\perp}$ is given by
\begin{equation}
    \label{eq:g_N perp}
    g_{\mathcal{N}^{\perp}} = \mathcal{T}_{\epsilon}(g) - g_{\mathcal{N}}.
\end{equation}
Applying the operator $\left\langle \cdot \right\rangle_{\mathbf{\Omega}}$ to equation \eqref{eq:g_N perp} and making use of the definition \eqref{eq:rho_g projection step}, we obtain $\langle g_{\mathcal{N}^{\perp}} \rangle_{\mathbf{\Omega}} = 0$. Therefore, to enforce the angular moment constraint, we set $g = g_{\mathcal{N}^{\perp}}$. While this projection can be applied following each stage of the RK method, we find that it is sufficient, in practice, to apply the projection only at the end of each time step. Additionally, to retain the efficiency of the low-rank methods, we perform a truncation step following the definition \eqref{eq:g_N perp}.

\subsection{Products Involving the Cross-sections}
\label{subsec:cross-sections and hadamard}

An important physical process of radiation is its interaction with background media, which is built-in to terms such as $\sigma_{a} \rho$, $\sigma_{a} g$, and $\sigma_{s} g$. When the domain is discretized and the data is represented using standard arrays for containers, the evaluation of these terms is achieved using point-wise multiplication (in space) via the Hadamard product. If $\rho$ or $g$ are represented in the HTT format discussed in the previous section, then the details are a bit more delicate. In such circumstances, we require a corresponding low-rank representation for the cross-sections to calculate this product. Consider the product $\sigma g$, where $\sigma = \sigma(\mathbf{x})$ could describe scattering or absorption. In a more general application, the cross-section $\sigma$ depends on the properties of the background material, such as the density and temperature. For the sake of this discussion, we shall assume it is only a function of space and is time-independent. Rather than make additional simplifying assumptions about the structure of the cross-section, we provide some insight into the complexity associated with the evaluation of products between the cross-section and a tensor expressed in the HTT format. These complications are worth mentioning as they are connected to the overall efficiency of the method. 

First, we consider the unsplit dimension tree, in which the calculation of these products is relatively straightforward. In this setting, the material cross-section can be interpreted as a rank-1 tensor $\boldsymbol{\sigma} (\mathbf{x}) \otimes \mathbf{1}_{\theta} \otimes \mathbf{1}_{\mu}$. If we assume that the matrix $\mathbf{U}^{(1,2)}$, the spatial basis for $\mathbf{g}$ at the leaf node $\{1,2\}$, has rank $r_{12}$, then the evaluation of the Hadamard product can be performed, in a direct way, using a total of $r_{12}$ point-wise multiplications between the cross-section $\boldsymbol{\sigma} (\mathbf{x})$ and the columns of the matrix $\mathbf{U}^{(1,2)}$.

The calculation for the split dimension tree is a bit more involved. Suppose that $g$ has a decomposition of the form \eqref{eq:g HTT discrete} with the basis defined by \eqref{eq:space HTT discrete} and \eqref{eq:angular HTT discrete}. The \texttt{htucker} library \cite{kressner2012htucker} provides two methods to compute a Hadamard product between two tensors in the HTT format. The first approach directly evaluates the product using the operator \texttt{.*}, while the second method \texttt{elem\textunderscore mult}, which is the one used in this work, computes an approximate Hadamard product. However, there are several behavioral differences between them, and the details can be found in the reference \cite{kressner2012htucker}. The former suffers from severe rank growth, while the latter forms a truncated product that greatly reduces the storage cost of the evaluation. Note that in order to use these methods, we require a similar tensorization of the cross-section. A two-dimensional HTT for the cross-section can be obtained from the full tensor using the \texttt{truncate} method, which can then be extended to a four-dimensional tensor as
\begin{align*}
    \boldsymbol{\sigma} &\approx \sum_{\ell_{12} = 1}^{k_{12}} \mathbf{B}_{\ell_{12}}^{(1,2,3,4)} \mathbf{U}_{\ell_{12}}^{(1,2)} \otimes \mathbf{1}_{\boldsymbol{\theta}} \otimes \mathbf{1}_{\boldsymbol{\mu}}, \\
    \mathbf{U}_{\ell_{12}}^{(1,2)} &= \sum_{\ell_{1}=1}^{k_{1}} \sum_{\ell_{2}=1}^{k_{2}} \mathbf{B}_{\ell_{1},\ell_{2},\ell_{12}}^{(1,2)}  \mathbf{U}_{\ell_{1}}^{(1)} \otimes \mathbf{U}_{\ell_{2}}^{(2)},
\end{align*}
where the basis ($U$'s) and transfer tensors ($B$'s) are different from \eqref{eq:g HTT discrete}. Once the function \texttt{elem\textunderscore mult} completes, an additional truncation step is necessary to mitigate additional rank growth incurred in the evaluation.

%
%
\section{Numerical Results}
\label{sec:results}

The implementation of the proposed low-rank methods was performed using the \texttt{htucker} library \cite{kressner2012htucker}. We consider the unsplit and split dimension trees provided in \cref{fig:HTT architectures} for the representation of the microscopic variable $g$. Unless otherwise specified, we use a third-order ARS-GSA IMEX time discretization \cite{BoscarinoAP-IMEX2013} and a fifth-order finite-difference WENO spatial discretization \cite{shu2009high}. The maximum rank of any tensor in the HTT format is set to $512$, which imposes an upper limit of $512^{3}$ on the sizes of the transfer tensors at each of the non-leaf nodes of the tensor tree. However, it should be noted that this limit can be adjusted based on the memory specifications of the system. Furthermore, in both the unsplit and split tensor representations, we use a relative truncation tolerance of $10^{-4}$ for the function $g$. For the split approach, in which the function $\rho$ is also stored in the HTT format, we set the relative truncation tolerance to $10^{-16}$. 

To compare the efficiency of the methods, we compare the storage requirements used in both the low-rank and full-grid representations of the function $g$ by tracking the total number of degrees-of-freedom (DOF) required by each method. The total DOF used in the full-grid method is $N_{x} N_{y} N_{\theta} N_{\mu}$, which remains fixed over time. In contrast, the total DOF for the low-rank method will vary between time steps due to rank adaptivity, with the total DOF at any given time corresponding to a sum of the total number of entries taken across the nodes of the tensor tree. We use the function \texttt{ndofs} in the \texttt{htucker} library to provide this information. In order to characterize the compression offered by the low-rank methods, we calculate a time-dependent compression ratio for the high-dimensional function $g$ defined as
\begin{equation}
    \label{eq:compression ratio}
    \text{Compression Ratio} = (\text{Low-rank total DOF at time $t$})/(\text{Full-grid total DOF})
\end{equation}

We first show convergence of the low-rank methods using a manufactured solution before considering some benchmark problems found in the literature. For problems without an analytical solution, we compare the results obtained with the proposed low-rank methods against an analogous full-grid implementation. The full-grid solutions used in this section are obtained with the same third-order IMEX time discretization and fifth-order WENO scheme. The efficiency of the proposed solvers is discussed, and we highlight the particular challenges in each of the examples.

\subsection{Refinement with a Manufactured Solution}
\label{subsec:manufactured test}

We first verify the order of accuracy of the method using a rank-2 manufactured solution of the form
\begin{equation*}
    f(x,y,\theta,\mu,t) = 2 + e^{-t} \sin \left( 2\pi x\right) \sin \left( 2\pi y\right) + \epsilon e^{-t} \sin \left( 2\pi x\right) \sin \left( 2\pi y\right) \sin \left(\theta\right) \sqrt{1-\mu^{2}}.
\end{equation*}
Under the macro-micro decomposition, we identify
\begin{align}
    \rho(x,y,t) &= 2 + e^{-t} \sin \left( 2\pi x\right) \sin \left( 2\pi y\right), \label{eq:manufactured solution rho} \\
    g(x,y,\theta, \mu, t) &= e^{-t} \sin \left( 2\pi x\right) \sin \left( 2\pi y\right) \sin \left(\theta\right) \sqrt{1-\mu^{2}}. \label{eq:manufactured solution g}
\end{align}
The manufactured solution defines a source $Q(\mathbf{x}, \mathbf{\Omega},t)$ for the kinetic equation \eqref{eq:linear kinetic equation}, which can be obtained from a direct calculation. This test uses periodic boundary conditions on the domain $[0, 1]^{2}$ and an $S_{8}$ angular discretization in each case. For the material parameters, we used $\sigma_{s}(\mathbf{x}) = 1$ and $\sigma_{a}(\mathbf{x}) = 0$. We also investigate the behavior in different asymptotic regimes, namely the kinetic, intermediate, and diffusive regimes. We ran the simulations to a final time of $T = 0.1$ and set the time step as $\Delta t = 0.1 \epsilon \Delta x + 0.1 \Delta x^{2}.$
No projections were used in this example.

\Cref{tab:manufactured test IMEX-3 unsplit} and \cref{tab:manufactured test IMEX-3 split} present the results of the refinement study using the unsplit and split representations, respectively. In many cases, we find that the unsplit scheme produces a more accurate solution in the kinetic and intermediate regimes than the fully split dimension tree. A likely explanation for this is that the unsplit approach does not need to truncate the density $\rho$, which can degrade the quality of the solution. However, in the kinetic regime, the scheme is quite accurate and refines roughly at the convergence rate of the spatial discretization. As $\epsilon$ decreases, the error for a given mesh increases and tends to a first-order discretization of the linear diffusion equation, but we expect second-order accuracy due to the way in which we select the time step. In particular, as $\epsilon \rightarrow 0$, the time step becomes $\Delta t \sim \mathcal{O}(\Delta x^{2})$. Further, as the grid is refined, we see that the hierarchical rank of the microscopic component $g$ tends to 1 for both the unsplit and split dimension trees, which agrees with the rank of the analytical solution for $g$. For coarser meshes, there is an increase in the rank to account for the larger time step sizes.

\begin{table}[!ht]
    \footnotesize
    \caption{Refinement in the density $\rho$ against the manufactured solution \eqref{eq:manufactured solution rho} for the smooth test problem using an unsplit dimension tree for $g$. The rank of $g$ at the final time is included as well.}
    \label{tab:manufactured test IMEX-3 unsplit}
    \begin{center}
    \begin{tabular}{| c || c || c || c || c |}
        \hline
        $\epsilon$ & $N$ & $L^{1}$ Error & $L^{1}$ Order & $\text{Rank}(g)$ \\
        \hline
                           & $16$  & $9.4977 \times 10^{-5}$  & -        & $[1,2,2,2,1]$ \\
         $1$               & $32$  & $9.8467 \times 10^{-7}$  & $6.5918$ & $[1,1,1,1,1]$ \\
                           & $64$  & $9.6960 \times 10^{-9}$  & $6.6661$ & $[1,1,1,1,1]$ \\
                           & $128$ & $8.0977 \times 10^{-11}$ & $6.9037$ & $[1,1,1,1,1]$ \\
        \hline
                           & $16$  & $7.3371 \times 10^{-5}$ & -        & $[1,8,8,8,2]$ \\
         $10^{-2}$         & $32$  & $2.8663 \times 10^{-6}$ & $4.6780$ & $[1,1,1,1,1]$ \\
                           & $64$  & $7.7766 \times 10^{-8}$ & $5.2039$ & $[1,1,1,1,1]$ \\
                           & $128$ & $4.0684 \times 10^{-9}$ & $4.2566$ & $[1,1,1,1,1]$ \\
        \hline
                           & $16$  & $9.4011 \times 10^{-5}$ & -        & $[1,4,4,4,1]$ \\
         $10^{-6}$         & $32$  & $9.2547 \times 10^{-7}$ & $6.6665$ & $[1,2,2,2,1]$ \\
                           & $64$  & $5.9271 \times 10^{-8}$ & $3.9648$ & $[1,1,1,1,1]$ \\
                           & $128$ & $4.7785 \times 10^{-8}$ & $0.3108$ & $[1,1,1,1,1]$ \\
        \hline
    \end{tabular}
    \end{center}
\end{table}

\begin{table}[!ht]
    \footnotesize
    \caption{Refinement in the density $\rho$ against the manufactured solution \eqref{eq:manufactured solution rho} for the smooth test problem using a split dimension tree for both $\rho$ and $g$. The rank of $g$ at the final time is included as well.}
    \label{tab:manufactured test IMEX-3 split}
    \begin{center}
    \begin{tabular}{| c || c || c || c || c |}
        \hline
        $\epsilon$ & $N$ & $L^{1}$ Error & $L^{1}$ Order & $\text{Rank}(g)$ \\
        \hline
                           & $16$  & $8.2941 \times 10^{-5}$  & -        & $[1,2,2,2,2,2,1]$ \\
         $1$               & $32$  & $2.3068 \times 10^{-6}$  & $5.1681$ & $[1,1,1,1,1,1,1]$ \\
                           & $64$  & $4.5972 \times 10^{-8}$  & $5.6490$ & $[1,1,1,1,1,1,1]$ \\
                           & $128$ & $9.7063 \times 10^{-10}$ & $5.5657$ & $[1,1,1,1,1,1,1]$ \\
        \hline
                           & $16$  & $7.6584 \times 10^{-5}$ & -        & $[1,8,8,6,6,8,2]$ \\
         $10^{-2}$         & $32$  & $1.6504 \times 10^{-6}$ & $5.5362$ & $[1,4,4,3,2,4,1]$ \\
                           & $64$  & $5.7782 \times 10^{-8}$ & $4.8361$ & $[1,1,1,1,1,1,1]$ \\
                           & $128$ & $3.4576 \times 10^{-9}$ & $4.0628$ & $[1,1,1,1,1,1,1]$ \\
        \hline
                           & $16$  & $9.7004 \times 10^{-5}$ & -        & $[1,4,4,6,5,4,1]$ \\
         $10^{-6}$         & $32$  & $7.1908 \times 10^{-7}$ & $7.0757$ & $[1,4,4,3,2,4,1]$ \\
                           & $64$  & $2.2573 \times 10^{-7}$ & $1.6716$ & $[1,1,1,1,1,2,1]$ \\
                           & $128$ & $4.7789 \times 10^{-8}$ & $2.2399$ & $[1,1,1,1,1,1,1]$ \\
        \hline
    \end{tabular}
    \end{center}
\end{table}

\Cref{fig:manufactured test IMEX-3 timing comparison} presents the results of a scaling study for the proposed low-rank methods. This test problem is particularly well-suited for such a study, as the rank remains small and nearly fixed throughout the simulation. This allows for isolating the scaling behavior of the algorithm with respect to mesh resolution, which is expected to have the most significant impact on computational time. The setup is a slight variation of the refinement experiment mentioned earlier, and we fix $\epsilon = 1$. We let $N$ denote the number of spatial mesh points per dimension, and progressively double $N$ from 32 to 4096 in the experiment. We also refine the mesh in the angular domain by pairing each case with an $S_{N/8}$ angular discretization. For each configuration, we apply the method for 10 steps and calculate the average wall time. As discussed in \cref{subsec:HTT}, if $r \ll N$, we expect the computational cost of the unsplit and split approaches to scale as $\mathcal{O}(N^{2})$ and $\mathcal{O}(N)$, respectively. Our timing results for the unsplit case align well with this prediction. However, the split case demonstrates constant complexity, which is faster than anticipated, even for highly refined meshes. This behavior seems to be an artifact of the MATLAB implementation, where interpreter overhead (independent of $N$) dominates the $\mathcal{O}(N)$ components. These components are highly vectorized and contribute only a small fraction to the overall simulation time.

\begin{figure}[!hbt]
    \centering
    \includegraphics[clip, trim={0cm, 0cm, 0cm, 0cm}, scale=0.175]{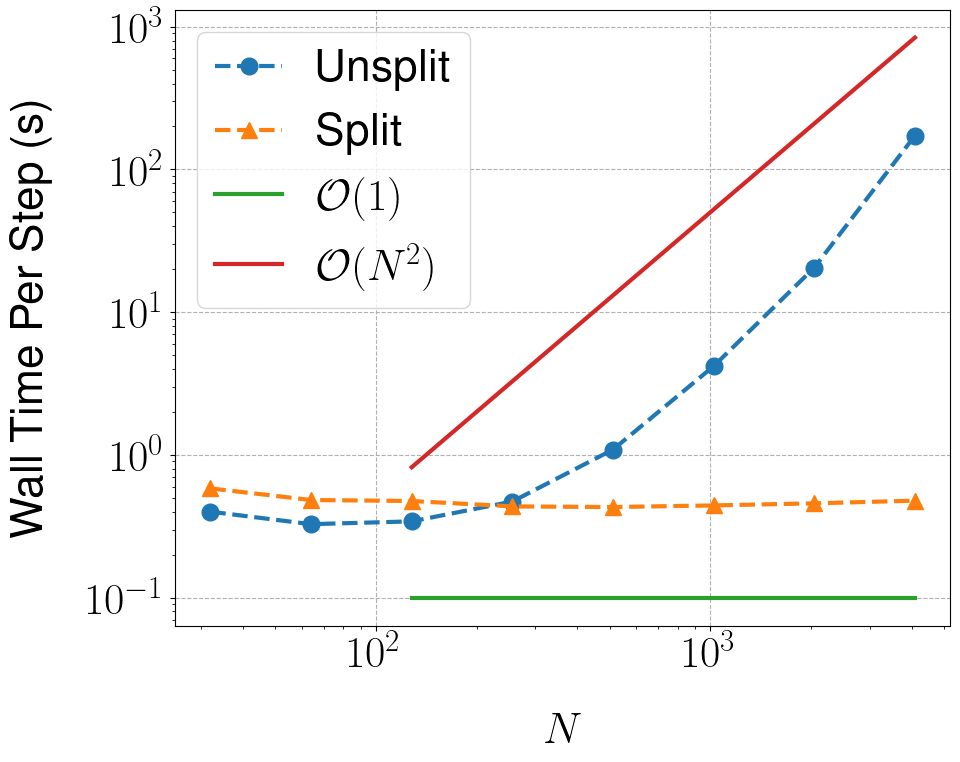}
    \caption{Wall time per timestep versus spatial mesh resolution for the unsplit and split methods.}
    \label{fig:manufactured test IMEX-3 timing comparison}
\end{figure}


\subsection{Variable Scattering Cross-section}
\label{subsec:variable scattering}

For the second test case, we consider a strongly varying scattering cross-section, following the setup given in \cite{EinkemmerDLR-AP}. The goal with this experiment is to show the effect of the variable material data on the rank of the tensors. In this example, $\epsilon = 10^{-2}$, and the scattering cross-section is given by
\begin{equation}
    \label{eq:variable scattering coeff}
    \sigma_{s}(x,y) =
    \begin{cases}
        0.999 c^{4} \left(c + \sqrt{2}\right)^{2} \left(c - \sqrt{2}\right)^{2} + 0.001, \quad c = \sqrt{x^{2} + y^{2}} < 1, \\
        1, \quad \text{otherwise}.
    \end{cases}
\end{equation}

\begin{figure}
    \centering
    \subfloat[$\sigma_{s}(\mathbf{x})$]{\includegraphics[clip, trim={0.0cm, 0cm, 0cm, 0cm}, scale=0.22]{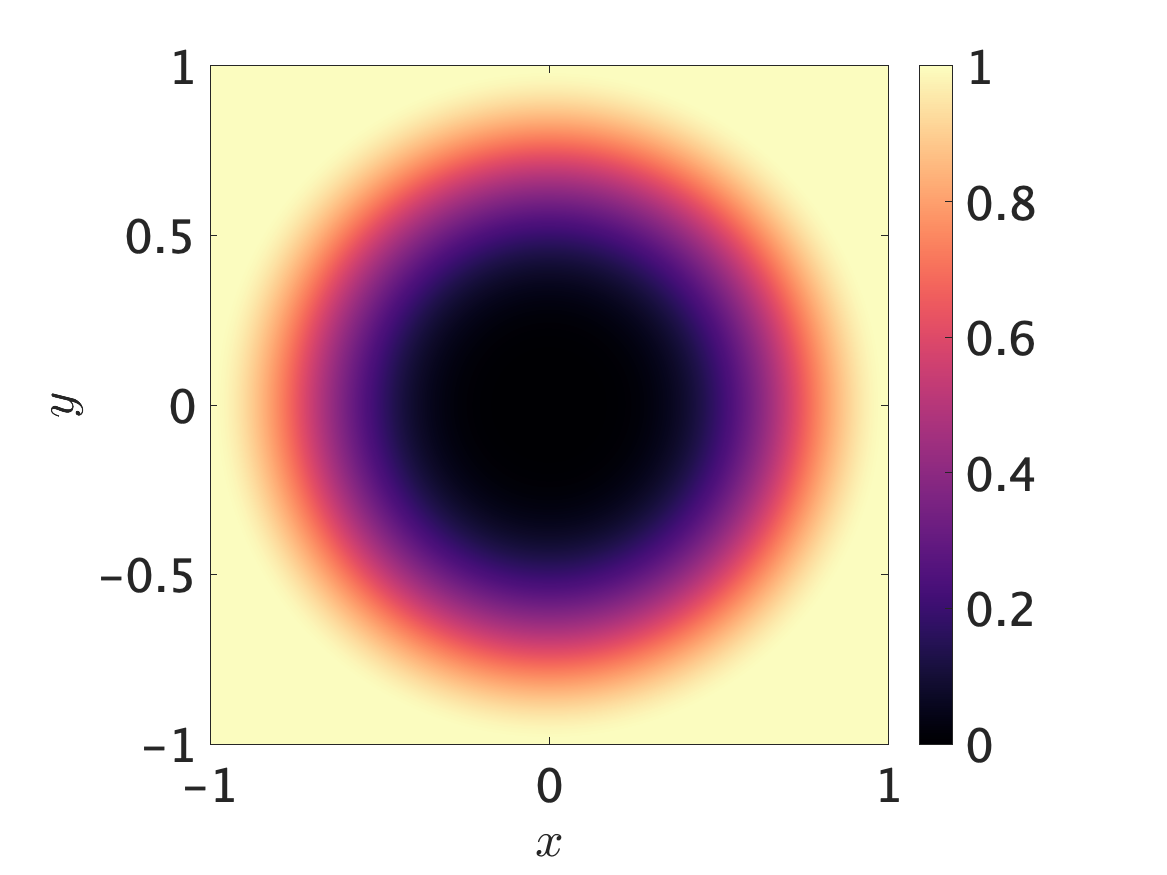}
    \label{fig:variable scattering cross-section}
    }
    \subfloat[Reference]{\includegraphics[clip, trim={0cm, 0cm, 0cm, 0cm}, scale=0.22]{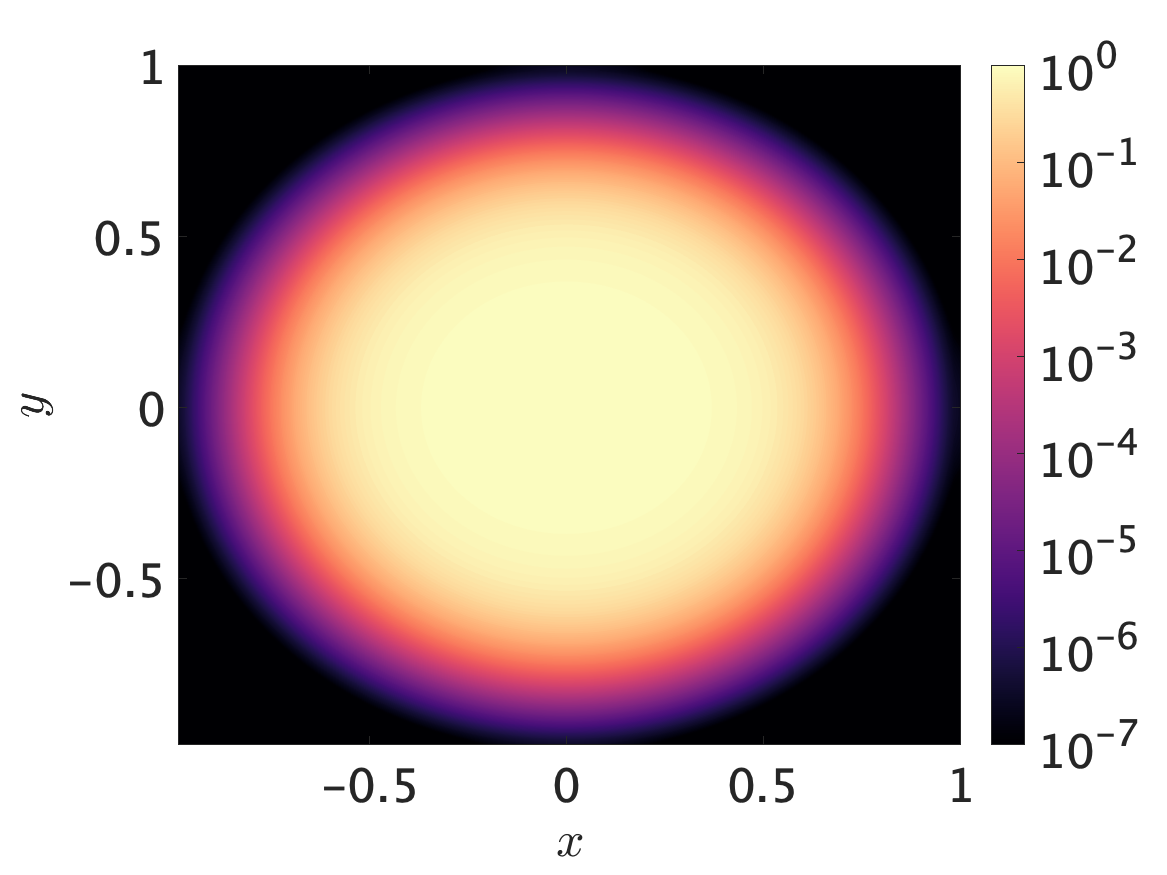}
    \label{fig:variable rho reference}
    }
    \caption{Material cross-section and the reference scalar density $\rho$ at the final time $T = 0.01$ computed using a high-order full-grid scheme for the variable scattering cross-section example.}
    \label{fig:variable rho cross-section and reference}
\end{figure}

The setup for this problem is a spatial domain given by $[-1,1]^{2}$, and we use an $S_{64}$ quadrature rule on $\mathbb{S}^{2}$. Initially, the distribution is isotropic, so we set
\begin{equation*}
    \rho(x,y,0) = \frac{1}{4\pi \zeta^{2}} \exp \left(- \frac{x^{2} + y^{2}}{4\zeta^{2}} \right), \quad g(x,y,\theta, \mu,0) = 0.
\end{equation*}
The final time used for the simulation was $T = 0.01$, and we set the time step as $\Delta t = 0.1 \epsilon \Delta x := 0.001 \Delta x.$
A plot of the material cross-section \eqref{eq:variable scattering coeff} as well as a plot of the high-order reference scalar density $\rho$ is shown in \cref{fig:variable rho cross-section and reference}. From equation \eqref{eq: IMEX g stage}, we can see that the term $\sigma_{s}/\epsilon^{2}$ varies significantly across the domain and spans both the collisional and free-streaming regimes. The reference solution was obtained using a $128^{2}$ spatial mesh and an $S_{64}$ angular discretization. The particles gradually expand outward from the center of the domain and transition from a free-streaming regime to one that is strongly collisional.

\Cref{fig:variable lr vs fr comparison} compares the density obtained with the proposed low-rank methods against the full-grid reference density $\rho$ presented in \cref{fig:variable rho reference}. The same $128^{2}$ spatial mesh and $S_{64}$ angular discretization was used in both of the low-rank approaches. Each plot shows the point-wise absolute difference on a logarithmic scale against a full-grid reference solution, and we use the same limits on the plots to permit a more direct comparison. It is clear that the unsplit method produces a more accurate solution than the split approach. In both cases, the most significant errors are concentrated near the center of the domain, where the density is the largest. Unlike the unsplit approach, we find that the errors in the split extend outward to the boundaries. 

\begin{figure}
    \centering
    \subfloat[High-order unsplit]{
        \includegraphics[clip, trim={0cm, 0cm, 3.1cm, 0cm}, scale=0.22]{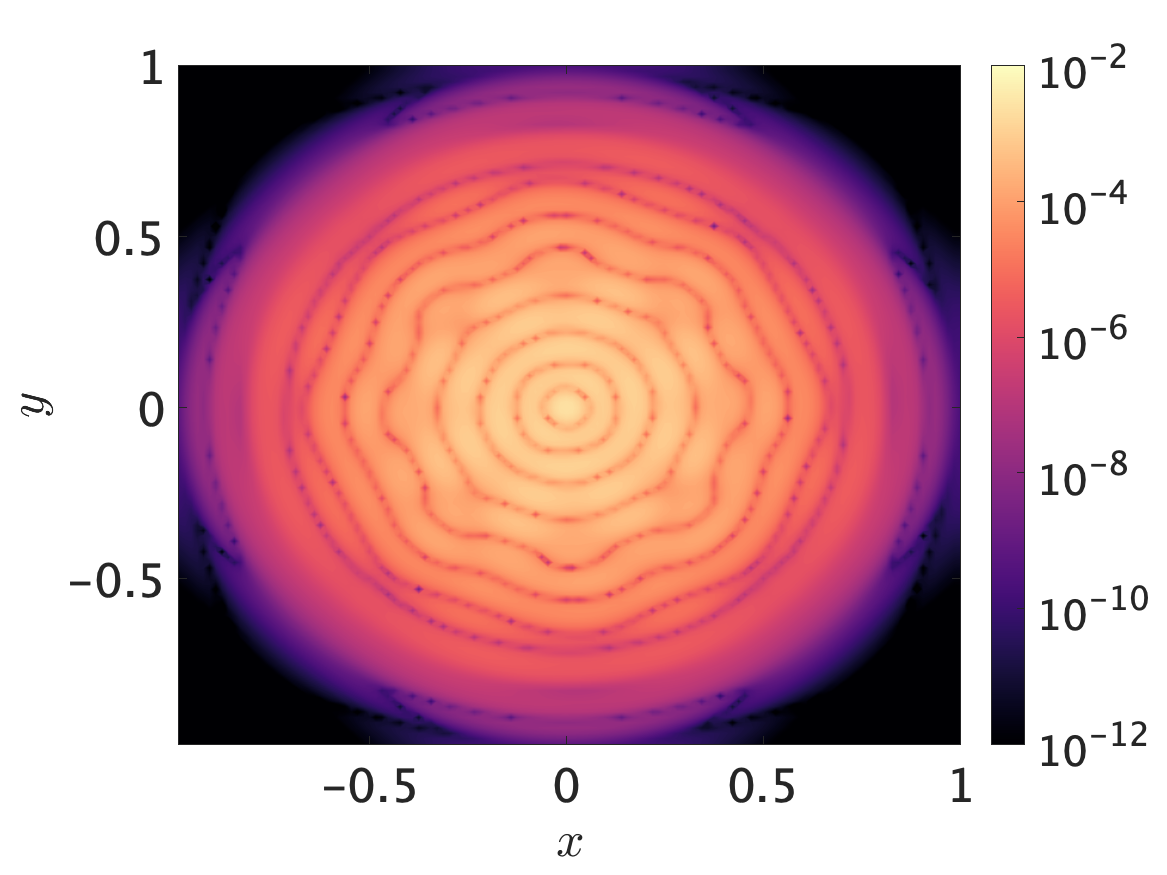}
        \label{fig:variable rho lr vs fr unsplit with proj}
    }
    \subfloat[High-order split]{
        \includegraphics[clip, trim={0cm, 0cm, 0cm, 0cm}, scale=0.22]{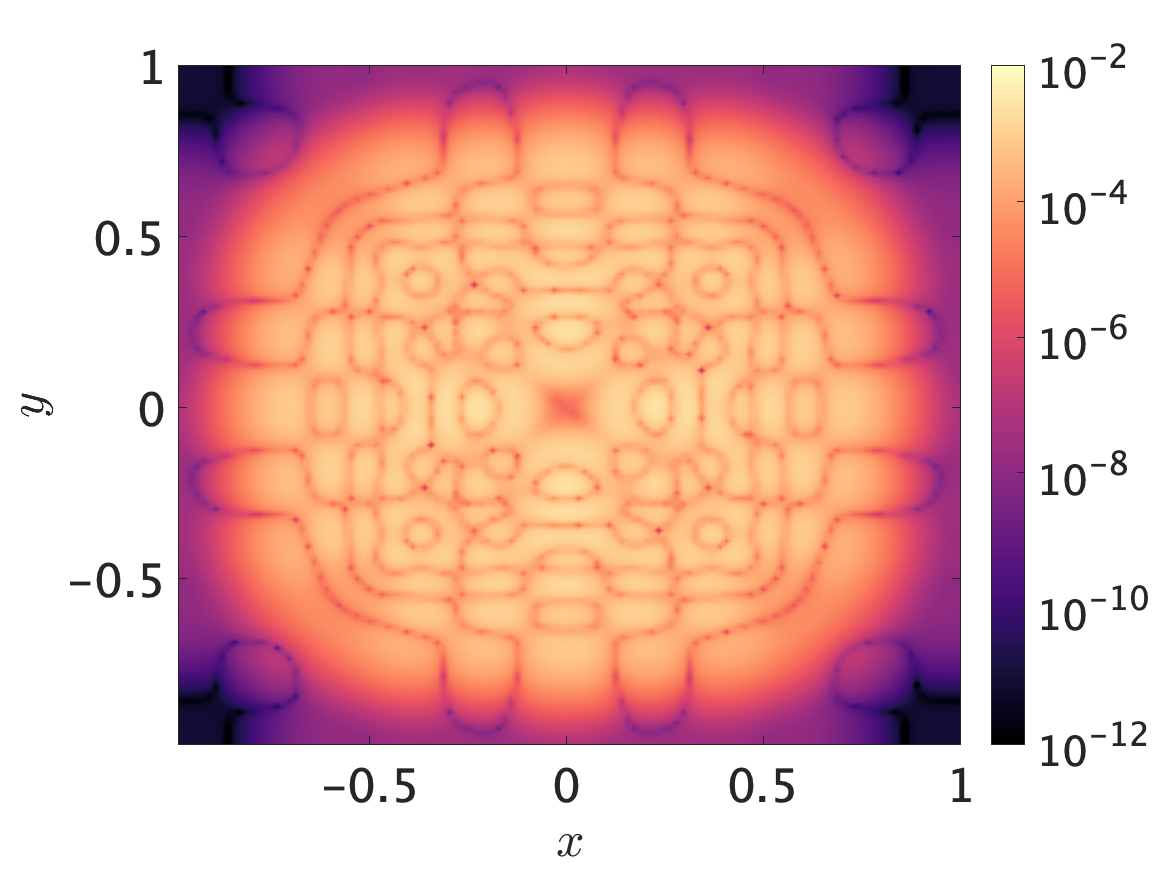}
        \label{fig:variable rho lr vs fr split with proj}
    }
    \caption{Point-wise absolute differences between the densities obtained with low-rank methods for the variable scattering cross-section example.}
    \label{fig:variable lr vs fr comparison}
\end{figure}

\Cref{fig:variable problem g svd data} compares the structure of the relative singular values associated with the matricizations at each of the nodes in the dimension trees for both the unsplit and split representations. Both results were obtained using a $128^{2}$ spatial mesh and an $S_{64}$ angular discretization. The hierarchical rank of $g$ recorded at the final step was $[1,88,88,23,8]$ for the unsplit case and $[1,87,87,16,16,23,9]$ for the split case. In each of the plot windows, the vertical axis represents the relative singular values ranging from $10^{-4}$ to $1$, and the ticks correspond to logarithmic spacing in one order of magnitude. The horizontal axis uses linear spacing of the ranks from $1$ to $90$ in increments of $10$. Since this problem is restricted to the $xy-$plane, we expect symmetries to be present in the dimension $\mu$, which represents the cosine of the polar angle (dimension 3 in \cref{fig:variable svd g unsplit} and dimension 4 in \cref{fig:variable svd g split}, respectively). Both methods capture this feature, as the singular values decay rapidly beyond a rank of $9$. This suggests a benefit to the tensorized angular discretization in the low-rank method. In the split approach, we note that the relative singular values of $g$ associated with the bases at the leaf nodes $\{x\}$ and $\{y\}$ (dimensions 1 and 2, respectively, in \cref{fig:variable svd g split}) decay rapidly beyond rank 16. 

\begin{figure}
    \centering
    \subfloat[Unsplit $g$]{
        \includegraphics[clip, trim={0cm, 0cm, 0cm, 0cm}, scale=0.3]{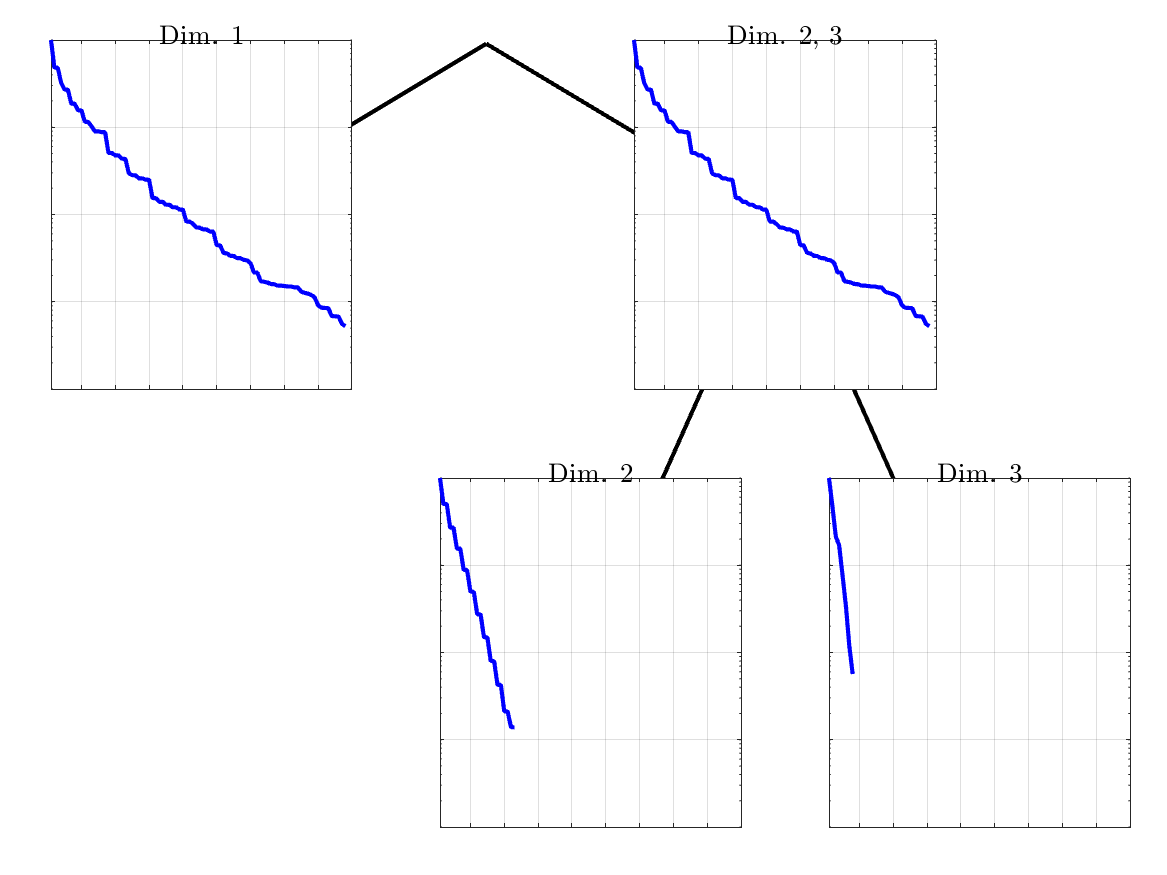}
        \label{fig:variable svd g unsplit}
    }
    \subfloat[Split $g$]{
        \includegraphics[clip, trim={0cm, 0cm, 0cm, 0cm}, scale=0.3]{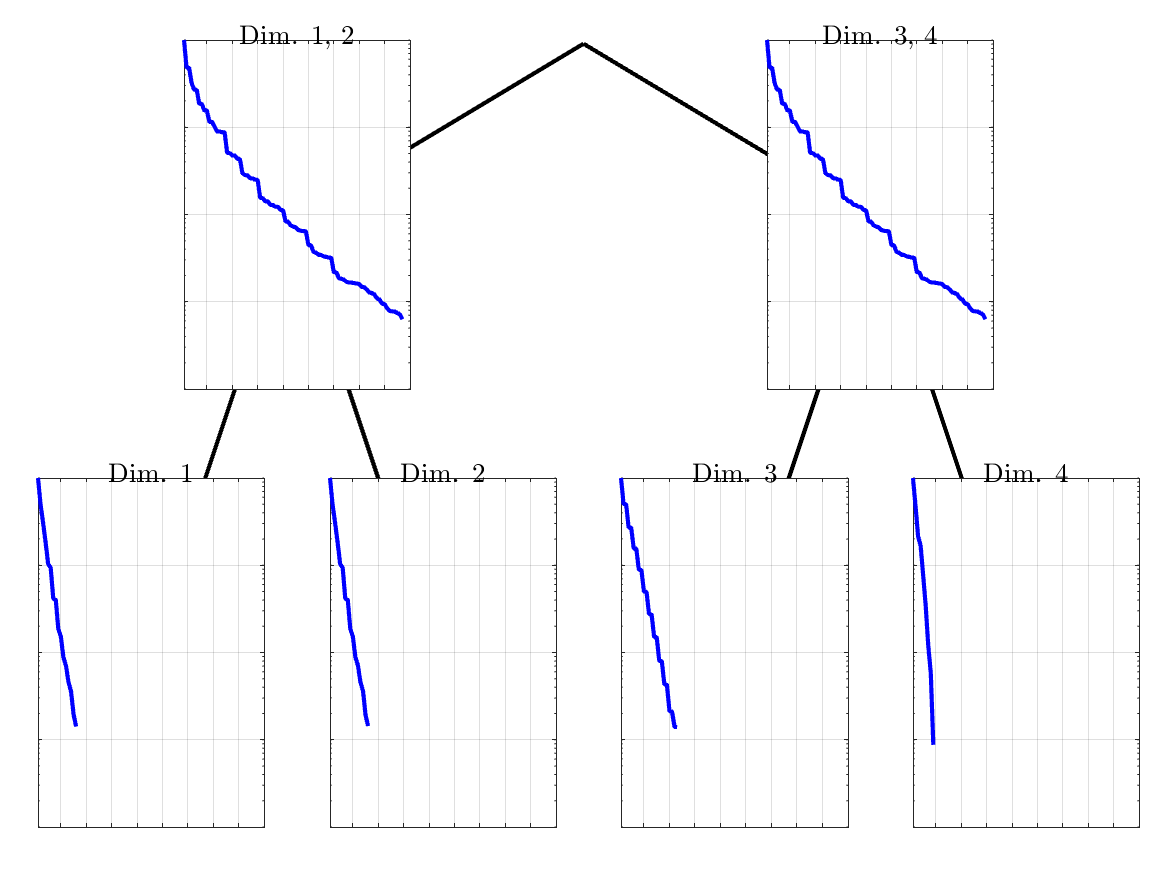}
        \label{fig:variable svd g split}
    }
    \caption{Plots of the relative singular values versus rank for the matricizations at each of the nodes in the dimension tree for $g$ at the final time $T = 0.01$ using unsplit and split dimension trees in the variable scattering cross-section example.}
    \label{fig:variable problem g svd data}.
\end{figure}

The time evolution of the hierarchical ranks for the tensors used to represent the microscopic component $g$ are shown in \cref{fig:variable problem g rank vs time}. We compare the unsplit and split representations using three different spatial meshes whose sizes are $32^{2}$, $64^{2}$, and $128^{2}$. While there are similarities in the growth patterns of the dimensions, both the unsplit and split approaches show a slight decrease in the hierarchical ranks for nodes associated with the spatial dimensions as the mesh resolution increases. This suggests that low-rank structures become increasingly prevalent in proportion to the mesh resolution. In both approaches, the ranks for nodes $\{x,y\}$ and $\{\theta,\mu\}$ grow the fastest, while those associated with the singleton dimensions $\{\theta\}$ and $\{\mu\}$ grow slowly with time and remain relatively small throughout the simulation. The rank of the singleton node $\{\mu\}$ remains small in each case due to the symmetry in that direction, which shows the benefit of the tensorized low-rank discretization. In problems with free-streaming and collisional features, it is worth pointing out that resolving the free-streaming component requires many ordinates, while regions of high collisionality reguire much less. The ability to use a tensorized angular domain in problems with low-rank structures means that many ordinates can be used with only a slight increase in memory cost. Similar conclusions were found in other low-rank methods, e.g., \cite{PengDLR2023discrete-ordinates}, where it was shown that ray effects could be eliminated by adopting more accurate angular discretizations at a slight cost in total memory. Aspects concerning the mitigation of ray-effects in kinetic problems with the proposed high-order methods is something we plan to explore in the future.

\begin{figure}[!htb]
    \subfloat[High-order unsplit]{
        \centering
        \includegraphics[clip, trim={0cm, 0cm, 0cm, 0cm}, scale=0.19]{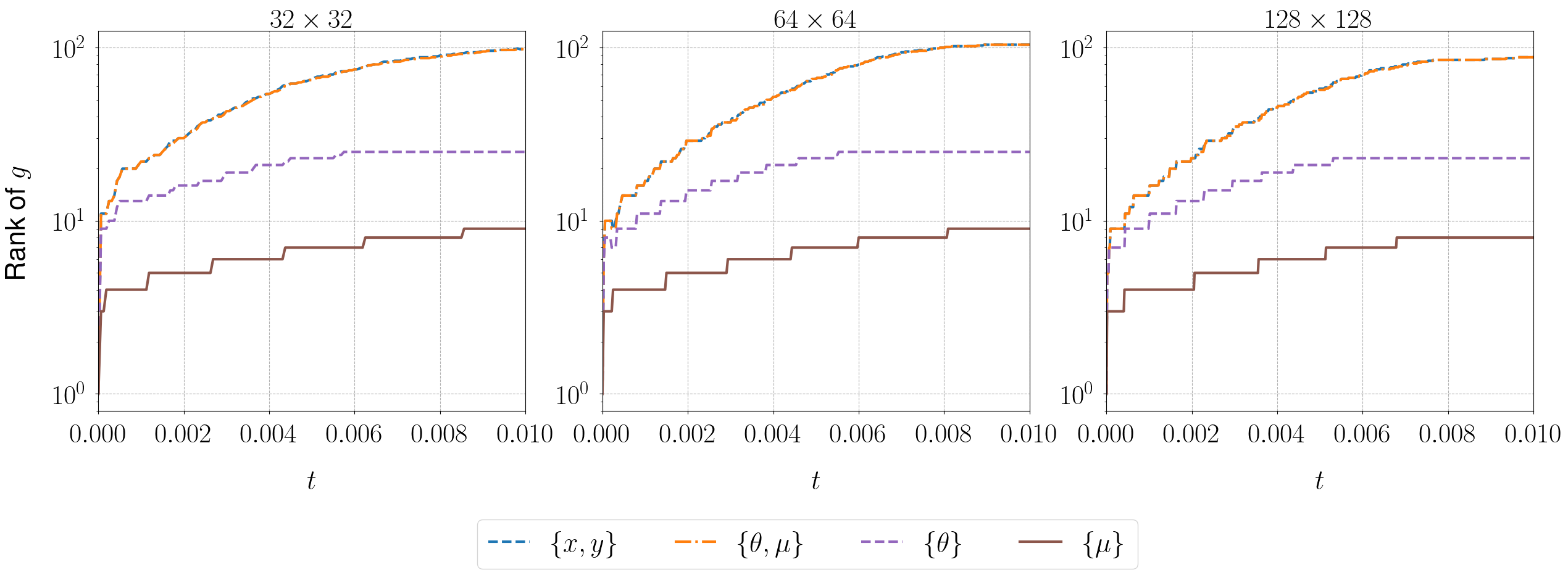}
    }
    \\
    \subfloat[High-order split]{
        \centering
        \includegraphics[clip, trim={0cm, 0cm, 0cm, 0cm}, scale=0.19]{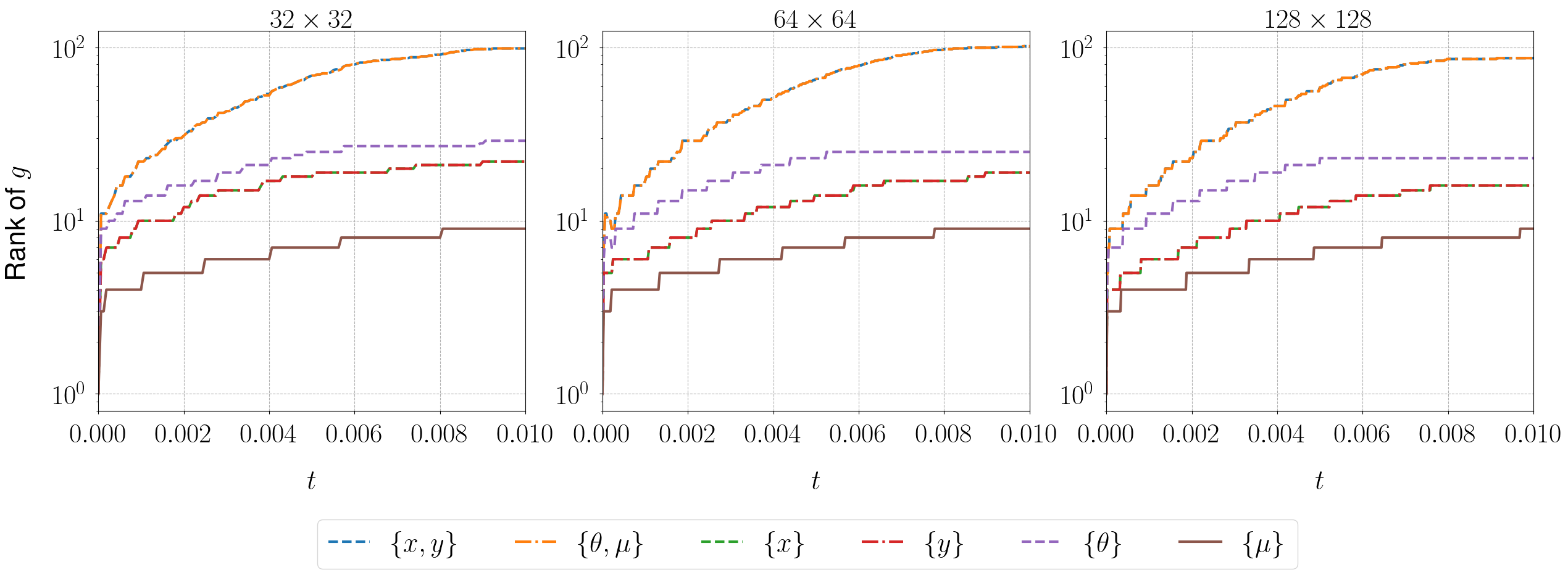}
    }
    \caption{Time evolution of hierarchical ranks in the variable scattering cross-section example with an unsplit and split representations of the tensor $g$.}
    \label{fig:variable problem g rank vs time}
\end{figure}

Plots showing the total DOF and the compression ratio defined by equation \eqref{fig:variable DOF comparison} for the proposed low-rank methods are provided in \cref{fig:variable DOF comparison}. We observe excellent compression ratios among both methods relative to the full-grid approach. In the unsplit approach, the compression ratio for different meshes is similar, since a full-grid representation is used in the spatial dimensions. For the split approach, the compression significantly improves as the mesh is refined, a characteristic that reflects both the CoD for the full-grid solution and the low-rank structure of the numerical solution. In the case with $N = 128$, we can see that the unsplit low-rank approach uses $\approx 1.09\%$ of the storage for the full-grid approach. The improvements are far more significant for the split approach, which uses $\approx 0.04\%$ of the storage for the full-grid solution. This difference in the compression results from a reduction in the total DOF by a factor of $\approx 26.48$ in the split approach versus the analogous unsplit approach, at the cost of an increase in error. The use of projections produces a comparable number of total DOF, irrespective of the resolution, when compared against the same method without the projection. This indicates that the use of the projections does not increase the total DOF and the memory complexity in a significant way.

Lastly, we include plots to confirm the conservation properties of the proposed low-rank methods in \cref{fig:Variable mass data comparison}. When the projection for $g$ is not used, the total mass for $g$ grows to a size which is $\mathcal{O}(1)$, which improves, only slightly, as the mesh is refined. The projection ensures that the total mass contributed by $g$ remains small throughout the simulation. In terms of overall mass conservation, we find both the split and unsplit approaches are effective. We remark that the unsplit approach does not apply truncation to $\rho$, so the total mass is conserved on the order of machine precision, so we do not show this result. The split approach, which applies truncation to $\rho$, shows more noticeable violations in total mass conservation, if the projection for $\rho$ discussed in \cref{subsubsec:mass projection} is not used.

\begin{figure}[!ht]
    \centering
    \subfloat[High-order unsplit]{
        \centering
        \includegraphics[clip, trim={0cm, 0cm, 0cm, 0cm}, scale=0.2]{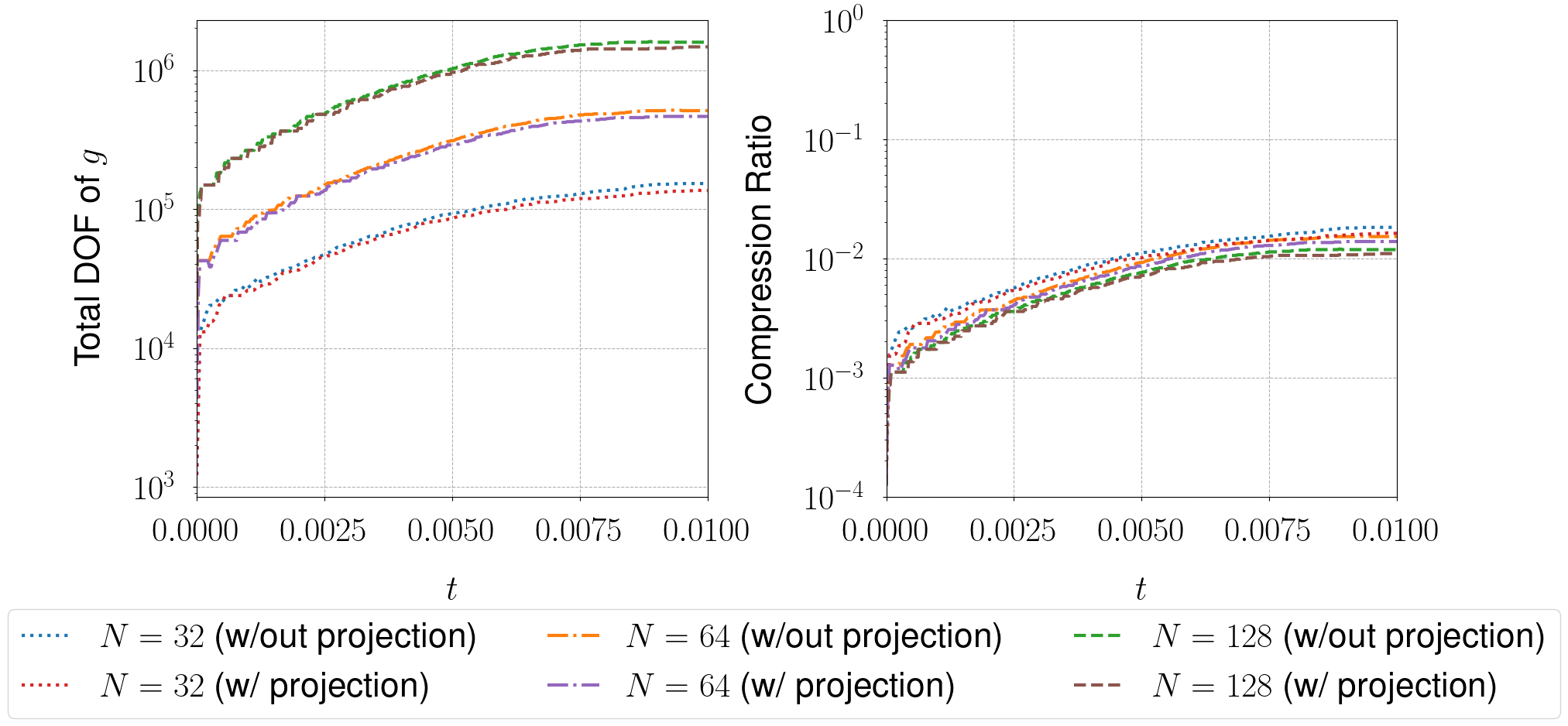}
    }
    \\
    \subfloat[High-order split]{
        \centering
        \includegraphics[clip, trim={0cm, 0cm, 0cm, 0cm}, scale=0.2]{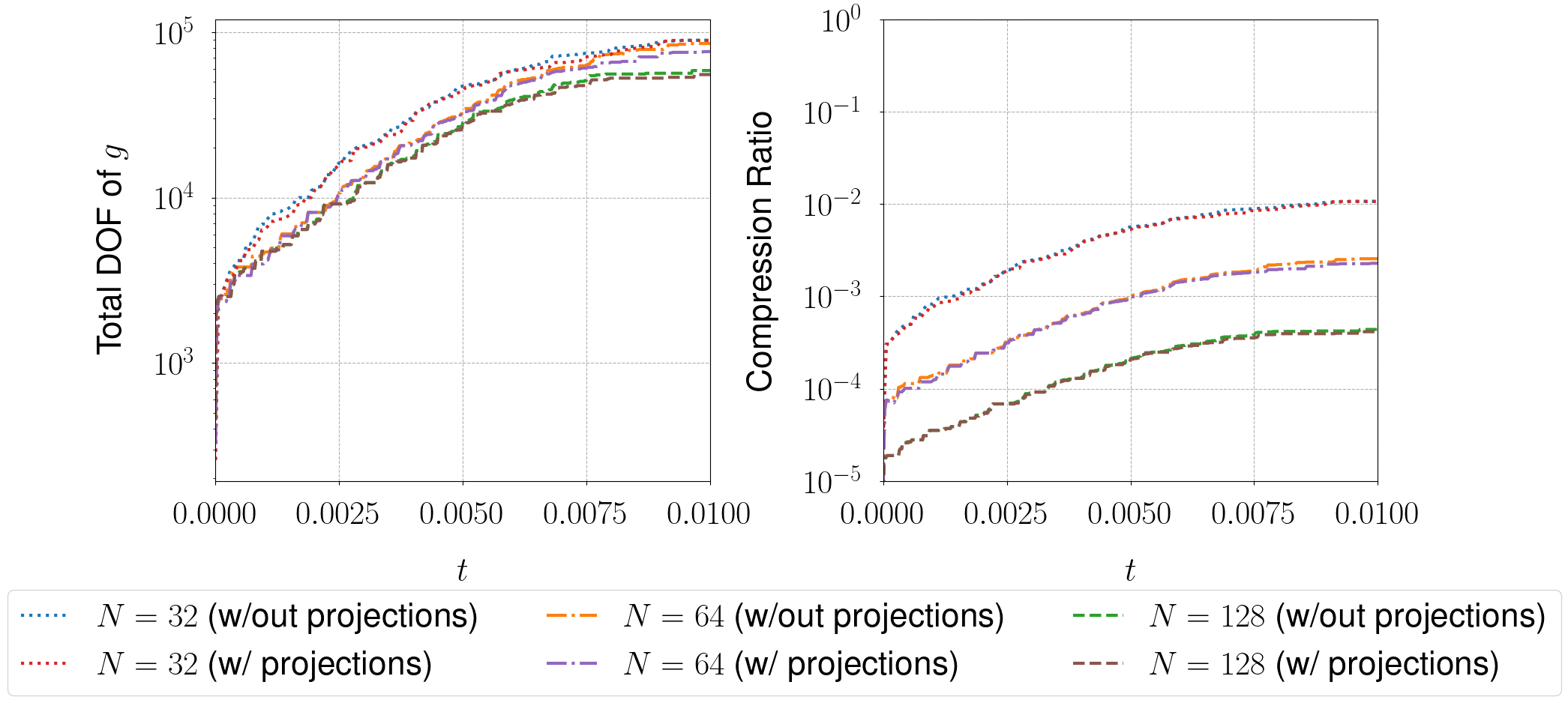}
    }
    \caption{Plots of the total DOF and compression ratio as functions of time for the variable scattering cross-section example.}
    \label{fig:variable DOF comparison}
\end{figure}

\begin{figure}[!ht]
    \centering
    \subfloat[High-order unsplit]{
        \centering
        \includegraphics[clip, trim={0cm, 0cm, 16.5cm, 0cm}, scale=0.18]{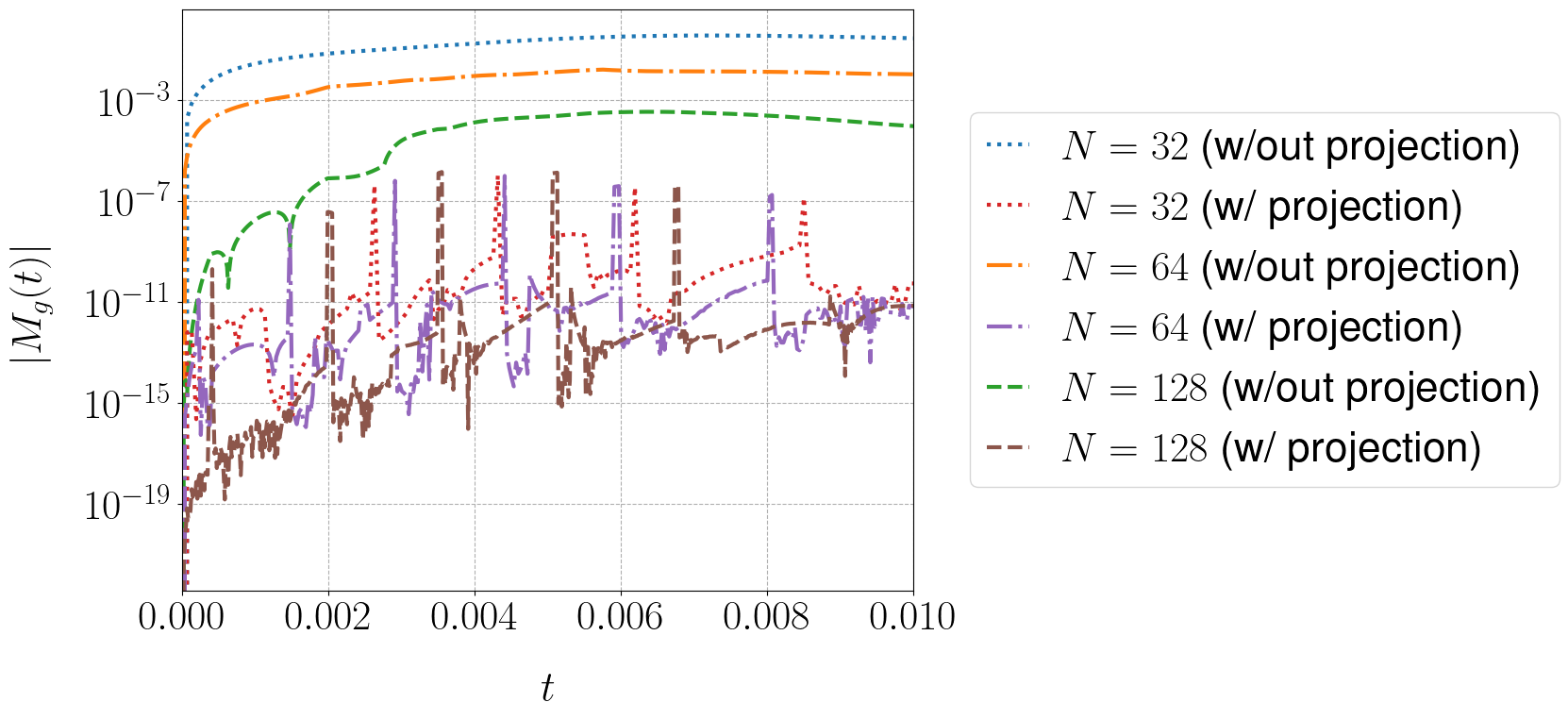}
    }
    \subfloat[High-order split]{
        \centering
        \includegraphics[clip, trim={2cm, 0cm, 0cm, 0cm}, scale=0.18]{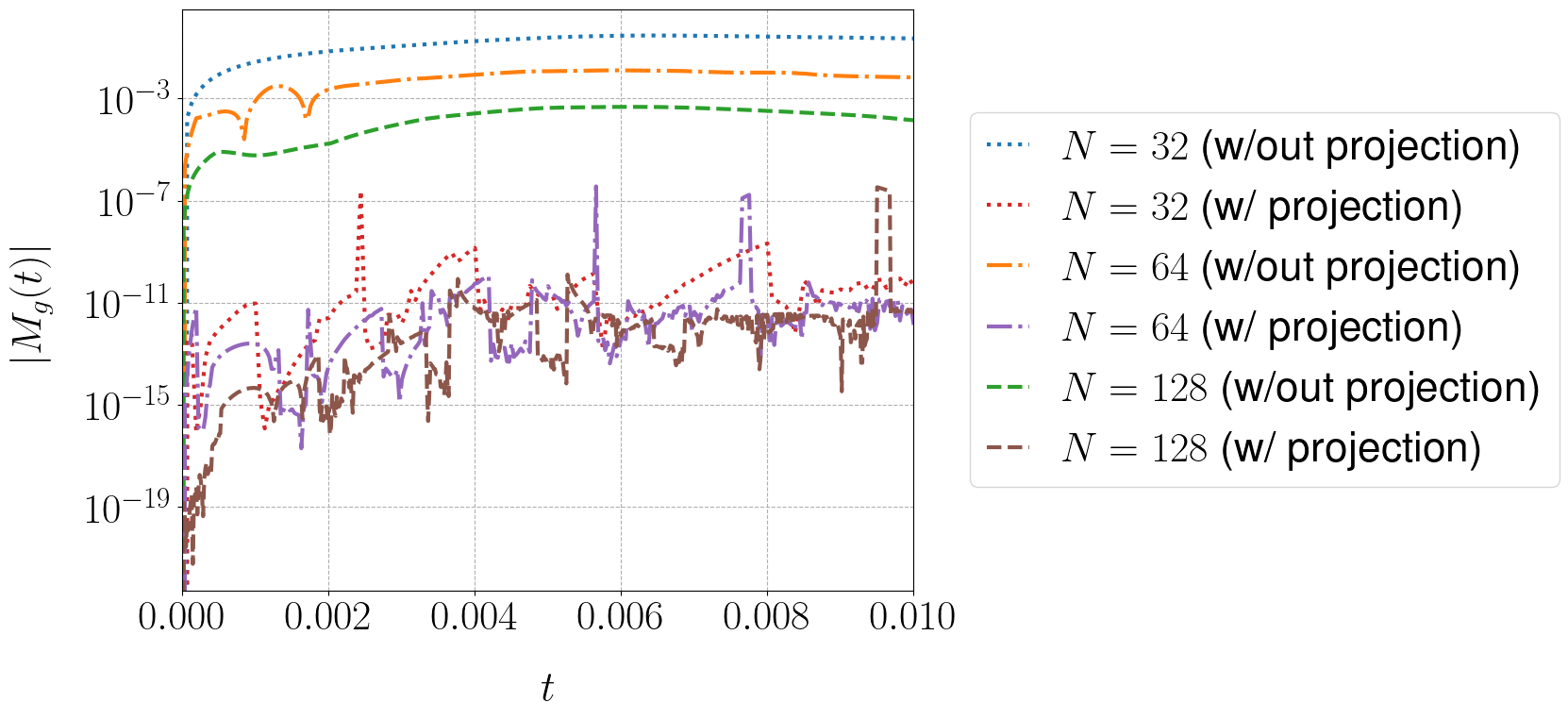}
    }
    \caption{Absolute mass of the microscopic component $g$ for the variable scattering cross-section example.}
    \label{fig:Variable mass data comparison}
\end{figure}


\subsection{Lattice Problem}
\label{subsec:lattice problem}

The next example we consider is a modification of the lattice problem \cite{Brunner2002,Brunner2005riemann}, which is used as a model of a simplified fuel rod assembly in a nuclear reactor. The spatial domain for this problem is the square $[0,7]^{2}$ which is divided into smaller squares of unit area. A source is placed in the center of the domain and absorbing regions are placed around it in a configuration that resembles a checkerboard. The remaining parts of the domain are treated as purely scattering regions. This is regarded as a challenging example due to the discontinuities featured in the material data, which abruptly alternate between optically thin and thick. While it is more conventional to use outflow boundary conditions, is it also possible to adopt periodic boundary conditions, which is the approach taken here.

The profiles of the absorption and scattering cross-sections, along with the source function, are provided in \cref{fig:lattice material data}. We initialize the macroscopic and microscopic variables as
\begin{equation*}
    \rho(x,y,0) = \frac{1}{4\pi \zeta^{2}} \exp \left(- \frac{\left( x - 3.5\right)^{2} + \left( y - 3.5\right)^{2}}{4\zeta^{2}} \right), \quad g(x,y,\theta, \mu,0) = 0,
\end{equation*}
and take $\zeta = 0.1$. We test the method in the kinetic regime by taking $\epsilon = 1$. The angular domain is discretized using an $S_{64}$ angular quadrature rule on $\mathbb{S}^{2}$. We run the simulation to a final time of $T = 2.0$ and set the time step as $\Delta t = 0.1\Delta x.$

\begin{figure}[t]
    \centering
    \subfloat[$\sigma_{a}(\mathbf{x})$]{
        \includegraphics[clip, trim={2.0cm, 0cm, 2.0cm, 0cm}, scale=0.22]{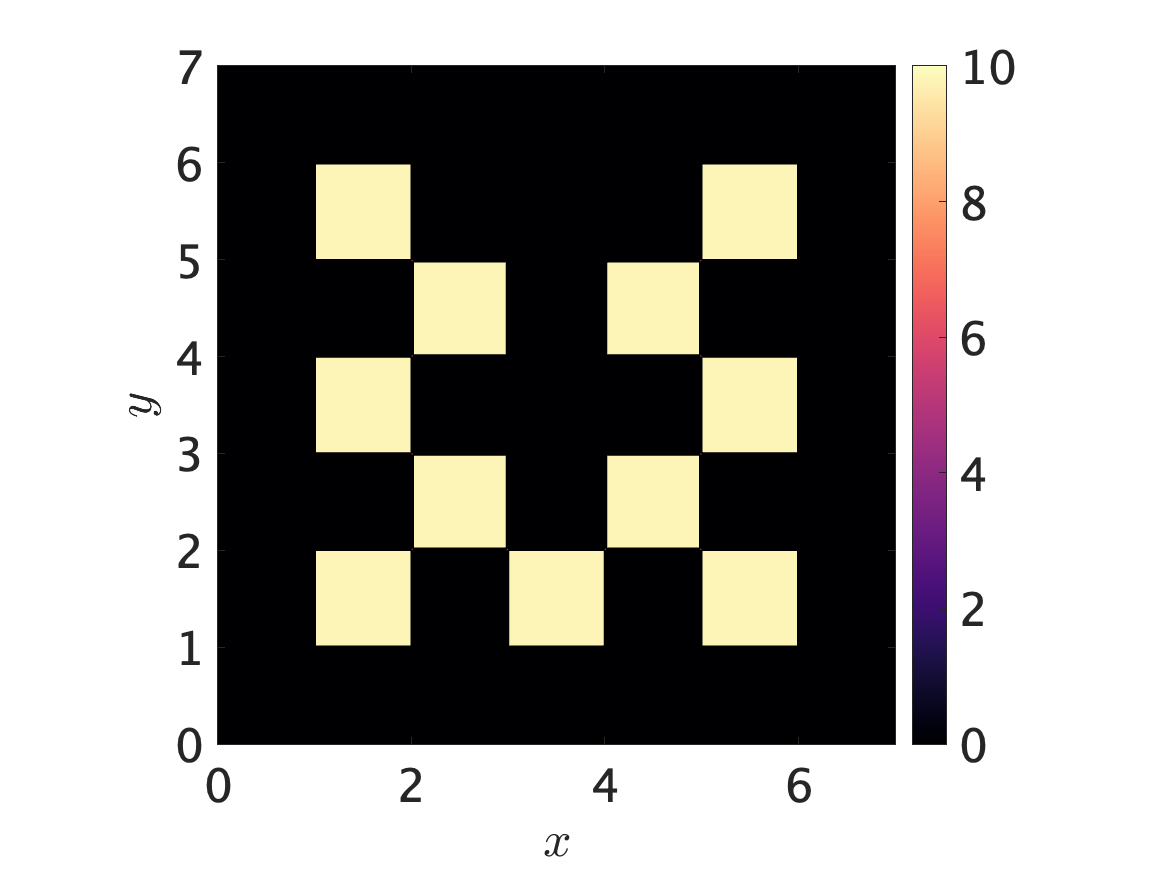}
        \label{fig:lattice sigma_a}
    }
    \subfloat[$\sigma_{s}(\mathbf{x})$]{
        \includegraphics[clip, trim={2.0cm, 0cm, 2.0cm, 0cm}, scale=0.22]{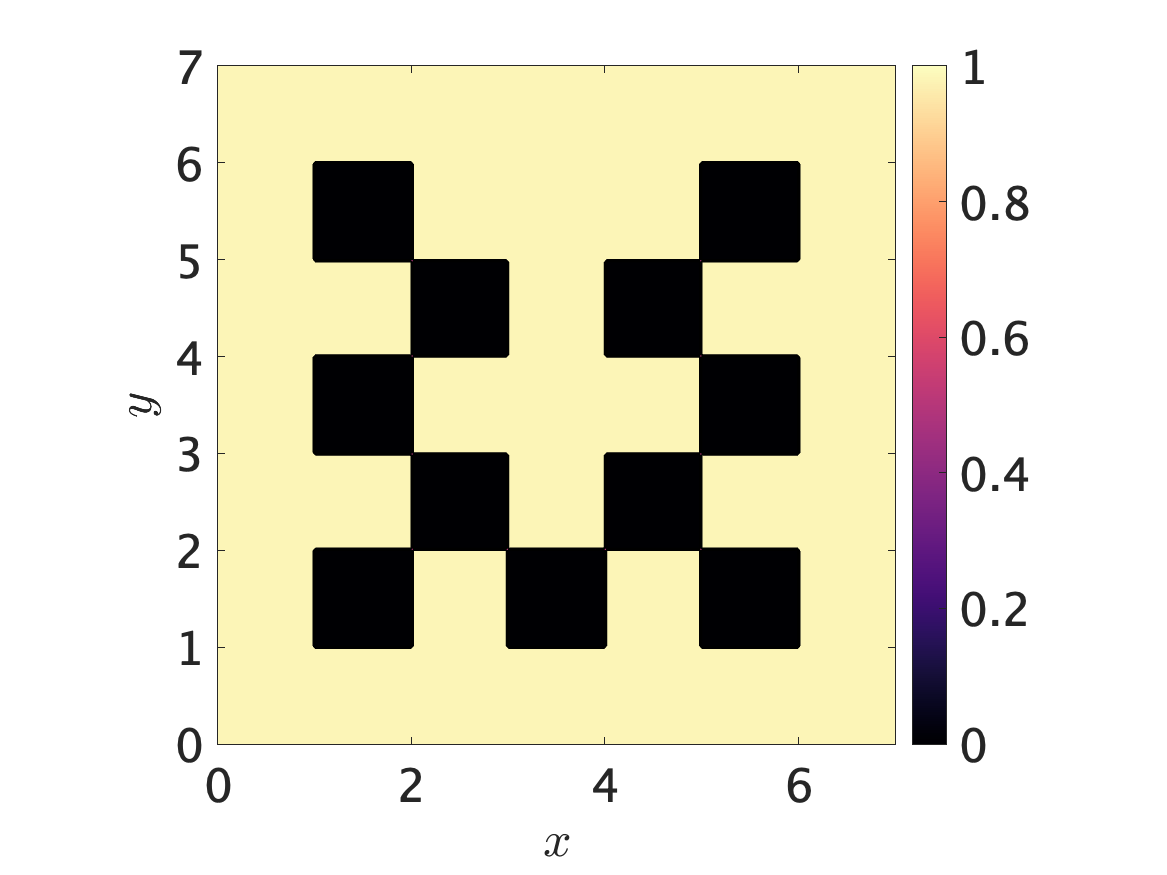}
        \label{fig:lattice sigma_s}
    }
    \subfloat[$Q(\mathbf{x})$]{
        \includegraphics[clip, trim={2.0cm, 0cm, 2.0cm, 0cm}, scale=0.22]{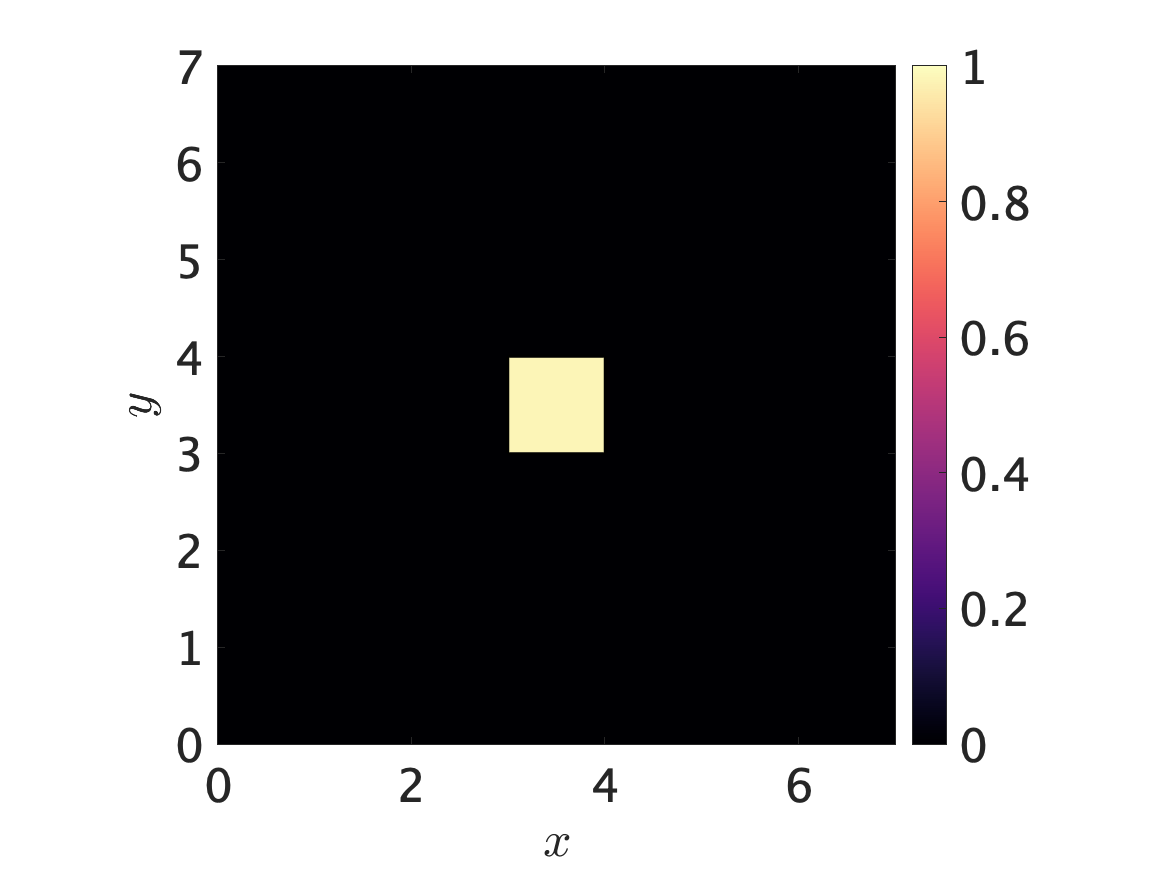}
        \label{fig:lattice source}
    }
    \caption{Material data used in the lattice example.}
    \label{fig:lattice material data}
\end{figure}

We consider the impact of a low-order scheme on low-rank structures by comparing against the high-order scheme adopted in this work. Low-order discretizations are known to be more dissipative than high-order methods, which is an important consideration in the development of low-rank methods, as it can lead to the introduction of artificial low-rank structures. Two different pairings of discretizations for time and space are considered. The high-order schemes use the time and space discretizations discussed at the beginning of \cref{sec:results}. For the low-order discretization, we combine the first-order IMEX method with a second-order, piecewise linear method for space using the MUSCL limiter \cite{vanLeer2006upwind}. Sample plots of the density $\rho$ obtained with different low-rank methods obtained with a $128^{2}$ spatial mesh and $S_{64}$ angular discretization are shown in \cref{fig:lattice low rank comparison} and \cref{fig:lattice low-order vs high-order slices comparison}. The former shows contour plots of the densities, while the latter shows $x$ and $y$ cross-sections of the densities. We note that negative values for the density are observed in the vicinity of the absorbing regions where the density is small. In the contour plots, which use logarithmic scales, we replace the negative values with $10^{-40}$. While this does not impact the stability of the method, such features more noticeable in the split approach. As expected, excessive dissipation is observed in the low-rank scheme based on the low-order discretization.

\begin{figure}
    \centering
    \subfloat[High-order unsplit]{
        \includegraphics[clip, trim={0cm, 0cm, 0cm, 0cm}, scale=0.22]{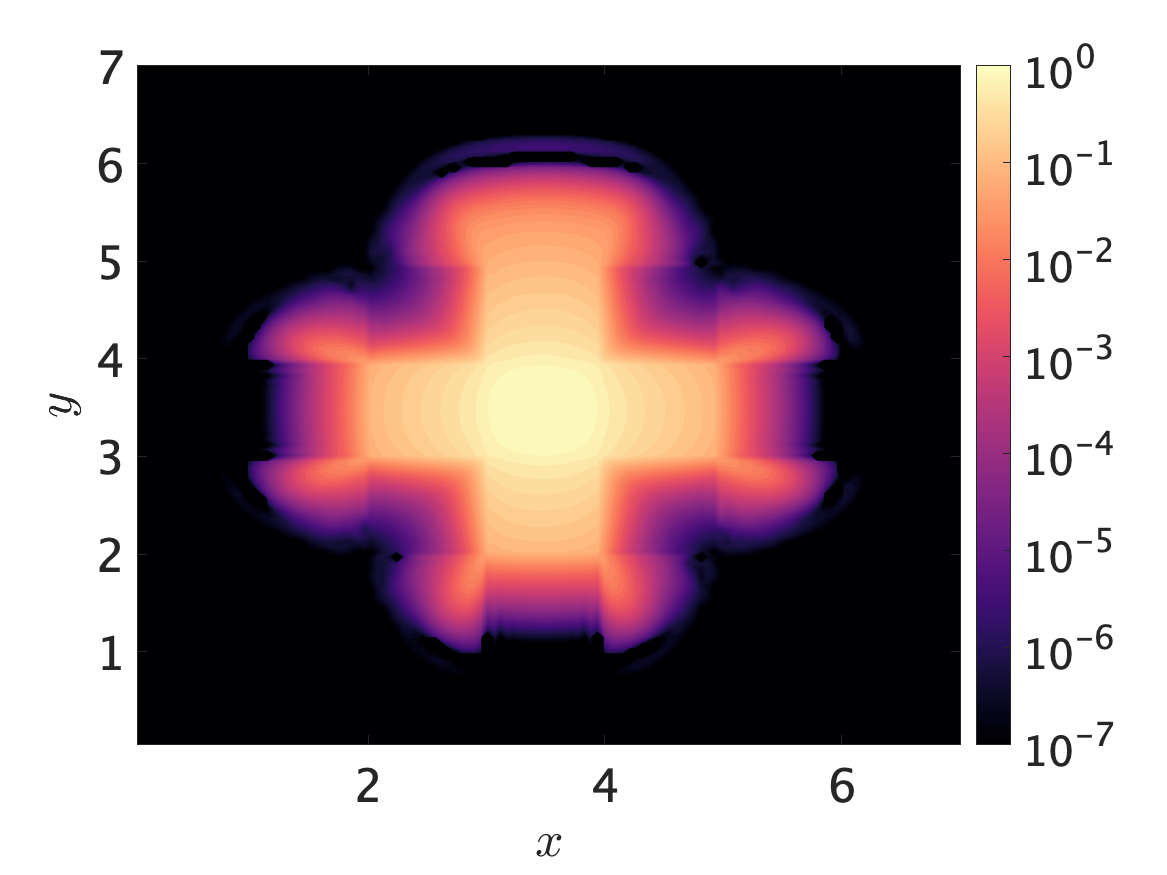}
        \label{fig:lattice rho unsplit}
    }
    \subfloat[High-order split]{
        \includegraphics[clip, trim={0cm, 0cm, 0cm, 0cm}, scale=0.22]{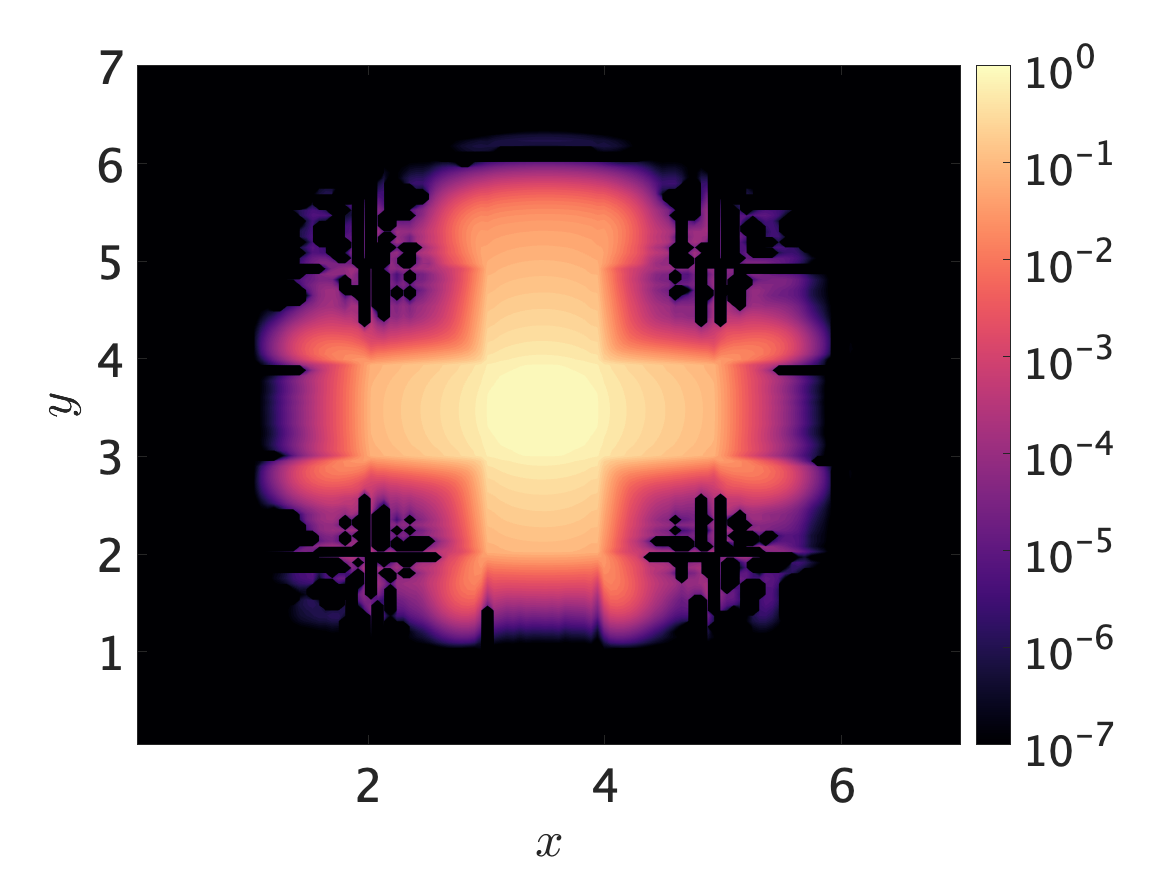}
        \label{fig:lattice rho split}
    }
    \\
    \subfloat[Low-order unsplit]{
        \includegraphics[clip, trim={0cm, 0cm, 0cm, 0cm}, scale=0.22]{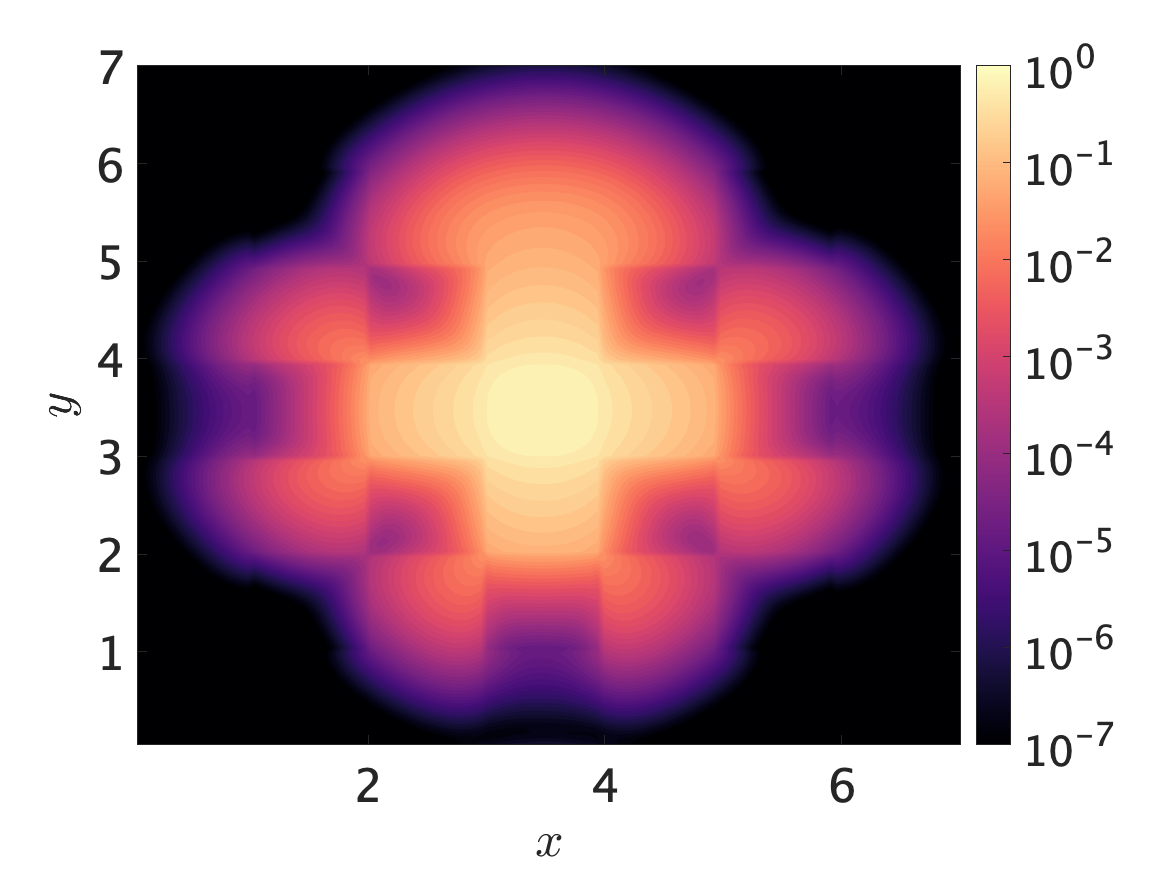}
        \label{fig:lattice rho unsplit low-order}
    }
    \subfloat[Reference]{
        \includegraphics[clip, trim={0cm, 0cm, 0cm, 0cm}, scale=0.22]{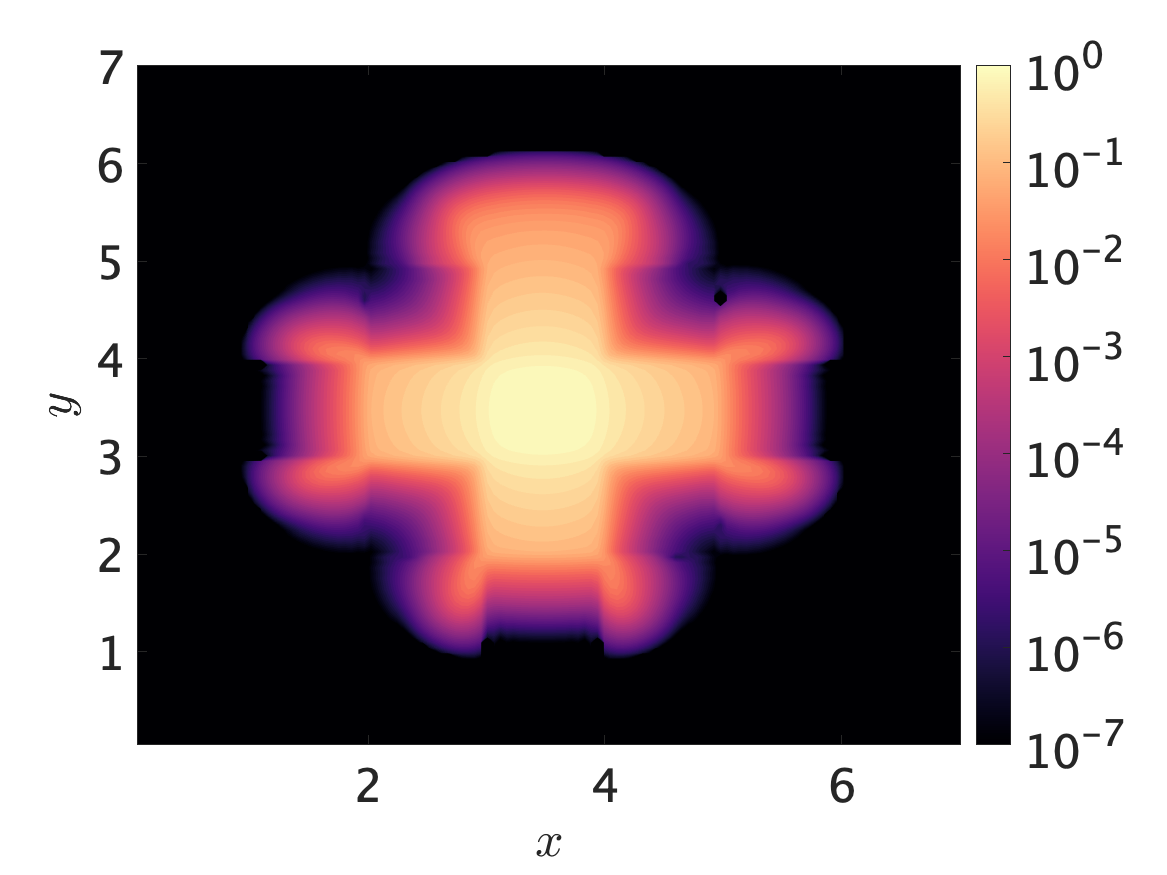}
        \label{fig:lattice rho reference}
    }
    \caption{The scalar density $\rho$ at the final time $T = 2.0$ obtained with different low-rank methods for the lattice problem. We include the high-order reference solution for comparison.}
    \label{fig:lattice low rank comparison}
\end{figure}

\begin{figure}
    \centering
    \includegraphics[clip, trim={0cm, 0cm, 0cm, 0cm}, scale=0.19]{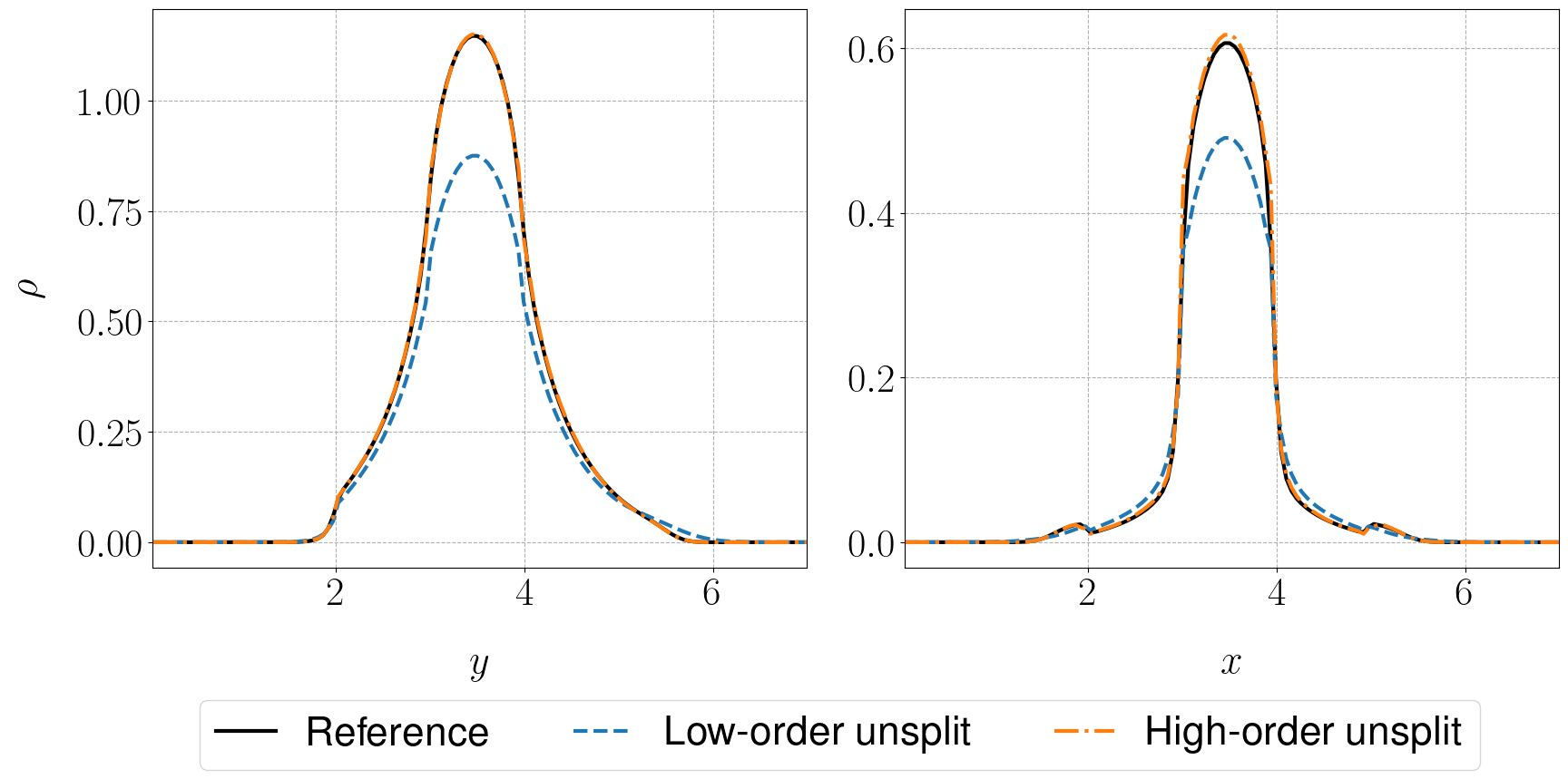}
    \caption{A comparison of unsplit low-rank methods constructed with low-order and high-order discretizations for time and space applied to the lattice problem. We show slices of the scalar density at time $T = 2.0$ along the lines $x = 3.5$ (left column) and $y = 4.046875$ (right column).}
    \label{fig:lattice low-order vs high-order slices comparison}
\end{figure}

In \cref{fig:lattice lr vs fr comparison}, we show plots of the point-wise absolute differences between the densities obtained with the different low-rank methods. In each case, we use the same $128^{2}$ spatial mesh with an identical $S_{64}$ angular discretization and compare the solution against the full-grid implementation. As with the previous example, we find that the high-order unsplit approach yields a better overall approximation of the full-grid solution than the corresponding split approach. Again, we find that the most significant errors are concentrated near the source, which is placed at the center of the domain. In particular, we observe more significant errors at points that coincide with the discontinuities present in the material cross-sections $\sigma_{a}$ and $\sigma_{s}$. The low-rank scheme using the low-order discretization shows considerably larger errors in these regions due to excessive dissipation. Additionally, we observe significant smearing of the density at the corners where the material abruptly changes from optically thin to thick. Such features are better preserved with the high-order discretizations.

\begin{figure}
    \centering
    \subfloat[High-order unsplit]{
        \includegraphics[clip, trim={0.75cm, 0cm, 3.5cm, 0cm}, scale=0.2225]{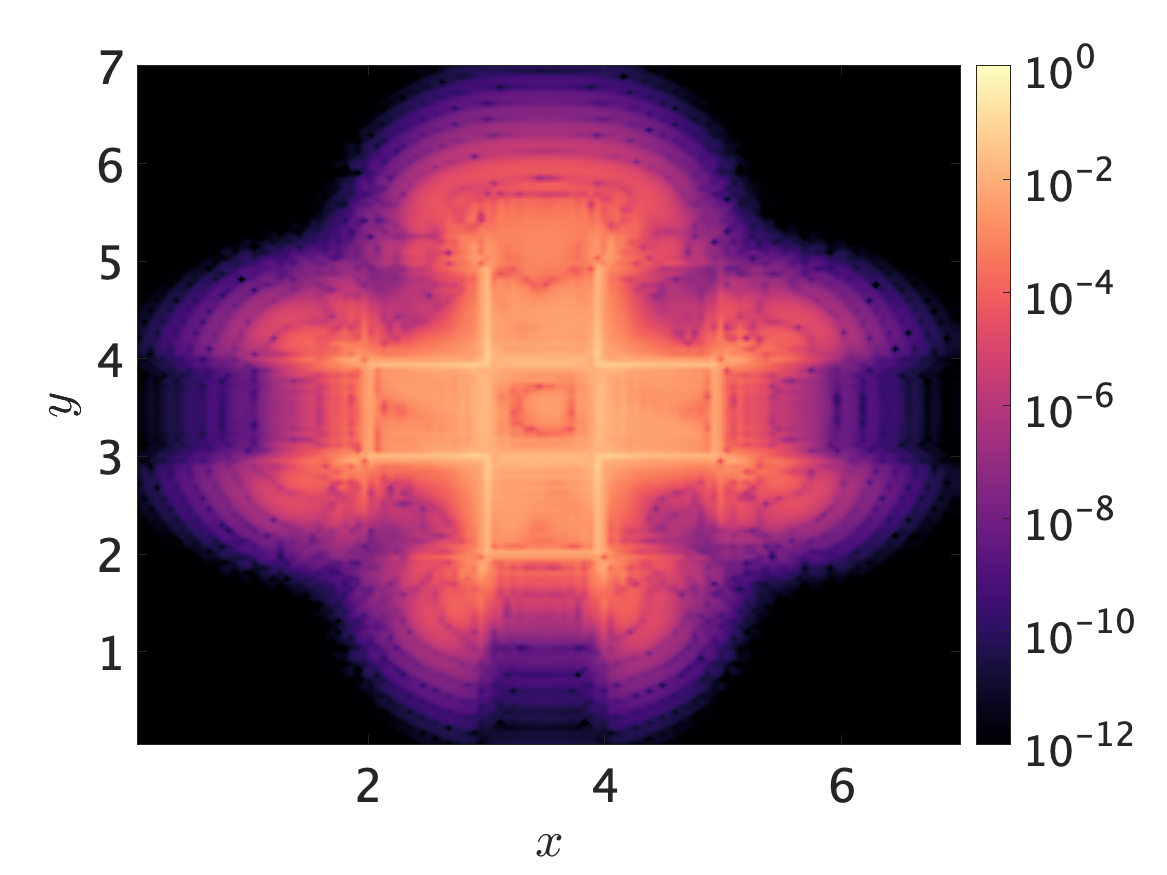}
        \label{fig:lattice rho lr vs fr unsplit with proj}
    }
    \subfloat[High-order split]{
        \includegraphics[clip, trim={0.75cm, 0cm, 3.5cm, 0cm}, scale=0.2225]{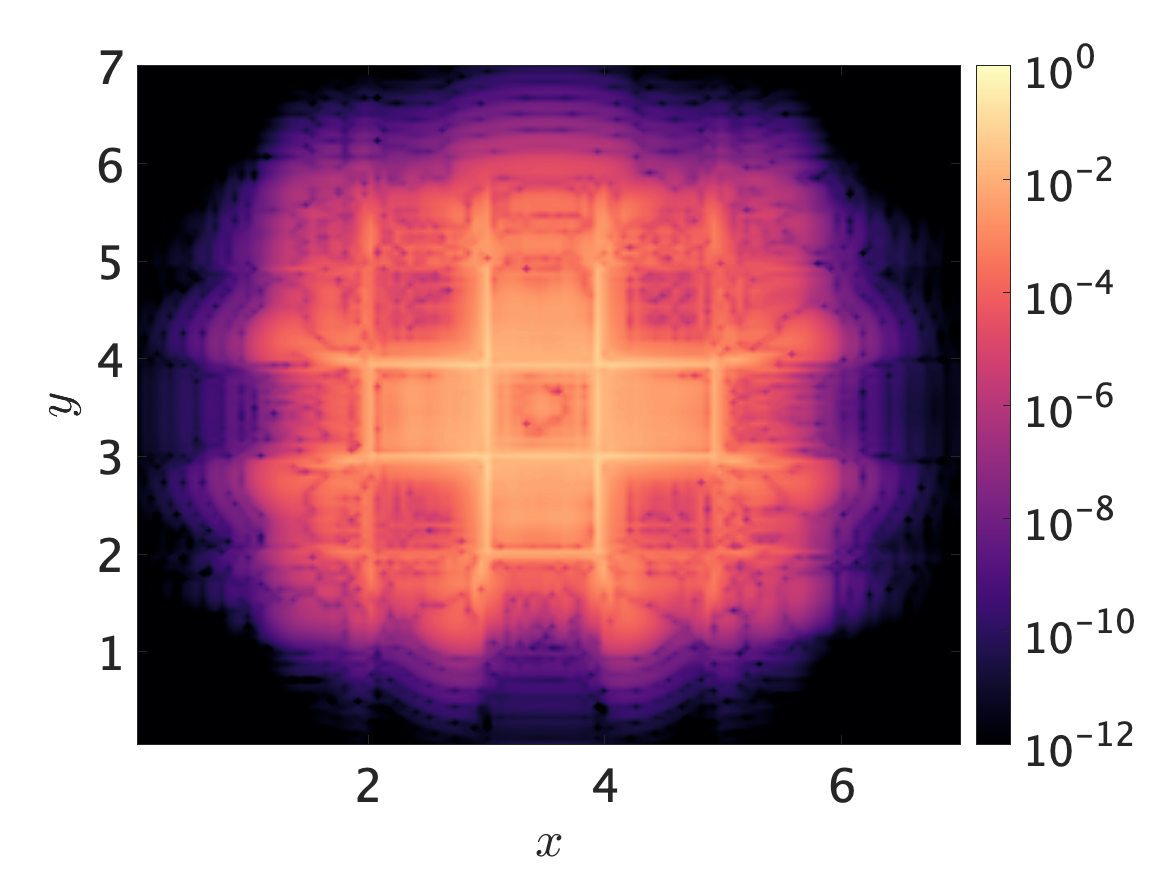}
        \label{fig:lattice rho lr vs fr split with proj}
    }
    \subfloat[Low-order unsplit]{
        \includegraphics[clip, trim={0.75cm, 0cm, 0.0cm, 0cm}, scale=0.2225]{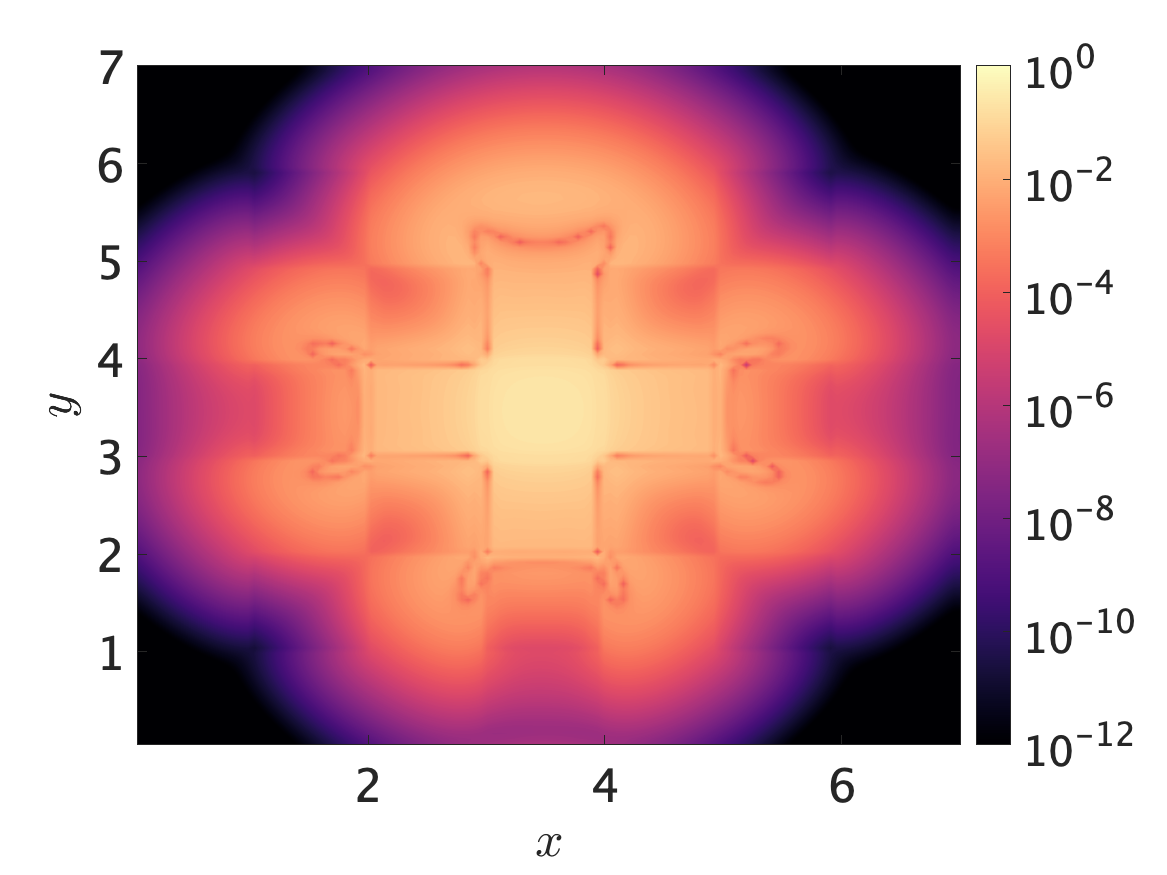}
        \label{fig:lattice rho lr vs fr unsplit with proj low-order}
    }
    \caption{Point-wise absolute differences between the densities obtained with a full-grid method and low-rank methods for the lattice problem.}
    \label{fig:lattice lr vs fr comparison}
\end{figure}

\Cref{fig:lattice problem g svd data} shows the structure of the relative singular values for the matricizations at each of the nodes in the dimension tree for $g$ obtained with unsplit and split approaches using high-order discretizations. These results were obtained using a $128^{2}$ spatial mesh and an $S_{64}$ angular discretization. The hierarchical rank of $g$ recorded at the final step was $[1,357,357,61,13]$ for the unsplit case and $[1,358,358,46,47,58,14]$ for the split case. In each of the plot windows, the vertical axis represents the relative singular values ranging from $10^{-4}$ to $1$, and the ticks correspond to logarithmic spacing in one order of magnitude. The horizontal axis uses linear spacing of the ranks from $1$ to $360$ in increments of $10$. We note that both methods show high rank in the nodes $\{x,y\}$ and $\{\theta,\mu\}$. Further, the split approach shows a reasonable decay along the dimensions associated with individual spatial dimensions (e.g., dimensions 1 and 2 in \cref{fig:lattice svd g split}). Additionally, we expect the solution $g$ to exhibit a stronger angular dependence given that this problem is more kinetic than the previous example. Further, the presence of jump discontinuities in the material data places a significant demand on the global basis of the tensor method. However, we still find tensorization of the angular domain to be beneficial for this problem, as the singular values associated with the dimensions $\{\theta\}$ and $\{\mu\}$ (dimensions 2 and 3 in \cref{fig:lattice svd g unsplit} and dimensions 3 and 4 in \cref{fig:lattice svd g split}, respectively) show reasonable decay that, again, reflects the symmetry in $\mu$.

\begin{figure}
    \centering
    \subfloat[High-order unsplit]{
        \includegraphics[clip, trim={0cm, 0cm, 0cm, 0cm}, scale=0.3]{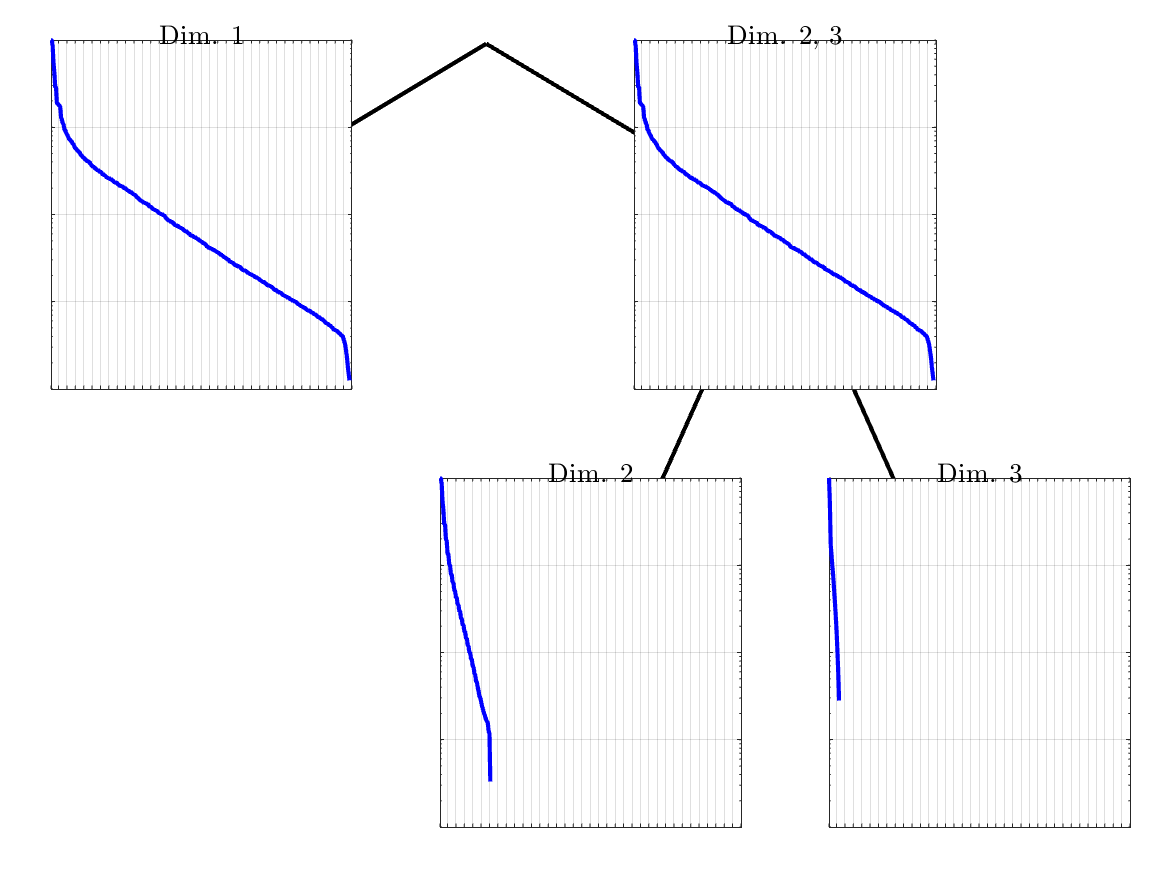}
        \label{fig:lattice svd g unsplit}
    }
    \subfloat[High-order split]{
        \centering
        \includegraphics[clip, trim={0cm, 0cm, 0cm, 0cm}, scale=0.3]{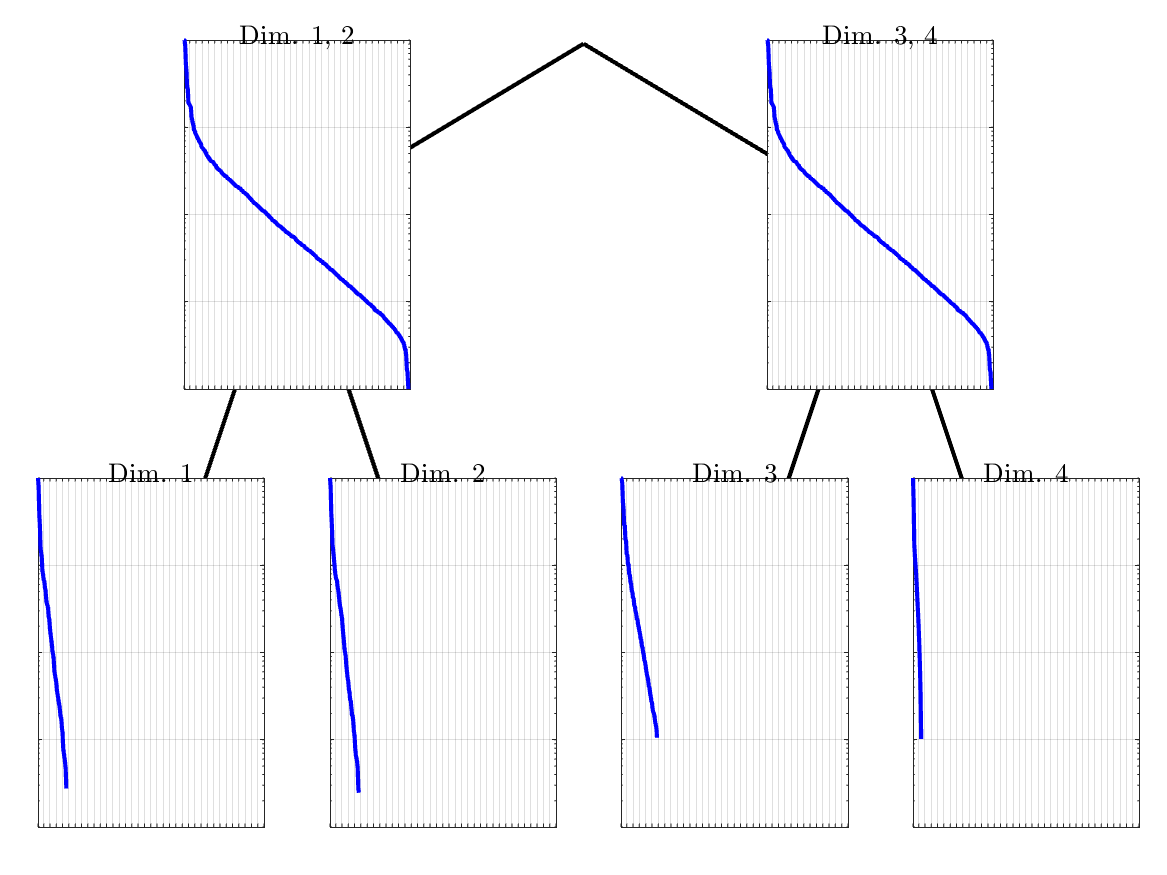}
        \label{fig:lattice svd g split}
    }
    \caption{Plots of the relative singular values versus the rank for the matricizations at each of the nodes in the dimension tree for $g$ at the final time $T = 2$ using unsplit (left) and split (right) dimension trees in the lattice problem.}
    \label{fig:lattice problem g svd data}
\end{figure}

The time evolution of the hierarchical ranks of $g$ with unsplit and split high-order discretizations are presented in \cref{fig:lattice problem g rank vs time}. We, again, considered three different spatial meshes whose sizes are $32^{2}$, $64^{2}$, and $128^{2}$. Unlike the previous example, we find that the maximum rank of $g$ increases with the resolution of the mesh across each of the dimensions. For both the unsplit and split dimension trees, we, again, find that the nodes associated with $\{x,y\}$ and $\{\theta,\mu\}$ grow significantly faster than those associated with the angular dimensions, a feature that persists for each of the simulations. We also find that the rank growth for the angular dimensions follows similar patterns in both the unsplit and split approaches. The hierarchical ranks associated with the nodes $\{x\}$ and $\{y\}$, in the split approach, indicate the presence of low-rank features that can reduce the total DOF. Unlike the previous example, the maximum hierarchical rank for this problem is seen to increase with the mesh resolution, as the grid begins to resolve fine-scale features. This is observed regardless of the dimension tree. The growth of the hierarchical ranks in the high-dimensional function $g$ obtained using the low-order approach is provided in \cref{fig:lattice low-order g ranks vs time}. The low-order scheme shows no growth in the hierarchical ranks of $g$ after time $t \approx 1.5$. Comparing this with \cref{fig:Lattice unsplit ranks g projection}, which corresponds to the high-order scheme, shows several key differences. Most notably, the latter shows growth in the ranks until the final time of the simulation. Additionally, the overall size of the ranks obtained with the low-order approach are considerably smaller than those obtained with the high-order method. For example, at the final time of the simulation using the $128^{2}$ mesh, the hierarchical ranks of g are $[1,132,132,35,8]$ and $[1,357,357,61,13]$ for the low-order and high-order approaches, respectively. This indicates that additional ranks are necessary in order to achieve a more faithful representation of the function $g$. A closer examination of the data collected at the final time for this case reveals that the high-order method uses a factor of $\approx 2.82$ times more total DOF of than the low-order scheme. However, the additional DOF is more than compensated by the increase in accuracy, as indicated in \Cref{fig:lattice lr vs fr comparison}.

\begin{figure}
    \centering
    \subfloat[High-order unsplit]{
        \includegraphics[clip, trim={0cm, 0cm, 0cm, 0cm}, scale=0.19]{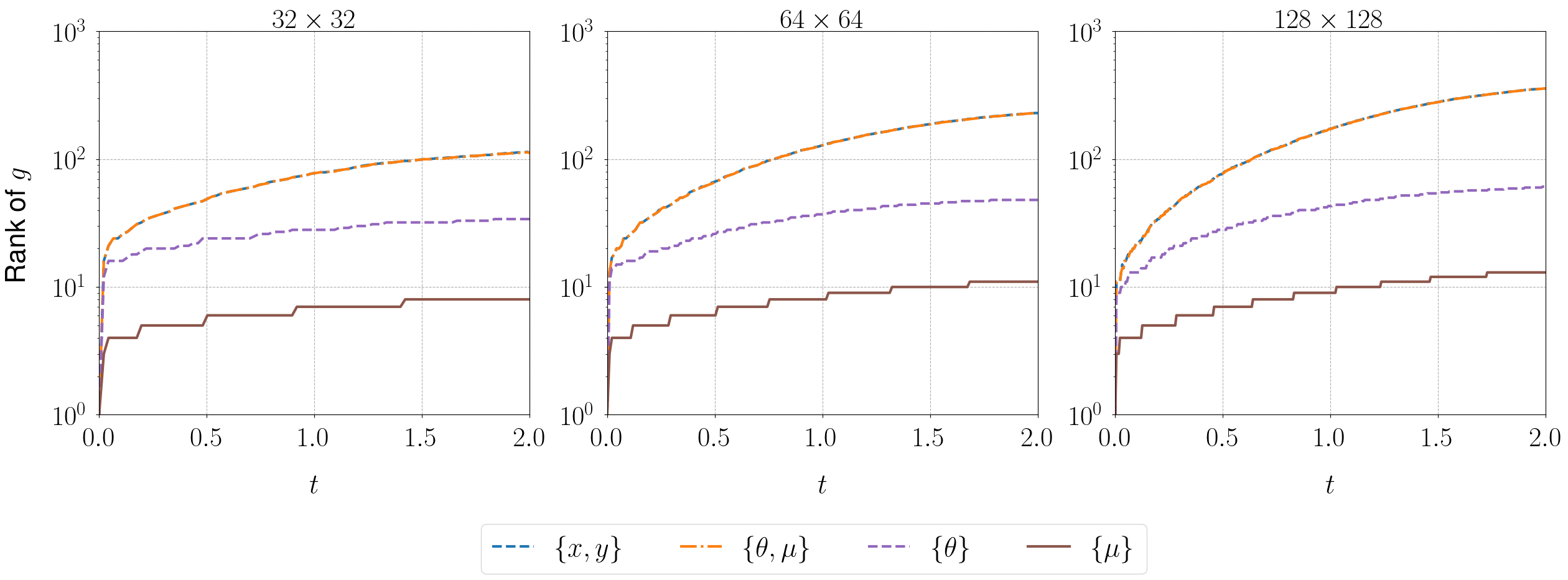}
        \label{fig:Lattice unsplit ranks g projection}
    }
    \\
    \subfloat[High-order split]{
        \includegraphics[clip, trim={0cm, 0cm, 0cm, 0cm}, scale=0.19]{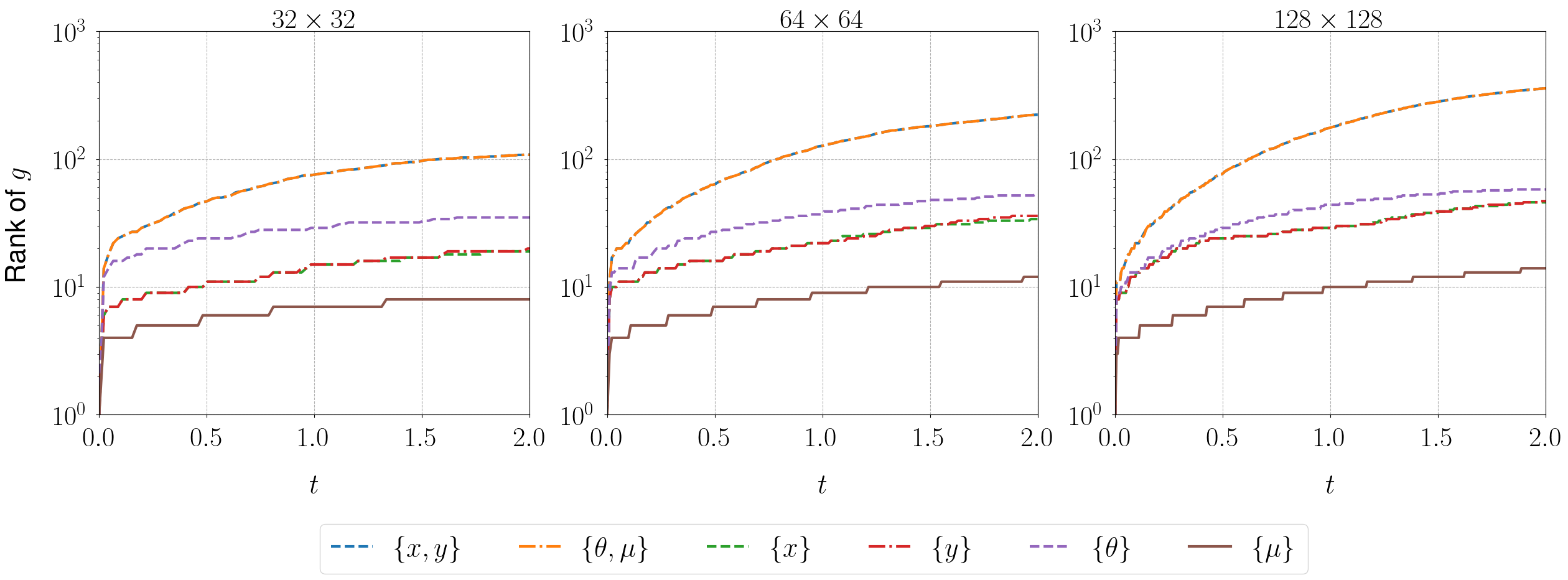}
        \label{fig:Lattice split ranks g projection}
    }
    \caption{Time evolution of hierarchical ranks in the lattice problem using unsplit (top row) and split (bottom row) representations of the tensor $g$.}
    \label{fig:lattice problem g rank vs time}
\end{figure}

\begin{figure}
    \centering
    \includegraphics[clip, trim={0cm, 0cm, 0cm, 0cm}, scale=0.19]{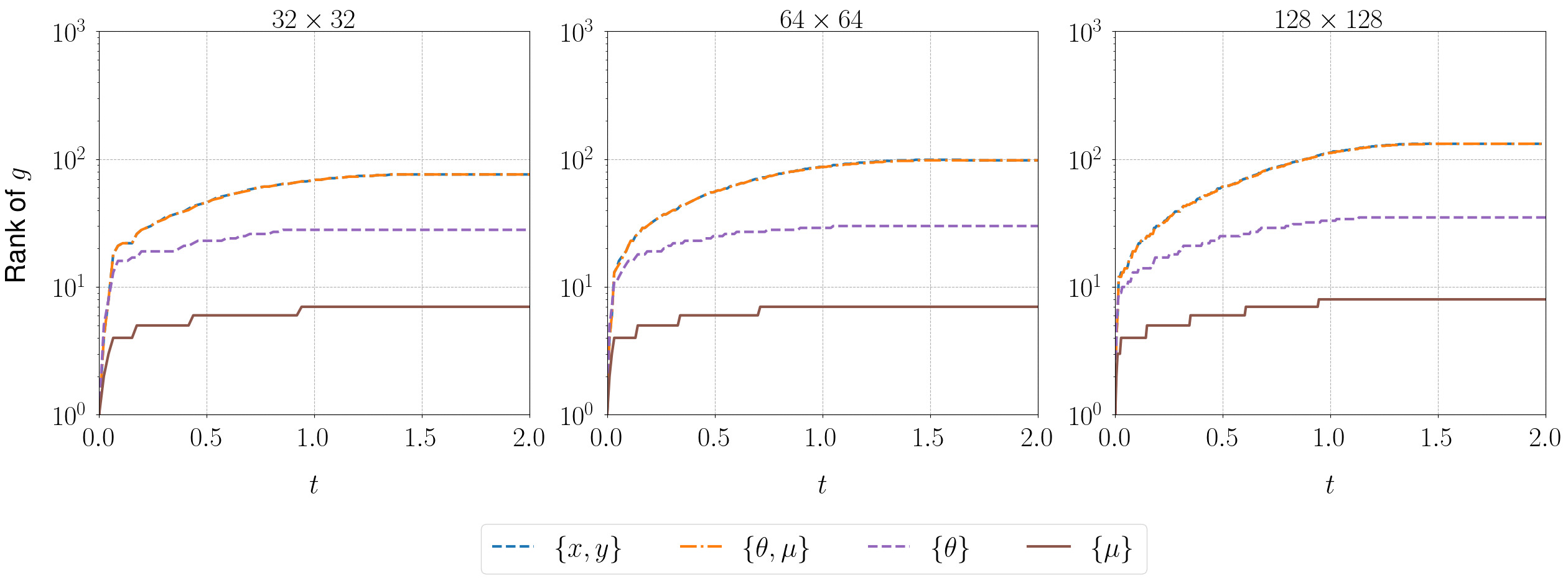}
    \caption{Time evolution of hierarchical ranks in the lattice problem obtained with an unsplit low-rank scheme based on low-order discretizations for time and space.}
    \label{fig:lattice low-order g ranks vs time}
\end{figure}

\begin{figure}
    \centering
    \subfloat[High-order unsplit]{
        \includegraphics[clip, trim={0cm, 0cm, 0cm, 0cm}, scale=0.19]{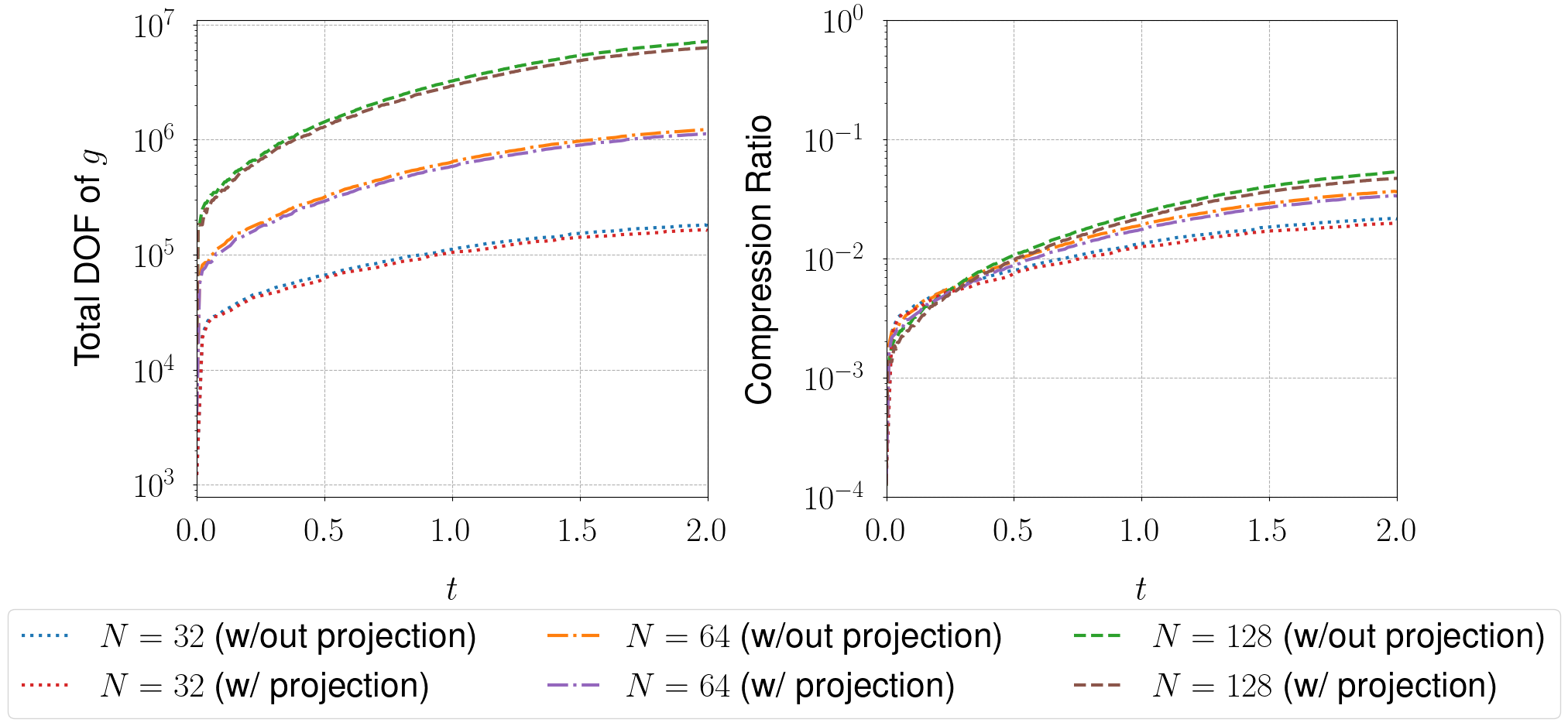}
    }
    \\
    \subfloat[High-order split]{
        \centering
        \includegraphics[clip, trim={0cm, 0cm, 0cm, 0cm}, scale=0.19]{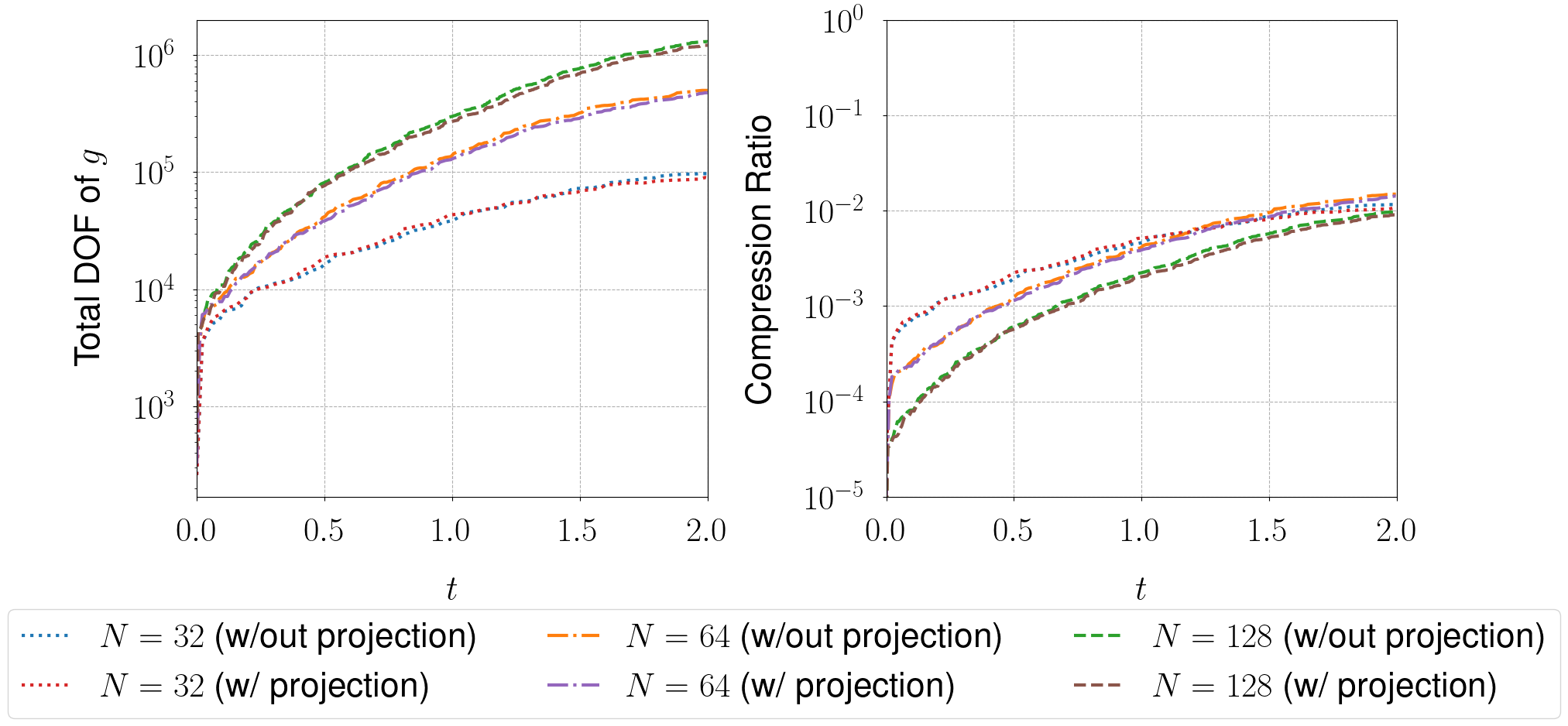}
    }
    \caption{Plots of the total DOF and compression ratio as functions of time for the lattice problem.}
    \label{fig:lattice DOF comparison}
\end{figure}

The growth in the total DOF and compression ratio for the low-rank methods is presented in \Cref{fig:lattice DOF comparison}. Despite the growth in the rank, we find that the low-rank methods still offer reasonable compression of the function $g$. For example, we find that unsplit low-rank method requires $\approx 4.67\%$ of the storage needed for the full-grid representation when a $128^{2}$ spatial mesh is used with an $S_{64}$ angular discretization. The same case, using the split approach, yields a compression ratio of $\approx 0.90\%$. As with the previous example, we find that the use of the projection technique does not have significant memory overhead in terms of the total DOF for the function $g$. Note that the unsplit method produces worse compression compared to the split method as the mesh resolution increases.

\section{Conclusions}
\label{sec:conclusion main}

This work contributed novel low-rank approximation techniques for the solution of the multi-scale linear kinetic transport equation. The approach we take utilizes a macro-micro decomposition of the distribution function to ensure that the proposed scheme recovers the AP property. The proposed low-rank methods support rank adaptivity and address challenges intrinsic to the CoD through the use of the HTT format. The new methods, which eliminate the projector-splitting used by existing low-rank methods, offer enhanced flexibility in coupling low-rank methods with high-order discretizations for both time and space. We presented a variety of numerical results to demonstrate the capabilities of the proposed methods. A key benefit of the proposed methods is the significant reduction in the storage costs relative to a full-grid implementation. We also demonstrated the impact of accuracy of the discretization on the low-rank structures, which indicates that great care should be taken when choosing discretizations. To accommodate more general boundary conditions in transport applications, these methods can be combined with discontinuous Galerkin discretizations. Additionally, we plan to extend the methods presented in this work to the implicit setting, which will further reduce the stiffness near the diffusion limit.

\bibliographystyle{siamplain}
\bibliography{ref}

\begin{thebibliography}{10}

\bibitem{ascher1997implicit}
{\sc U.~M. Ascher, S.~J. Ruuth, and R.~J. Spiteri}, {\em Implicit-explicit
  {R}unge-{K}utta methods for time-dependent partial differential equations},
  Applied Numerical Mathematics, 25 (1997), pp.~151--167.

\bibitem{azmy2010advances}
{\sc Y.~Azmy, E.~Sartori, E.~W. Larsen, and J.~E. Morel}, {\em Advances in
  discrete-ordinates methodology}, Nuclear computational science: A century in
  review,  (2010), pp.~1--84.

\bibitem{BoscarinoAP-IMEX2013}
{\sc S.~Boscarino, L.~Pareschi, and G.~Russo}, {\em Implicit-explicit
  {R}unge--{K}utta schemes for hyperbolic systems and kinetic equations in the
  diffusion limit}, SIAM Journal on Scientific Computing, 35 (2013),
  pp.~A22--A51.

\bibitem{Brunner2002}
{\sc T.~A. Brunner}, {\em Forms of approximate radiation transport}, Tech.
  Report SAND2002-1778, Sandia National Lab.(SNL-NM), Albuquerque, NM (United
  States), 2002.

\bibitem{Brunner2005riemann}
{\sc T.~A. Brunner and J.~P. Holloway}, {\em Two-dimensional time dependent
  {R}iemann solvers for neutron transport}, Journal of Computational Physics,
  210 (2005), pp.~386--399.

\bibitem{ceruti2022rank-adaptive}
{\sc G.~Ceruti, J.~Kusch, and C.~Lubich}, {\em A rank-adaptive robust
  integrator for dynamical low-rank approximation}, BIT Numerical Mathematics,
  62 (2022), pp.~1149--1174.

\bibitem{ceruti2022unconventional}
{\sc G.~Ceruti and C.~Lubich}, {\em An unconventional robust integrator for
  dynamical low-rank approximation}, BIT Numerical Mathematics, 62 (2022),
  pp.~23--44.

\bibitem{coughlin2024robust}
{\sc J.~Coughlin, J.~Hu, and U.~Shumlak}, {\em Robust and conservative
  dynamical low-rank methods for the {V}lasov equation via a novel macro-micro
  decomposition}, Journal of Computational Physics, 509 (2024), p.~113055.

\bibitem{dimarco2013asymptotic}
{\sc G.~Dimarco and L.~Pareschi}, {\em Asymptotic preserving implicit-explicit
  {R}unge--{K}utta methods for nonlinear kinetic equations}, SIAM Journal on
  Numerical Analysis, 51 (2013), pp.~1064--1087.

\bibitem{dimarco2018APMC}
{\sc G.~Dimarco, L.~Pareschi, and G.~Samaey}, {\em Asymptotic-preserving
  {M}onte {C}arlo methods for transport equations in the diffusive limit}, SIAM
  Journal on Scientific Computing, 40 (2018), pp.~A504--A528.

\bibitem{DingDLRdiffusion2021}
{\sc Z.~Ding, L.~Einkemmer, and Q.~Li}, {\em Dynamical low-rank integrator for
  the linear {B}oltzmann equation: error analysis in the diffusion limit}, SIAM
  Journal on Numerical Analysis, 59 (2021), pp.~2254--2285.

\bibitem{Einkemmer2024alfven}
{\sc L.~Einkemmer}, {\em Accelerating the simulation of kinetic shear
  {A}lfv\'{e}n waves with a dynamical low-rank approximation}, Journal of
  Computational Physics, 501 (2024), p.~112757.

\bibitem{Einkemmer_AP_DLR_energy2024}
{\sc L.~Einkemmer, J.~Hu, and J.~Kusch}, {\em Asymptotic-preserving and energy
  stable dynamical low-rank approximation}, SIAM Journal on Numerical Analysis,
  62 (2024), pp.~73--92.

\bibitem{EinkemmerDLR-AP}
{\sc L.~Einkemmer, J.~Hu, and Y.~Wang}, {\em An asymptotic-preserving dynamical
  low-rank method for the multi-scale multi-dimensional linear transport
  equation}, Journal of Computational Physics, 439 (2021), p.~110353.

\bibitem{EinkemmerDLR-AP-BGK}
{\sc L.~Einkemmer, J.~Hu, and L.~Ying}, {\em An efficient dynamical low-rank
  algorithm for the {B}oltzmann-{BGK} equation close to the compressible
  viscous flow regime}, SIAM Journal on Scientific Computing, 43 (2021),
  p.~1057–B1080.

\bibitem{einkemmer2021mass}
{\sc L.~Einkemmer and I.~Joseph}, {\em A mass, momentum, and energy
  conservative dynamical low-rank scheme for the {V}lasov equation}, Journal of
  Computational Physics, 443 (2021), p.~110495.

\bibitem{einkemmer2018VPsplitting}
{\sc L.~Einkemmer and C.~Lubich}, {\em A low-rank projector-splitting
  integrator for the {V}lasov--{P}oisson equation}, SIAM Journal on Scientific
  Computing, 40 (2018), pp.~B1330--B1360.

\bibitem{frank2016convergence}
{\sc M.~Frank, C.~Hauck, and K.~K{\"u}pper}, {\em Convergence of filtered
  spherical harmonic equations for radiation transport}, Communications in
  Mathematical Sciences, 14 (2016), pp.~1443--1465.

\bibitem{GrasedyckHTT2010}
{\sc L.~Grasedyck}, {\em Hierarchical singular value decomposition of tensors},
  {SIAM} Journal on Matrix Analysis and Applications, 31 (2010),
  p.~2029–2054.

\bibitem{GrellaSchwab}
{\sc K.~Grella and C.~Schwab}, {\em Sparse discrete ordinates method in
  radiative transfer}, Computational Methods in Applied Mathematics, 11 (2011),
  pp.~305--326.

\bibitem{GuoChengSparseGridDG2016}
{\sc W.~Guo and Y.~Cheng}, {\em A sparse-grid discontinuous {G}alerkin method
  for high-dimensional transport equations and its application to kinetic
  simulations}, {SIAM} Journal on Scientific Computing, 38 (2016),
  pp.~A3381--A3409.

\bibitem{GuoChengSparseGridDG2017}
{\sc W.~Guo and Y.~Cheng}, {\em An adaptive multiresolution discontinuous
  {G}alerkin method for time-dependent transport equations in multidimensions},
  {SIAM} Journal on Scientific Computing, 39 (2017), pp.~A2962--A2992.

\bibitem{GuoVlasovLoMacDG2023}
{\sc W.~Guo, J.~F. Ema, and J.-M. Qiu}, {\em A local macroscopic conservative
  ({L}o{M}a{C}) low rank tensor method with the discontinuous {G}alerkin method
  for the {V}lasov dynamics}, Communications on Applied Mathematics and
  Computation,  (2023).

\bibitem{GuoVlasovFlowMap2022}
{\sc W.~Guo and J.-M. Qiu}, {\em A low rank tensor representation of linear
  transport and nonlinear {V}lasov solutions and their associated flow maps},
  Journal of Computational Physics, 458 (2022), p.~111089.

\bibitem{GuoVlasovDLRVlasovDynamics}
{\sc W.~Guo and J.-M. Qiu}, {\em A conservative low rank tensor method for the
  {V}lasov dynamics}, SIAM Journal on Scientific Computing, 46 (2024),
  pp.~A232--A263.

\bibitem{HackbuschKuhnHT2009}
{\sc W.~Hackbusch and S.~K\"{u}hn}, {\em A new scheme for the tensor
  representation}, Journal of Fourier Analysis and Applications, 15 (2009),
  pp.~706--722.

\bibitem{HauckPCdynamic23}
{\sc C.~D. Hauck and S.~Schnake}, {\em A predictor-corrector strategy for
  adaptivity in dynamical low-rank approximations}, SIAM Journal on Matrix
  Analysis and Applications, 44 (2023), pp.~971--1005.

\bibitem{hu2022adaptiveBoltzmann}
{\sc J.~Hu and Y.~Wang}, {\em An adaptive dynamical low rank method for the
  nonlinear {B}oltzmann equation}, Journal of Scientific Computing, 92 (2022),
  p.~75.

\bibitem{JangAP-IMEX-DG-2015}
{\sc J.~Jang, F.~Li, J.-M. Qiu, and T.~Xiong}, {\em High order asymptotic
  preserving {DG-IMEX} schemes for discrete-velocity kinetic equations in a
  diffusive scaling}, Journal of Computational Physics, 281 (2015),
  pp.~199--224.

\bibitem{kormann2015TTvlasov}
{\sc K.~Kormann}, {\em A semi-{L}agrangian {V}lasov solver in tensor train
  format}, SIAM Journal on Scientific Computing, 37 (2015), pp.~B613--B632.

\bibitem{kressner2012htucker}
{\sc D.~Kressner and C.~Tobler}, {\em {htucker}—a {MATLAB} toolbox for
  tensors in hierarchical {T}ucker format}, Mathicse, EPF Lausanne,  (2012).

\bibitem{krotz2023hybrid}
{\sc J.~Krotz, C.~D. Hauck, and R.~G. McClarren}, {\em A hybrid {M}onte
  {C}arlo, discontinuous {G}alerkin method for linear kinetic transport
  equations}, 2023, \url{https://arxiv.org/abs/2312.04217}.

\bibitem{DeLathauwerHOSVD2000}
{\sc L.~D. Lathauwer, B.~D. Moor, and J.~Vandewalle}, {\em A multilinear
  singular value decomposition}, SIAM Journal on Matrix Analysis and
  Applications, 21 (2000).

\bibitem{Lebedev1976}
{\sc V.~Lebedev}, {\em Quadratures on a sphere}, USSR Computational Mathematics
  and Mathematical Physics, 16 (1976), pp.~10--24.

\bibitem{Lemou-Mieussens2008}
{\sc M.~Lemou and L.~Mieussens}, {\em A new asymptotic preserving scheme based
  on micro-macro formulation for linear kinetic equations in the diffusion
  limit}, SIAM Journal on Scientific Computing, 31 (2008), pp.~334--368.

\bibitem{lubich2014projector}
{\sc C.~Lubich and I.~V. Oseledets}, {\em A projector-splitting integrator for
  dynamical low-rank approximation}, BIT Numerical Mathematics, 54 (2014),
  pp.~171--188.

\bibitem{mcclarren2010robust}
{\sc R.~G. McClarren and C.~D. Hauck}, {\em Robust and accurate filtered
  spherical harmonics expansions for radiative transfer}, Journal of
  Computational Physics, 229 (2010), pp.~5597--5614.

\bibitem{McClarrenTRTEspherical2008}
{\sc R.~G. McClarren, J.~P. Holloway, and T.~A. Brunner}, {\em On solutions to
  the $p_{n}$ equations for thermal radiative transfer}, Journal of
  Computational Physics, 227 (2008), pp.~2864--2885.

\bibitem{MillerLewis93}
{\sc W.~F. Miller and E.~E. Lewis}, {\em Computational methods of neutron
  transport}, Wiley, 1993.

\bibitem{TToseledets2011}
{\sc I.~Oseledets}, {\em Tensor-train decomposition}, {SIAM} Journal on
  Scientific Computing, 33 (2011), p.~295–2317.

\bibitem{peng2024RBM}
{\sc Z.~Peng, Y.~Chen, Y.~Cheng, and F.~Li}, {\em A micro-macro decomposed
  reduced basis method for the time-dependent radiative transfer equation},
  Multiscale Modeling \& Simulation, 22 (2024), pp.~639--666.

\bibitem{PengDLR2021Holo}
{\sc Z.~Peng and R.~G. McClarren}, {\em A high-order/low-order ({HOLO})
  algorithm for preserving conservation in time-dependent low-rank transport
  calculations}, Journal of Computational Physics, 447 (2021), p.~110672.

\bibitem{PengDLR2023discrete-ordinates}
{\sc Z.~Peng and R.~G. McClarren}, {\em A sweep-based low-rank method for the
  discrete ordinate transport equation}, Journal of Computational Physics, 473
  (2023), p.~111748.

\bibitem{pengDLR2020spherical-harmonics}
{\sc Z.~Peng, R.~G. McClarren, and M.~Frank}, {\em A low-rank method for
  two-dimensional time-dependent radiation transport calculations}, Journal of
  Computational Physics, 421 (2020), p.~109735.

\bibitem{shu2009high}
{\sc C.-W. Shu}, {\em High order weighted essentially nonoscillatory schemes
  for convection dominated problems}, SIAM review, 51 (2009), pp.~82--126.

\bibitem{truong2024tensor}
{\sc D.~P. Truong, M.~I. Ortega, I.~Boureima, G.~Manzini, K.~{\O}. Rasmussen,
  and B.~S. Alexandrov}, {\em Tensor networks for solving the time-independent
  boltzmann neutron transport equation}, Journal of Computational Physics, 507
  (2024), p.~112943.

\bibitem{vanLeer2006upwind}
{\sc B.~van Leer}, {\em Upwind and high-resolution methods for compressible
  flow: From donor cell to residual-distribution schemes}, in 16th {AIAA}
  Computational Fluid Dynamics Conference, 2006, p.~3559.

\bibitem{vassiliev2017mc}
{\sc O.~N. Vassiliev}, {\em {M}onte {C}arlo methods for radiation transport},
  Fundamentals and Advanced Topics,  (2017).

\bibitem{zhang2023asymptotic}
{\sc G.~Zhang, H.~Zhu, and T.~Xiong}, {\em Asymptotic preserving and uniformly
  unconditionally stable finite difference schemes for kinetic transport
  equations}, SIAM Journal on Scientific Computing, 45 (2023), pp.~B697--B730.

\end{thebibliography}

\end{document}